\documentclass[letter,10pt]{amsart}
\usepackage{amssymb,amsmath,amsfonts}
\usepackage{mathrsfs}
\usepackage[left=1 in, right=1 in,top=1 in, bottom=1 in]{geometry}
\usepackage{mathenv}
\usepackage{subfigure}
\usepackage{float}
\usepackage{graphicx}

\newtheorem{theorem}{Theorem}[section]
\newtheorem{lemma}[theorem]{Lemma}

\newtheorem{remark}[theorem]{Remark}

\makeatletter
\renewcommand \theequation {%
\ifnum \c@section>\z@ \@arabic\c@section.%
\fi\@arabic\c@equation} \@addtoreset{equation}{section}
\makeatother

\providecommand{\ud}[1]{\mathrm{d}{#1}}
\providecommand{\abs}[1]{\left\vert#1\right\vert}
\providecommand{\nm}[1]{\left\Vert#1\right\Vert}

\providecommand{\br}[1]{\langle #1 \rangle}
\providecommand{\tm}[2]{\left\Vert#1\right\Vert_{L^2(#2)}}
\providecommand{\im}[2]{\left\Vert#1\right\Vert_{L^{\infty}(#2)}}

\providecommand{\lnnm}[1]{{\left\Vert#1\right\Vert}_{L^{\infty}L^{\infty}}}
\providecommand{\lnm}[1]{\left\Vert#1\right\Vert_{L^{\infty}}}
\providecommand{\tnm}[1]{\left\Vert#1\right\Vert_{L^{2}}}
\providecommand{\tnnm}[1]{{\left\Vert#1\right\Vert}_{L^{2}L^2}}
\providecommand{\ltnm}[1]{{\left\Vert#1\right\Vert}_{L^{\infty}L^{2}}}

\def\p{\partial}

\def\half{\frac{1}{2}}
\def\rt{\rightarrow}
\def\r{\mathbb{R}}
\def\no{\nonumber}
\def\ue{\mathrm{e}}

\def\u{U^{\e}}
\def\ub{\mathscr{U}^{\e}}
\def\bu{\bar U^{\e}}
\def\bub{\bar{\mathscr{U}}^{\e}}
\def\uc{U}
\def\ubc{\mathscr{U}}
\def\buc{\bar U}
\def\bubc{\bar{\mathscr{U}}}

\def\half{\frac{1}{2}}
\def\e{\epsilon}
\def\s{\mathcal{S}}
\def\vx{\vec x}
\def\vw{\vec w}

\def\nx{\nabla_{x}}
\def\t{\mathcal{T}}
\def\a{\mathcal{A}}

\def\px{\p_{\eta}}
\def\ps{\p_{\theta}}
\def\rt{\rightarrow}
\def\l{\lambda}

\def\ll{\mathcal{L}}

\def\r{\mathbb{R}}
\def\f{f}
\def\k{\kappa}
\def\q{Q}
\def\qb{\mathscr{Q}}
\def\v{\mathscr{V}}

\begin{document}
\title{Asymptotic Analysis of Transport Equation in Annulus}

\author[L. Wu]{Lei Wu}
\address[L. Wu]{
   \newline\indent Department of Mathematical Sciences, Carnegie Mellon University
\newline\indent Pittsburgh, PA 15213, USA}
\email{lwu2@andrew.cmu.edu}

\author[X.F. Yang]{Xiongfeng Yang}
\address[X.F. Yang]{
   \newline\indent Department of Mathematics, MOE-LSC and SHL-MAC, Shanghai Jiao Tong University,
\newline\indent Shanghai, 200240, P.R. China}
\email{xf-yang@sjtu.edu.cn}

\author[Y. Guo]{Yan Guo}
\address[Y. Guo]{\newline\indent
Division of Applied Mathematics, Brown University,
\newline\indent Providence, RI 02912, USA }
\email{Yan\_Guo@brown.edu}

\thanks{
L. Wu is supported by NSF grant
0967140.  X.F. Yang is supported by NNSF of China under the Grant 11171212 and
the SJTU's SMC Projection A. Y. Guo is supported in part by NSFC grant 10828103, NSF grant 1209437, Simon Research Fellowship and BICMR.}

\begin{abstract}
We consider the diffusive limit of a steady neutron transport
equation with one-speed velocity in a two-dimensional annulus.
A classical theorem in \cite{Bensoussan.Lions.Papanicolaou1979}
states that the solution can be approximated in $L^{\infty}$ by
the leading order interior solution plus the corresponding Knudsen layers
in the diffusive limit. In this paper, we construct a counterexample of
this result via a different boundary layer expansion with geometric correction.
\end{abstract}
\maketitle

\pagestyle{myheadings} \thispagestyle{plain} \markboth{L. WU, X. YANG AND
Y. GUO}{ASYMPTOTIC ANALYSIS OF
TRANSPORT EQUATION}

\tableofcontents

\listoffigures

\newpage

%%%%%%%%%%%%%%%%%%%%%%%%%%%%%%%%%%%%%%%%%%%%%%%%%%%%%%%%%%%%%%%%%%%%%%%%
\section{Introduction and Notation}%%%%%%%%%%%%%%%%%%%%%%%%%%%%%%%%%%%%%
%%%%%%%%%%%%%%%%%%%%%%%%%%%%%%%%%%%%%%%%%%%%%%%%%%%%%%%%%%%%%%%%%%%%%%%%

%%%%%%%%%%%%%%%%%%%%%%%%%%%%%%%%%%%%%%%%%%%%%%%%%%%%%%%%%%%%%%%%%%%%%%%%
\subsection{Problem Formulation}
%%%%%%%%%%%%%%%%%%%%%%%%%%%%%%%%%%%%%%%%%%%%%%%%%%%%%%%%%%%%%%%%%%%%%%%%

We consider a homogeneous isotropic steady neutron transport
equation with one-speed velocity
$\Sigma=\{\vw=(w_1,w_2):\ \vw\in\s^1\}$
in a two-dimensional annulus $\Omega=\{\vx=(x_1,x_2):\
0<R_{-}<\abs{\vx}< R_{+}<\infty\}$ as
\begin{eqnarray}
\left\{
\begin{array}{rcl}
\e \vw\cdot\nabla_x u^{\e}+u^{\e}-\bar u^{\e}&=&0\label{transport}\
\ \ \text{in}\ \ \Omega,\\\rule{0ex}{1.0em}
u^{\e}(\vx_0,\vw)&=&g_{\pm}(\vx_0,\vw)\ \ \text{for}\ \ \vw\cdot\vec
n<0\ \ \text{and}\ \ \abs{\vx_0}=R_{\pm},
\end{array}
\right.
\end{eqnarray}
where
\begin{eqnarray}\label{average 1}
\bar u^{\e}(\vx)=\frac{1}{2\pi}\int_{\s^1}u^{\e}(\vx,\vw)\ud{\vw}.
\end{eqnarray}
and $\vec n$ is the outward normal vector on $\p\Omega$, $0<\e<<1$ is the
Knudsen number. The in-flow boundary condition is given on the two circles $R_{\pm}$. In this paper, we will study the diffusive limit of the solution $u^{\e}$ to (\ref{transport}) as $\e\rt0$.

In the physical space, we consider the two-dimensional annulus $\Omega=\{\vx=(x_1,x_2):\
0<R_{-}<\abs{\vx}< R_{+}<\infty\}$. Its boundary $\p\Omega$ includes the inner boundary and outer boundary, that is,
$\p\Omega=\p\Omega_{-}\cup\p\Omega_{+}$ where
\begin{eqnarray}
\p\Omega_{-}&=&\{\vx:\ \abs{\vx}=R_-\},\\
\p\Omega_{+}&=&\{\vx:\ \abs{\vx}=R_+\}.
\end{eqnarray}

Based on the flow direction, we can divide the boundary in phase space
$\Gamma=\{(\vx,\vw):\ \vx\in\p\Omega\}$ into the in-flow boundary
$\Gamma^-$, the out-flow boundary $\Gamma^+$, and the grazing set
$\Gamma^0$ as
\begin{eqnarray}
\Gamma^{-}&=&\{(\vx,\vw):\ \vx\in\p\Omega,\ \vw\cdot\vec n<0\},\\
\Gamma^{+}&=&\{(\vx,\vw):\ \vx\in\p\Omega,\ \vw\cdot\vec n>0\},\\
\Gamma^{0}&=&\{(\vx,\vw):\ \vx\in\p\Omega,\ \vw\cdot\vec n=0\}.
\end{eqnarray}
So $\Gamma=\Gamma^+\cup\Gamma^-\cup\Gamma^0$.  For the in-flow boundary condition, the boundary value is only given on $\Gamma^{-}$.

A classical result in \cite{Bensoussan.Lions.Papanicolaou1979}
states that the solution $u^{\e}$ of (\ref{transport}) satisfies
\begin{eqnarray}\label{wrong}
\lnm{u^\e-\uc_0-\ubc_{+,0}-\ubc_{-,0}}=O(\e)
\end{eqnarray}
where $\ubc_{\pm,0}$ is the Knudsen layer to the Milne problem
(\ref{classical temp 1}) while $\uc_0$ is the corresponding interior
solution to the Laplace equation (\ref{classical temp 2}). The goal
of this paper is to construct a counterexample to such a result in
an annulus.

%%%%%%%%%%%%%%%%%%%%%%%%%%%%%%%%%%%%%%%%%%%%%%%%%%%%%%%%%%%%%%%%%%%%%%%%
\subsection{Background and Idea}
%%%%%%%%%%%%%%%%%%%%%%%%%%%%%%%%%%%%%%%%%%%%%%%%%%%%%%%%%%%%%%%%%%%%%%%%

The study of neutron transport equation can date back to 1960s. Since then, this type of problems have been extensively studied in many different settings: steady or unsteady, linear or nonlinear, strong solution or weak solution, etc, (see \cite{Larsen1974}, \cite{Larsen1974=}, \cite{Larsen1975}, \cite{Larsen1977}, \cite{Larsen.D'Arruda1976}, \cite{Larsen.Habetler1973}, \cite{Larsen.Keller1974}, \cite{Larsen.Zweifel1974}, \cite{Larsen.Zweifel1976}). Among all these variations, one of the simplest but most important models - steady neutron transport equation with one-speed velocity in bounded domains, where the boundary layer effect shows up, has long been believed to be satisfactorily solved since Bensoussan, Lions and Papanicolaou published their remarkable paper \cite{Bensoussan.Lions.Papanicolaou1979} in 1979.

The basic idea in \cite{Bensoussan.Lions.Papanicolaou1979} is to consider the boundary layer $f_0(\eta,\phi)$ satisfies that in the domain $(\eta,\phi)\in[0,\infty)\times[-\pi,\pi)$,
\begin{eqnarray}\label{graph 0}
\left\{
\begin{array}{rcl}\displaystyle
\sin\phi\frac{\p
f_{0}}{\p\eta}+f_{0}-\bar f_{0}&=&S_{0}(\eta,\phi),\\
f_{0}(0,\phi)&=&h_{0}(\phi)\ \ \text{for}\ \ \sin\phi>0,\\
\lim_{\eta\rt\infty}f_{0}(\eta,\phi)&=&f_{0,\infty},
\end{array}
\right.
\end{eqnarray}
where $\eta$ denotes the normal variable and $\phi$ the velocity variable. This is the well-known Milne problem and $f_0$ can be shown to be well-posedness and decays exponentially fast to $f_{0,\infty}$ in $L^{\infty}$.

However, in \cite{AA003} the authors pointed out that the construction of boundary layer in \cite{Bensoussan.Lions.Papanicolaou1979} based on Milne problem will break down due to singularity near the grazing set. This brings our attention back to the starting point and we have to reexamine all of the related results. Also, in \cite{AA003}, a new approach was introduced and shown to be effective when the domain is a two-dimensional plate.

The central idea of constructing boundary layer is to consider so-called $\e$-Milne problem with geometric correction. For annulus, the boundary layer $f_+(\eta,\phi)$ near outer circle satisfies
\begin{eqnarray}\label{graph 1}
\left\{
\begin{array}{rcl}\displaystyle
\sin\phi\frac{\p
f_{+}}{\p\eta}-\frac{\e}{R_+-\e\eta}\cos\phi\frac{\p f_+}{\p\phi}+f_{+}-\bar f_{+}&=&S_{+}(\eta,\phi),\\
f_{+}(0,\phi)&=&h_{+}(\phi)\ \ \text{for}\ \ \sin\phi>0,\\
\lim_{\eta\rt\infty}f_{+}(\eta,\phi)&=&f_{+,\infty},
\end{array}
\right.
\end{eqnarray}
Simply speaking, the boundary layer of outer circle is similar to the boundary in a plate. This problem has been extensively studied in \cite[Section 4]{AA003} and we know the solution $f_+$ is well-posedness and decays exponentially fast to $f_{+,\infty}$.

However, for inner circle, we must consider the boundary $f_-(\eta,\phi)$ satisfying
\begin{eqnarray}\label{graph 2}
\left\{
\begin{array}{rcl}\displaystyle
\sin\phi\frac{\p
f_{-}}{\p\eta}+\frac{\e}{R_-+\e\eta}\cos\phi\frac{\p f_-}{\p\phi}+f_{-}-\bar f_{-}&=&S_{-}(\eta,\phi),\\
f_{-}(0,\phi)&=&h_{-}(\phi)\ \ \text{for}\ \ \sin\phi>0,\\
\lim_{\eta\rt\infty}f_{-}(\eta,\phi)&=&f_{-,\infty},
\end{array}
\right.
\end{eqnarray}
The proof in \cite{AA003} relies on the analysis along the characteristics. However, this changed sign of the second term in (\ref{graph 2}) will greatly affect the shape of characteristics, which is shown by Figures \ref{fig 0}, \ref{fig 1} and \ref{fig 2}.

\begin{figure}[H]
\begin{minipage}[t]{0.5\linewidth}
\centering
\includegraphics[width=2.0in]{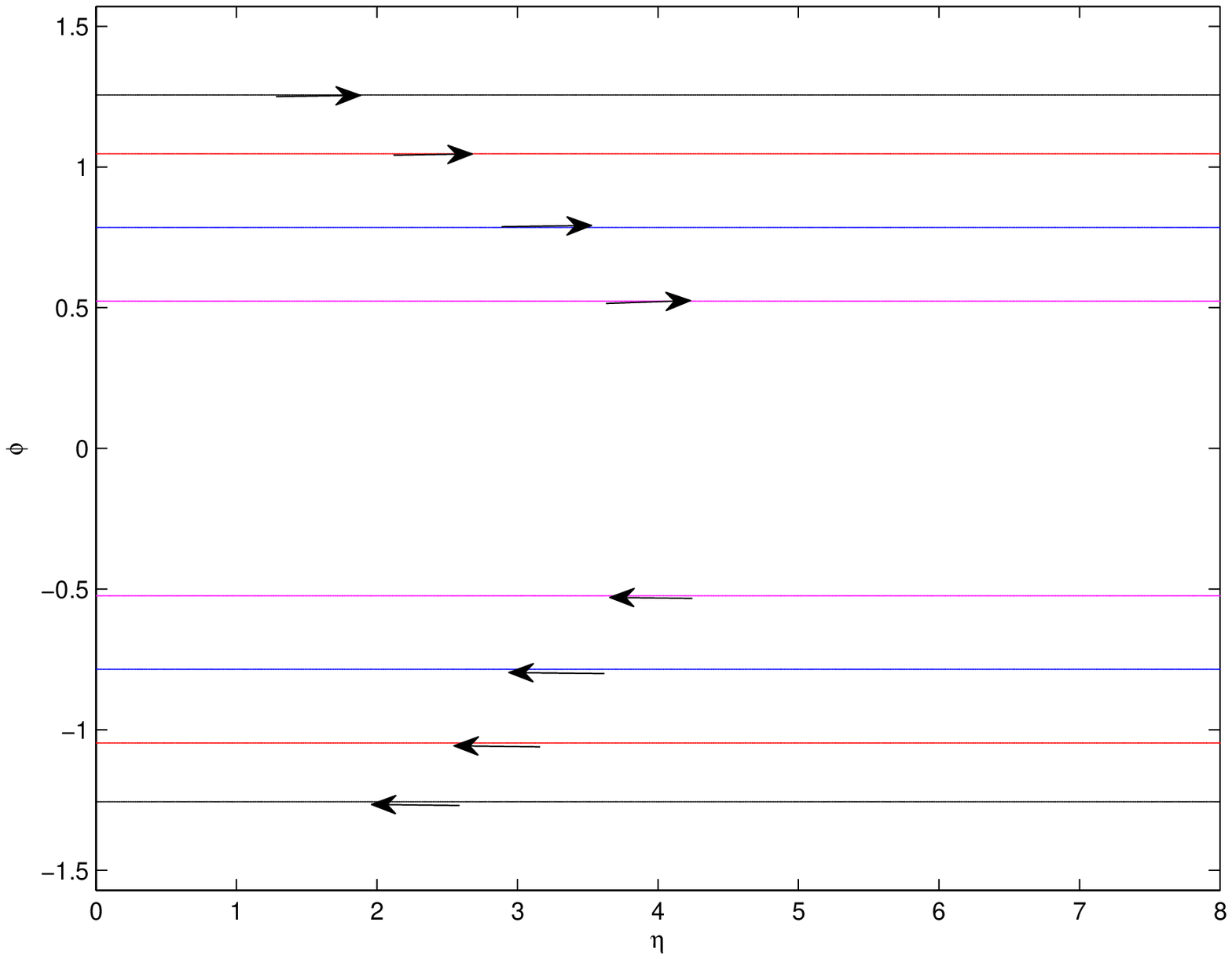}
\caption{Characteristics of Flat Milne Problem for $f_0$} \label{fig 0}
\end{minipage}
\begin{minipage}[t]{0.5\linewidth}
\centering
\includegraphics[width=2.0in]{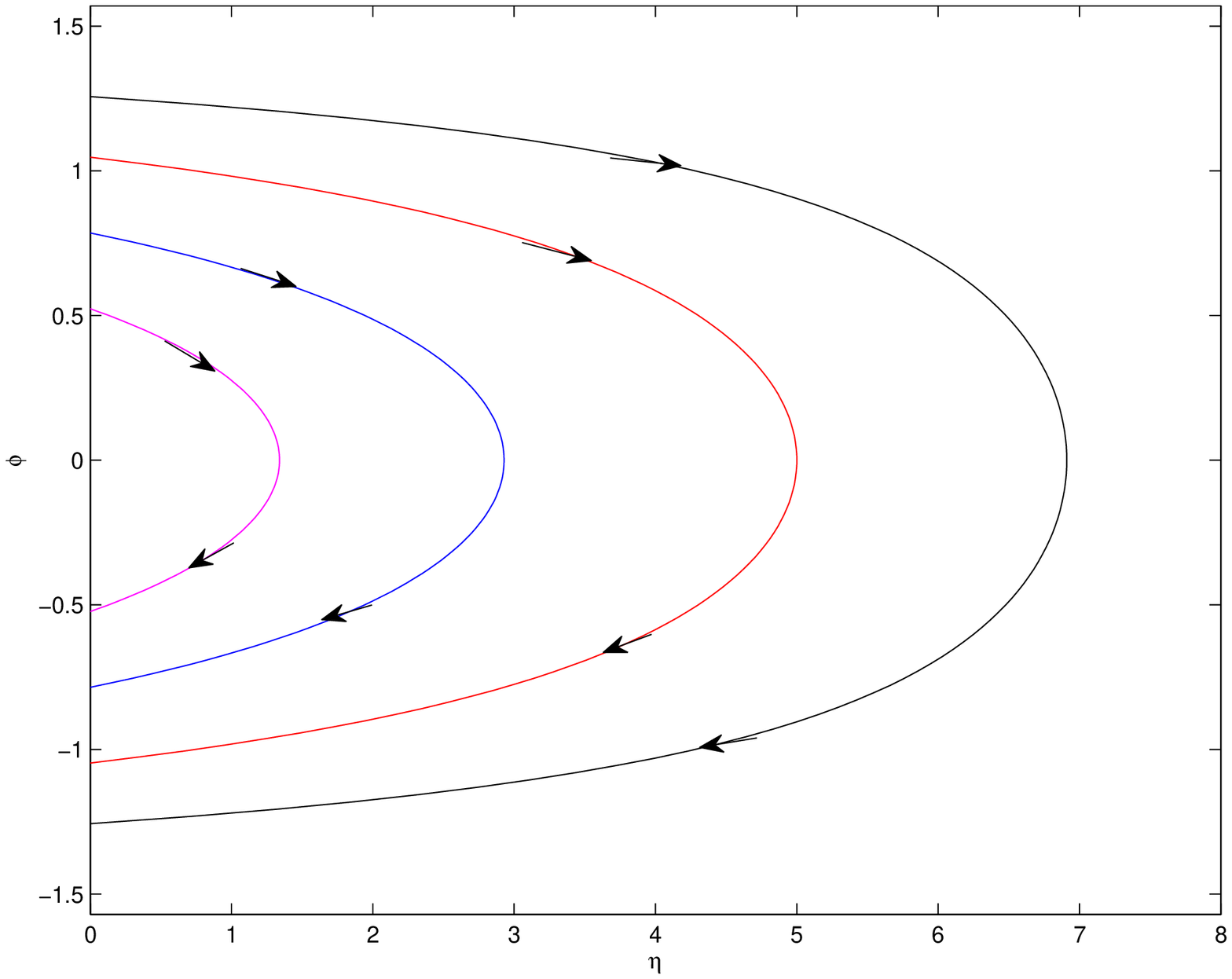}
\caption{Characteristics of $\e$-Milne Problem in Convex Domain for $f_+$} \label{fig 1}
\end{minipage}%
\begin{minipage}[t]{0.5\linewidth}
\centering
\includegraphics[width=2.0in]{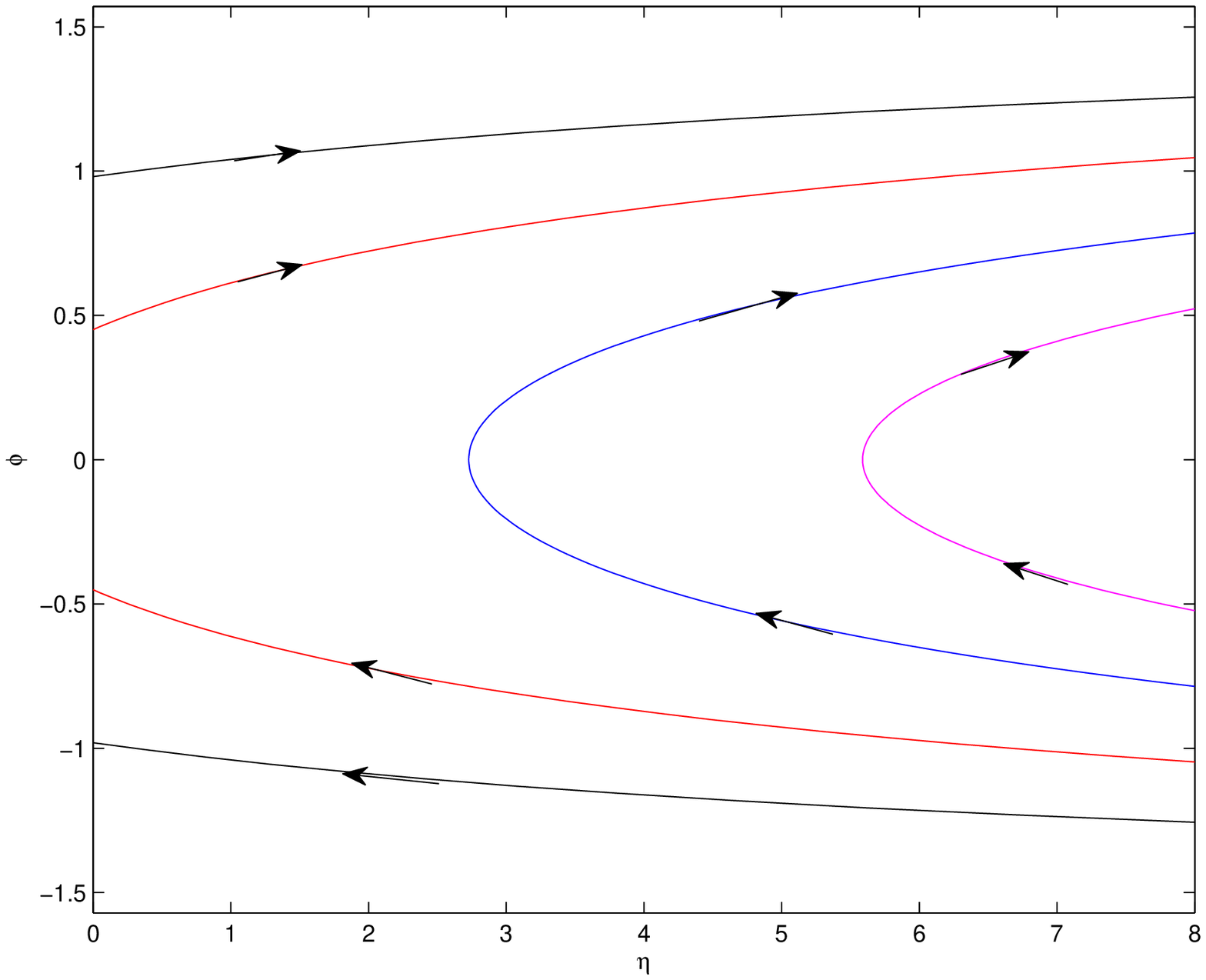}
\caption{Characteristics of $\e$-Milne Problem in Non-Convex Domain for $f_-$} \label{fig 2}
\end{minipage}
\end{figure}

%\begin{figure}[H]
%\centering
%\includegraphics[width=5.0in]{Figure0}
%\caption{Characteristics of Flat Milne Problem for $f_0$} \label{fig 0}
%\end{figure}
%
%\begin{figure}[H]
%\centering
%\includegraphics[width=5.0in]{Figure1}
%\caption{Characteristics of $\e$-Milne Problem in Convex Domain for $f_+$} \label{fig 1}
%\end{figure}
%
%\begin{figure}[H]
%\centering
%\includegraphics[width=5.0in]{Figure2}
%\caption{Characteristics of $\e$-Milne Problem in Non-Convex Domain for $f_-$} \label{fig 2}
%\end{figure}

This adds new difficulties to our estimate and we have to resort to different formulation to bound $f_-$.

%%%%%%%%%%%%%%%%%%%%%%%%%%%%%%%%%%%%%%%%%%%%%%%%%%%%%%%%%%%%%%%%%%%%%%%%
\subsection{Main Results}
%%%%%%%%%%%%%%%%%%%%%%%%%%%%%%%%%%%%%%%%%%%%%%%%%%%%%%%%%%%%%%%%%%%%%%%%

We will use the transformation in (\ref{substitution 2}) and
(\ref{substitution 4}). That is, we define $\vx = (r\cos\theta,r\sin\theta)$ with
$R_-\leq r\leq R_+$ and $-\pi\leq \theta \leq \pi$ while $\vw=(-\sin \xi,\cos\xi)$ with $-\pi\leq \theta \leq \pi$.
We also define the $\phi=\theta+\xi$.
The boundary value is given by $\tilde{g}_{\pm}(\theta,\phi)=g_{\pm}(\vx_0,\vw)$. For convenience, we denote the boundary condition as $g_{\pm}(\theta,\phi)$.
Then, the diffusive limit is stated as follows:
\begin{theorem}\label{main 1}
Assume $g_{\pm}(\vx_0,\vw)\in C^3(\Gamma^-)$. Then there exists a unique solution
$u^{\e}(\vx,\vw)\in L^{\infty}(\Omega\times\s^1)$ for the steady
neutron transport equation (\ref{transport}). Furthermore, it satisfies
\begin{eqnarray}\label{main theorem 1}
\lnm{u^{\e}-\u_0-\ub_{+,0}-\ub_{-,0}}=O(\e)
\end{eqnarray}
where the interior solution $\u_0$ and boundary layer $\ub_{\pm,0}$
are defined in (\ref{expansion temp 8}) and (\ref{expansion temp
9}). In particular, if $g_+(\phi,\theta))=0$ and
$g_{-}(\theta, \phi)=\cos\phi$, then there exists a positive constant $C>0$ such that
\begin{eqnarray}\label{main theorem 2}
\lnm{u^{\e}-\uc_0-\ubc_{+,0}-\ubc_{-,0}}\geq C>0
\end{eqnarray}
when $\e$ is sufficiently small, where $\uc_0$ and $\ubc_{\pm,0}$
are defined in (\ref{classical temp 2}) and (\ref{classical temp
1}).
\end{theorem}

%%%%%%%%%%%%%%%%%%%%%%%%%%%%%%%%%%%%%%%%%%%%%%%%%%%%%%%%%%%%%%%%%%%%%%%%
\subsection{Notation and Structure}
%%%%%%%%%%%%%%%%%%%%%%%%%%%%%%%%%%%%%%%%%%%%%%%%%%%%%%%%%%%%%%%%%%%%%%%%

Throughout this paper, $C>0$ denotes a constant only depends on
the parameter $\Omega$, but does not depend on the given data. It is
referred as universal and can is different from one line to another.
When we write $C(z)$, it means a certain positive constant depending
on the quantity $z$.

Our paper is organized as follows: In Section 2, we first present
the asymptotic analysis of both the interior solution and boundary
layer; In Section 3, we give a complete
analysis of the $\e$-Milne problem with geometric correction; In Section 4, we establish the $L^{\infty}$ remainder estimate;
finally, in Section 5, we prove the diffusive limit of the solution, i.e. Theorem
\ref{main 1}.

%%%%%%%%%%%%%%%%%%%%%%%%%%%%%%%%%%%%%%%%%%%%%%%%%%%%%%%%%%%%%%%%%%%%%%%%
\section{Asymptotic Analysis}%%%%%%%%%%%%%%%%%%%%%%%%%%%%%%%%%%%%%%%%%%%
%%%%%%%%%%%%%%%%%%%%%%%%%%%%%%%%%%%%%%%%%%%%%%%%%%%%%%%%%%%%%%%%%%%%%%%%

%%%%%%%%%%%%%%%%%%%%%%%%%%%%%%%%%%%%%%%%%%%%%%%%%%%%%%%%%%%%%%%%%%%%%%%%
\subsection{Interior Expansion}
%%%%%%%%%%%%%%%%%%%%%%%%%%%%%%%%%%%%%%%%%%%%%%%%%%%%%%%%%%%%%%%%%%%%%%%%

We define the interior expansion as follows:
\begin{eqnarray}\label{interior expansion}
\uc(\vx,\vw)\sim\sum_{k=0}^{\infty}\e^k\uc_k(\vx,\vw).
\end{eqnarray}
Plugging (\ref{interior expansion}) into the equation
(\ref{transport}) and comparing the order of $\e$, the functions $\uc_k~(k=0,1,2\cdots,)$ should satisfy
\begin{eqnarray}
\uc_0-\buc_0&=&0,\label{expansion temp 1}\\
\uc_1-\buc_1&=&-\vw\cdot\nx\uc_0,\label{expansion temp 2}\\
\uc_2-\buc_2&=&-\vw\cdot\nx\uc_1,\label{expansion temp 3}\\
\ldots\nonumber\\
\uc_k-\buc_k&=&-\vw\cdot\nx\uc_{k-1}.
\end{eqnarray}
\ \\
The following analysis reveals the equation satisfied by
$\uc_k$:\\
The equalities (\ref{expansion temp 1}) and (\ref{expansion temp 2}) could be rewritten as
\begin{eqnarray}
\uc_1=\buc_1-\vw\cdot\nx\buc_0.\label{expansion temp 4}
\end{eqnarray}
Thus, from (\ref{expansion temp 4}) into (\ref{expansion temp 3}), we
get
\begin{eqnarray}\label{expansion temp 13}
\uc_2-\buc_2=-\vw\cdot\nx(\buc_1-\vw\cdot\nx\buc_0)=-\vw\cdot\nx\buc_1+|\vw|^2\Delta_x\buc_0+2w_1w_2\p_{x_1x_2}\buc_0.
\end{eqnarray}
Integrating (\ref{expansion temp 13}) over $\vw\in\s^1$, we achieve
the final form
\begin{eqnarray}
\Delta_x\buc_0=0,
\end{eqnarray}
which further implies $\uc_0(\vx,\vw)$ satisfies the equation
\begin{eqnarray}\label{interior 1}
\left\{
\begin{array}{rcl}
\uc_0&=&\buc_0\\
\Delta_x\uc_0&=&0
\end{array}
\right.
\end{eqnarray}
Similarly, we can derive $\uc_k(\vx,\vw)$ for $k\geq1$ satisfies
\begin{eqnarray}\label{interior 2}
\left\{
\begin{array}{rcl}
\uc_k&=&\buc_k-\vw\cdot\nx\uc_{k-1}\\
\Delta_x\buc_k&=&0
\end{array}
\right.
\end{eqnarray}

%%%%%%%%%%%%%%%%%%%%%%%%%%%%%%%%%%%%%%%%%%%%%%%%%%%%%%%%%%%%%%%%%%%%%%%%
\subsection{Milne Expansion}
%%%%%%%%%%%%%%%%%%%%%%%%%%%%%%%%%%%%%%%%%%%%%%%%%%%%%%%%%%%%%%%%%%%%%%%%

In general, the value of (\ref{interior expansion}) on the boundary is different from the in-flow boundary condition in (\ref{transport}). In order to match the boundary condition, we need to give the boundary layer expansion. Here we firstly recall the idea of this expansion in
\cite[pp.136]{Bensoussan.Lions.Papanicolaou1979} in the following several substitutions:\\
\ \\
{\it Substitution 1}: We consider the substitution into quasi-polar coordinates
$(x_1,x_2,\vw)\rt (\mu_{\pm},\theta,\vw)$ with
$(\mu_{\pm},\theta, \vw)\in
[0,R_+-R_-]\times[-\pi,\pi)\times\s^1$ defined as
\begin{eqnarray}\label{substitution 1}
\left\{
\begin{array}{rcl}
x_1&=&(R_{\pm}\mp\mu_{\pm})\cos\theta,\\
x_2&=&(R_{\pm}\mp\mu_{\pm})\sin\theta,\\
w_1&=&w_1,\\
w_2&=&w_2.
\end{array}
\right.
\end{eqnarray}
Here $\mu_{\pm}$ denotes the distance to the boundary
$\p\Omega_{\pm}$ and $\theta$ is the space angular variable. In
these new variables,  we also denote the solution as
$u^{e}(\mu_{\pm},\theta,\vw)$. Then, the equation (\ref{transport}) can be rewritten as
\begin{eqnarray}
\left\{ \begin{array}{l} \displaystyle
\mp\e\big(w_1\cos\theta+w_2\sin\theta\big)\frac{\p
u^{\e}}{\p\mu_{\pm}}-\frac{\e\big(w_1\sin\theta-w_2\cos\theta\big)}{R_{\pm}\mp\mu_{\pm}}\frac{\p
u^{\e}}{\p\theta}+u^{\e}-\frac{1}{2\pi}\int_{\s^1}u^{\e}\ud{\vw}=0,\\\rule{0ex}{1.0em}
u^{\e}(0,\theta,w_1,w_2)=g_{\pm}(\theta,w_1,w_2)\ \ \text{for}\ \
\pm (w_1\cos\theta+w_2\sin\theta)<0.
\end{array}
\right.
\end{eqnarray}

\noindent {\it Substitution 2}:
We further define the stretched variable $\eta_{\pm}$ by making the
scaling transform for $u^{\e}(\mu_{\pm},\theta,w_1,w_2)\rt
u^{\e}(\eta_{\pm},\theta,w_1,w_2)$ with
$(\eta_{\pm},\theta,w_1,w_2)\in
[0,(R_{+}-R_{-})/\e]\times[-\pi,\pi)\times\s^1$ as
\begin{eqnarray}\label{substitution 2}
\left\{
\begin{array}{rcl}
\eta_{\pm}&=&\mu_{\pm}/\e,\\
\theta&=&\theta,\\
w_1&=&w_1,\\
w_2&=&w_2,
\end{array}
\right.
\end{eqnarray}
which implies
\begin{eqnarray}
\frac{\p u^{\e}}{\p\mu_{\pm}}=\frac{1}{\e}\frac{\p
u^{\e}}{\p\eta_{\pm}}.
\end{eqnarray}
Then equation (\ref{transport}) is transformed into
\begin{eqnarray}
\left\{ \begin{array}{l}\displaystyle
\mp\big(w_1\cos\theta+w_2\sin\theta\big)\frac{\p
u^{\e}}{\p\eta_{\pm}}-\frac{\e\big(w_1\sin\theta-w_2\cos\theta\big)}{R_{\pm}\mp\e\eta_{\pm}}\frac{\p
u^{\e}}{\p\theta}+u^{\e}-\frac{1}{2\pi}\int_{\s^1}u^{\e}\ud{\vw}=0,\\\rule{0ex}{1.0em}
u^{\e}(0,\theta,w_1,w_2)=g_{\pm}(\theta,w_1,w_2)\ \ \text{for}\ \
\pm (w_1\cos\theta+w_2\sin\theta)<0.
\end{array}
\right.
\end{eqnarray}
\ \\
\noindent {\it Substitution 3}:
Define the velocity substitution for
$u^{\e}(\eta_{\pm},\theta,w_1,w_2)\rt u^{\e}(\eta_{\pm},\theta,\xi)$
with $(\eta_{\pm},\theta,\xi)\in
[0,(R_{+}-R_{-})/\e]\times[-\pi,\pi)\times[-\pi,\pi)$ as
\begin{eqnarray}\label{substitution 3}
\left\{
\begin{array}{rcl}
\eta_{\pm}&=&\eta_{\pm},\\
\theta&=&\theta,\\
w_1&=&-\sin\xi,\\
w_2&=&-\cos\xi.
\end{array}
\right.
\end{eqnarray}
Here $\xi$ denotes the velocity angular variable. We have the
succinct form for (\ref{transport}) as
\begin{eqnarray}\label{classical temp}
\left\{ \begin{array}{l}\displaystyle \pm\sin(\theta+\xi)\frac{\p
u^{\e}}{\p\eta_{\pm}}-\frac{\e\cos(\theta+\xi)}{R_{\pm}\mp\e\eta_{\pm}}\frac{\p
u^{\e}}{\p\theta}+u^{\e}-\frac{1}{2\pi}\int_{-\pi}^{\pi}u^{\e}\ud{\xi}=0,\\\rule{0ex}{1.0em}
u^{\e}(0,\theta,\xi)=g_{\pm}(\theta,\xi)\ \ \text{for}\ \
\pm\sin(\theta+\xi)>0.
\end{array}
\right.
\end{eqnarray}
\ \\
We now define the Milne expansion of boundary layer as follows:
\begin{eqnarray}\label{classical expansion}
\ubc_{\pm}(\eta_{\pm},\theta,\phi)\sim\sum_{k=0}^{\infty}\e^k\ubc_{\pm,k}(\eta_{\pm},\theta,\phi),
\end{eqnarray}
where $\ubc_{\pm,k}$ can be determined by comparing the order of
$\e$ via plugging (\ref{classical expansion}) into the equation
(\ref{classical temp}). Thus, in a neighborhood of the boundary, we
have
\begin{eqnarray}
\pm\sin(\theta+\xi)\frac{\p
\ubc_{\pm,0}}{\p\eta_{\pm}}+\ubc_{\pm,0}-\bubc_{\pm,0}&=&0,\label{cexpansion temp 5}\\
\pm\sin(\theta+\xi)\frac{\p
\ubc_{\pm,1}}{\p\eta_{\pm}}+\ubc_{\pm,1}-\bubc_{\pm,1}&=&\frac{\cos(\theta+\xi)}{R_{\pm}\mp\e\eta_{\pm}}\frac{\p
\ubc_{\pm,0}}{\p\theta},\label{cexpansion temp 6}\\
\ldots\nonumber\\
\pm\sin(\theta+\xi)\frac{\p
\ubc_{\pm,k}}{\p\eta_{\pm}}+\ubc_{\pm,k}-\bubc_{\pm,k}&=&\frac{\cos(\theta+\xi)}{R_{\pm}\mp\e\eta_{\pm}}\frac{\p
\ubc_{\pm,k-1}}{\p\theta},
\end{eqnarray}
where
\begin{eqnarray}
\bar
\ubc_{\pm,k}(\eta_{\pm},\theta)=\frac{1}{2\pi}\int_{-\pi}^{\pi}\ubc_{\pm,k}(\eta_{\pm},\theta,\xi)\ud{\xi}.
\end{eqnarray}
We hope that the solution is formulated from the interior solution and the boundary layer
solution. So it should satisfy the boundary condition of (\ref{transport}). The boundary condition expansion derives to
\begin{eqnarray}
\uc_0+\ubc_{\pm,0}&=&g_{\pm},\\
\uc_1+\ubc_{\pm,1}&=&0,\\
\ldots\nonumber\\
\uc_k+\ubc_{\pm,k}&=&0.
\end{eqnarray}
The construction of $\uc_{k}$ and $\ubc_{\pm,k}$ in
\cite{Bensoussan.Lions.Papanicolaou1979} can be
summarized as follows:\\
\ \\
Step 1: Construction of $\ubc_{\pm,0}$ and $\uc_0$.\\
To deal with the sigularity according to $r$, we define the cut-off function $\psi$ and $\psi_0$ as
\begin{eqnarray}\label{cut-off 1}
\psi(\mu)=\left\{
\begin{array}{ll}
1&0\leq\mu\leq1/2(R_+-R_-),\\
0&3/4(R_+-R_-)\leq\mu\leq\infty.
\end{array}
\right.
\end{eqnarray}
\begin{eqnarray}\label{cut-off 2}
\psi_0(\mu)=\left\{
\begin{array}{ll}
1&0\leq\mu\leq1/4(R_+-R_-),\\
0&3/8(R_+-R_-)\leq\mu\leq\infty.
\end{array}
\right.
\end{eqnarray}
Then the zeroth order boundary layer solution is defined as
\begin{eqnarray}\label{classical temp 1}
\left\{
\begin{array}{rclll}
\ubc_{\pm,0}(\eta_{\pm},\theta,\xi)&=&\psi_0(\e\eta_{\pm})\bigg(\f_{\pm,0}(\eta_{\pm},\theta,\xi)-f_{\pm,0}(\infty,\theta)\bigg),&\\
\pm\sin(\theta+\xi)\dfrac{\p \f_{\pm,0}}{\p\eta_{\pm}}+\f_{\pm,0}-\bar \f_{\pm,0}&=&0, \,\, 0<\eta_{\pm}<\infty,\\
\f_{\pm,0}(0,\theta,\xi)&=&g_{\pm}(\theta,\xi) \,\, \text{for}\ \
\pm\sin(\theta+\xi)>0,\\\rule{0ex}{1em}
\lim_{\eta_{\pm}\rt\infty}\f_{\pm,0}(\eta_{\pm},\theta,\xi)&=&f_{\pm,0}(\infty,\theta).
\end{array}
\right.
\end{eqnarray}
Assuming $g_{\pm}\in L^{\infty}(\Gamma^-)$, by Theorem \ref{Milne
theorem 1}, we can show there exist unique solutions
$\f_{\pm,0}(\eta_{\pm},\theta,\xi)\in
L^{\infty}([0,\infty)\times[-\pi,\pi)\times[-\pi,\pi))$. Hence,
$\ubc_{\pm,0}$ is well-defined. Then we define the zeroth order
interior solution as
\begin{eqnarray}\label{classical temp 2}
\left\{
\begin{array}{rcl}
\uc_0&=&\buc_0,\\\rule{0ex}{1em} \Delta_x\buc_0&=&0\ \ \text{in}\ \
\Omega,\\\rule{0ex}{1em} \buc_0&=&f_{\pm,0}(\infty,\theta)\ \
\text{on}\ \ \p\Omega_{\pm}.
\end{array}
\right.
\end{eqnarray}
\ \\
Step 2: Construction of $\ubc_{\pm,1}$ and $\uc_1$. \\
Let $\vx_{\pm}=R_{\pm}(\cos\theta,\sin\theta)$ and $\vw=(-\sin\xi,\cos\xi)$. Then, we define the first order boundary layer solution as
\begin{eqnarray}\label{classical temp 3}
\left\{
\begin{array}{rcl}
\ubc_{\pm,1}(\eta_{\pm},\theta,\xi)&=&\psi_0(\e\eta_{\pm})\bigg(\f_{\pm,1}(\eta_{\pm},\theta,\xi)-f_{\pm,1}(\infty,\theta)\bigg),\\
\pm\sin(\theta+\xi)\dfrac{\p
\f_{\pm,1}}{\p\eta_{\pm}}+\f_{\pm,1}-\bar
\f_{\pm,1}&=&\cos(\theta+\xi)\dfrac{\psi(\e\eta_{\pm})}{R_{\pm}\mp\e\eta_{\pm}}\dfrac{\p
\ubc_{\pm,0}}{\p\theta},\\\rule{0ex}{1em}
\f_{\pm,1}(0,\theta,\xi)&=&\vw\cdot\nx\uc_0(\vx_{\pm},\vw)\ \
\text{for}\ \ \pm\sin(\theta+\xi)>0,\\\rule{0ex}{1em}
\lim_{\eta_{\pm}\rt\infty}\f_{\pm,1}(\eta_{\pm},\theta,\xi)&=&f_{\pm,1}(\infty,\theta).
\end{array}
\right.
\end{eqnarray}
At the same time, we define the first order
interior solution as
\begin{eqnarray}\label{classical temp 5}
\left\{
\begin{array}{rcl}
\uc_1&=&\buc_1-\vw\cdot\nx\uc_0,\\\rule{0ex}{1em}
\Delta_x\buc_1&=&0\ \ \text{in}\ \ \Omega,\\\rule{0ex}{1em}
\buc_1&=&f_{\pm,1}(\infty,\theta)\ \ \text{on}\ \ \p\Omega_{\pm}.
\end{array}
\right.
\end{eqnarray}
\ \\
Step 3: Generalization to arbitrary $k$.\\
Similar to above procedure, we can define the $k^{th}$ order
boundary layer solution as
\begin{eqnarray}
\left\{
\begin{array}{rcl}
\ubc_{\pm,k}(\eta_{\pm},\theta,\xi)&=&\psi_0(\e\eta_{\pm})\bigg(\f_{\pm,k}(\eta_{\pm},\theta,\xi)-f_{\pm,k}(\infty,\theta)\bigg),\\
\pm\sin(\theta+\xi)\dfrac{\p
\f_{\pm,k}}{\p\eta_{\pm}}+\f_{\pm,k}-\bar
\f_{\pm,k}&=&\cos(\theta+\xi)\dfrac{\psi(\e\eta_{\pm})}{R_{\pm}\mp\e\eta_{\pm}}\dfrac{\p
\ubc_{\pm,k-1}}{\p\theta},\\\rule{0ex}{1em}
\f_{\pm,k}(0,\theta,\xi)&=&\vw\cdot\nx\uc_{k-1}(\vx_{\pm},\vw)\ \
\text{for}\ \ \pm\sin(\theta+\xi)>0,\\\rule{0ex}{1em}
\lim_{\eta_{\pm}\rt\infty}\f_{\pm,k}(\eta_{\pm},\theta,\xi)&=&f_{\pm,k}(\infty,\theta).
\end{array}
\right.
\end{eqnarray}
Define the $k^{th}$ order interior solution as
\begin{eqnarray}
\left\{
\begin{array}{rcl}
\uc_k&=&\buc_k-\vw\cdot\nx\uc_{k-1},\\\rule{0ex}{1em}
\Delta_x\buc_k&=&0\ \ \text{in}\ \ \Omega,\\\rule{0ex}{1em}
\buc_k&=&f_{\pm,k}(\infty,\theta)\ \ \text{on}\ \ \p\Omega_{\pm}.
\end{array}
\right.
\end{eqnarray}
Combing the above discussion, we are able to prove the following result:
\begin{theorem}\label{main fake 1}
Assume $g_{\pm}\in L^{\infty}(\Gamma^-)$ are sufficiently smooth. Then there exits a unique
solution $u^{\e}(\vx,\vw)\in L^{\infty}(\Omega\times\s^1)$  for the
steady neutron transport equation (\ref{transport}), which satisfies
\begin{eqnarray}\label{main fake theorem 1}
\lnm{u^{\e}-\uc_0-\ubc_{+,0}-\ubc_{-,0}}=O(\e).
\end{eqnarray}
\end{theorem}
\ \\

Our work begins with a crucial observation that based on Remark
\ref{Milne remark}, the existence of solution $\f_{\pm,1}$ requires
the source term
\begin{eqnarray}
\cos(\theta+\xi)\frac{\psi(\e\eta_{\pm})}{R_{\pm}\mp\e\eta_{\pm}}\frac{\p
\ubc_{\pm,0}}{\p\theta}\in
L^{\infty}([0,\infty)\times[-\pi,\pi)\times[-\pi,\pi)).
\end{eqnarray}
Since the support of $\psi(\e\eta_{\pm})$ depends on $\e$, by
(\ref{classical temp 1}), this in turn requires
\begin{eqnarray}
\frac{\p
}{\p\theta}\bigg(\f_{\pm,0}(\eta_{\pm},\theta,\xi)-f_{\pm,0}(\infty,\theta)\bigg)\in
L^{\infty}([0,\infty)\times[-\pi,\pi)\times[-\pi,\pi))
\end{eqnarray}
We can check that $Z_{\pm}=\ps (\f_{\pm,0}-f_{\pm,0}(\infty,\theta))$
satisfy
\begin{eqnarray}
\left\{
\begin{array}{rcl}
\pm\sin(\theta+\xi)\dfrac{\p Z_{\pm}}{\p\eta}+Z_{\pm}-\bar
Z_{\pm}&=&-\cos(\theta+\xi)\dfrac{\p
\f_{\pm,0}}{\p\eta_{\pm}},\\\rule{0ex}{2em}
Z_{\pm}(0,\theta,\xi)&=&\dfrac{\p
g_{\pm}(\theta,\xi)}{\p\theta}-\dfrac{\p
f_{\pm,0}(\infty,\theta)}{\p\theta}\ \ \text{for}\ \
\pm\sin(\theta+\xi)>0,\\\rule{0ex}{1.5em}
\lim_{\eta_{\pm}\rt\infty}Z_{\pm}(\eta_{\pm},\theta,\xi)&=&
Z_{\pm}(\infty,\theta).
\end{array}
\right.
\end{eqnarray}
On the one hand, we require the source term satisfy
\begin{eqnarray}
-\cos(\theta+\xi)\frac{\p \f_{\pm,0}}{\p\eta_{\pm}}\in
L^{\infty}([0,\infty)\times[-\pi,\pi)\times[-\pi,\pi))
\end{eqnarray}
to get a solution $Z_{\pm}\in
L^{\infty}([0,\infty)\times[-\pi,\pi)\times[-\pi,\pi))$ since we assume that $\ps g_{\pm}\in L^{\infty}(\Gamma^-)$.

On the other hand, as be shown by the Appendix of \cite{AA003}, it holds that $\px\f_{\pm,0}\notin
L^{\infty}([0,\infty)\times[-\pi,\pi)\times[-\pi,\pi))$ for some specific boundary condition $g_{\pm}$.
Due to the intrinsic singularity for (\ref{classical temp 1}), the construction
in \cite{Bensoussan.Lions.Papanicolaou1979} breaks down.\\

In fact, in general geometry domain with curved boundary, we need to
control the normal derivative of the boundary layer solution for the
Milne expansion. It is the main reason that we consider the following
$\e$-Milne expansion with geometric correction.

%%%%%%%%%%%%%%%%%%%%%%%%%%%%%%%%%%%%%%%%%%%%%%%%%%%%%%%%%%%%%%%%%%%%%%%%
\subsection{$\e$-Milne Expansion with Geometric Correction}
%%%%%%%%%%%%%%%%%%%%%%%%%%%%%%%%%%%%%%%%%%%%%%%%%%%%%%%%%%%%%%%%%%%%%%%%

Our main goal is to overcome the difficulty in estimating
\begin{eqnarray}\label{problematic}
\cos(\theta+\xi)\frac{\psi(\e\eta_{\pm})}{R_{\pm}\mp \e\eta_{\pm}}\frac{\p
\ubc_{\pm,k}}{\p\theta}.
\end{eqnarray}
We introduce one more substitution to decompose the term (\ref{problematic}). Now, we give the solution expansion in the following steps.\\
\ \\
{\it Substitution 1}: Define the interior expansion as follows:
\begin{eqnarray}
\u(\vx,\vw)\sim\sum_{k=0}^{\infty}\e^k\u_k(\vx,\vw)
\end{eqnarray}
where $\u_k$ satisfies the same equations as $\uc_k$ in
(\ref{interior 1}) and (\ref{interior 2}). Here, to highlight its
dependence on $\e$ via the $\e$-Milne problem and boundary data, we
add the superscript $\e$.\\
\ \\
{\it Substitution 2}: We make the rotation substitution for
$u^{\e}(\eta_{\pm},\theta,\xi)\rt u^{\e}(\eta_{\pm},\theta,\phi)$
with $(\eta_{\pm},\theta,\phi)\in
[0,(R_+-R_-)/\e]\times[-\pi,\pi)\times[-\pi,\pi)$ as
\begin{eqnarray}\label{substitution 4}
\left\{
\begin{array}{rcl}
\eta_{\pm}&=&\eta_{\pm},\\
\theta&=&\theta,\\
\phi&=&\theta+\xi,
\end{array}
\right.
\end{eqnarray}
and transform the equation (\ref{transport}) into
\begin{eqnarray}\label{transport temp}
\left\{ \begin{array}{l}\displaystyle \pm\sin\phi\frac{\p
u^{\e}}{\p\eta_{\pm}}-\frac{\e}{R_{\pm}\mp\e\eta_{\pm}}\cos\phi\bigg(\frac{\p
u^{\e}}{\p\phi}+\frac{\p
u^{\e}}{\p\theta}\bigg)+u^{\e}-\frac{1}{2\pi}\int_{-\pi}^{\pi}u^{\e}\ud{\phi}=0,\\\rule{0ex}{1.0em}
u^{\e}(0,\theta,\phi)=g_{\pm}(\theta,\phi)\ \ \text{for}\ \
\pm\sin\phi>0.
\end{array}
\right.
\end{eqnarray}
We define the $\e$-Milne expansion with geometric correction of
boundary layer as follows:
\begin{eqnarray}\label{boundary layer expansion}
\ub_{\pm}(\eta_{\pm},\theta,\phi)\sim\sum_{k=0}^{\infty}\e^k\ub_{\pm,k}(\eta_{\pm},\theta,\phi),
\end{eqnarray}
where $\ub_k$ can be determined by comparing the order of $\e$ via
plugging (\ref{boundary layer expansion}) into the equation
(\ref{transport temp}). Thus, in a neighborhood of the boundary, we
have
\begin{eqnarray}
\pm\sin\phi\frac{\p
\ub_{\pm,0}}{\p\eta_{\pm}}-\frac{\e}{R_{\pm}\mp\e\eta_{\pm}}\cos\phi\frac{\p
\ub_{\pm,0}}{\p\phi}+\ub_{\pm,0}-\bub_{\pm,0}&=&0,\label{expansion temp 5}\\
\pm\sin\phi\frac{\p
\ub_{\pm,1}}{\p\eta_{\pm}}-\frac{\e}{R_{\pm}\mp\e\eta_{\pm}}\cos\phi\frac{\p
\ub_{\pm,1}}{\p\phi}+\ub_{\pm,1}-\bub_{\pm,1}&=&\frac{\cos\phi}{R_{\pm}\mp\e\eta_{\pm}}\frac{\p
\ub_{\pm,0}}{\p\theta},\label{expansion temp 6}\\
\ldots\nonumber\\
\pm\sin\phi\frac{\p
\ub_{\pm,k}}{\p\eta_{\pm}}-\frac{\e}{R_{\pm}\mp\e\eta_{\pm}}\cos\phi\frac{\p
\ub_{\pm,k}}{\p\phi}+\ub_{\pm,k}-\bub_{\pm,k}&=&\frac{\cos\phi}{R_{\pm}\mp\e\eta_{\pm}}\frac{\p
\ub_{\pm,k-1}}{\p\theta}.
\end{eqnarray}
where
\begin{eqnarray}
\bub_{\pm,k}(\eta_{\pm},\theta)=\frac{1}{2\pi}\int_{-\pi}^{\pi}\ub_{\pm,k}(\eta_{\pm},\theta,\phi)\ud{\phi}.
\end{eqnarray}
Here the most important idea is to include the singular term
\begin{eqnarray}
-\frac{\e}{R_{\pm}\mp\e\eta_{\pm}}\cos\phi\frac{\p u^{\e}}{\p\phi}
\end{eqnarray}
in the Milne problem. It is notable that the solution $\ub_{\pm,k}$
depends on $\e$.

\noindent{\it Substitution 3}:
Similar to the classical expansion, we first
consider the boundary condition expansion
\begin{eqnarray}
\u_0+\ub_{\pm,0}&=&g_{\pm},\\
\u_1+\ub_{\pm,1}&=&0,\\
\ldots\nonumber\\
\u_k+\ub_{\pm,k}&=&0.
\end{eqnarray}
The construction of $\u_k$ and $\ub_k$ are as follows:\\
\ \\
Step 1: Construction of $\ub_{\pm,0}$ and $\u_0$.\\
We refer to the cut-off function $\psi$ and $\psi_0$ as
(\ref{cut-off 1}) and (\ref{cut-off 2}), and define the force as
\begin{eqnarray}\label{force}
F_{\pm}(\e;\eta_{\pm})=-\frac{\e\psi(\e\eta_{\pm})}{R_{\pm}\mp\e\eta_{\pm}},
\end{eqnarray}
The zeroth order boundary layer solutions is defined as
\begin{eqnarray}\label{expansion temp 9}
\left\{
\begin{array}{ccl}
\ub_{\pm,0}(\eta_{\pm},\theta,\phi)=\psi_0(\e\eta_{\pm})\bigg(f_{\pm,0}^{\e}(\eta_{\pm},\theta,\phi)-f^{\e}_{\pm,0}(\infty,\theta)\bigg),
& & \vspace{3pt}\\
\pm\sin\phi\dfrac{\p
f_{\pm,0}^{\e}}{\p\eta_{\pm}}+F_{\pm}(\e;\eta_{\pm})\cos\phi\dfrac{\p
f_{\pm,0}^{\e}}{\p\phi}+f_{\pm,0}^{\e}-\bar f_{\pm,0}^{\e}= 0,\vspace{3pt}\\
f_{\pm,0}^{\e}(0,\theta,\phi)= g_{\pm}(\theta,\phi)\ \ \text{for}\
\ \pm\sin\phi>0,\vspace{3pt}\\\rule{0ex}{1em}
\lim_{\eta_{\pm}\rt\infty}f_{\pm,0}^{\e}(\eta_{\pm},\theta,\phi)= f^{\e}_{\pm,0}(\infty,\theta).
\end{array}
\right.
\end{eqnarray}
In contrast to the classical Milne problem (\ref{classical temp 1}),
the key advantage is, due to the geometry, $\dfrac{\p
F_{\pm}(\e;\eta_{\pm})}{\p\theta}=0$, such that (\ref{expansion temp
9}) is invariant in $\theta$.

With the asymptotic behavior of the boundary layer solution in hand, we define the zeroth order interior solution $\u_0(\vx)$ as
\begin{eqnarray}\label{expansion temp 8}
\left\{
\begin{array}{rcl}
\u_0&=&\bu_0,\\\rule{0ex}{1em} \Delta_x\bu_0&=&0\ \ \text{in}\ \
\Omega,\\\rule{0ex}{1em} \bu_0&=&f_{\pm,0}^{\e}(\infty,\theta)\ \
\text{on}\ \ \p\Omega_{\pm}.
\end{array}
\right.
\end{eqnarray}
Notice that the asymptotic state depends on $\epsilon$, then the interior solution depends on $\epsilon$ too.
\ \\
Step 2: Estimates of $\dfrac{\p\ub_{\pm,0}}{\p\theta}$.\\
By Theorem \ref{Milne theorem 2}, we can easily see $f_{\pm,0}^{\e}$
are well-defined in $L^{\infty}(\Omega\times\s^1)$ and approach
to $f_{\pm,0}^{\e}(\infty)$ exponentially fast as
$\eta_{\pm}\rt\infty$. Then we can derive $Z_{\pm}=\ps
(f_{\pm,0}^{\e}-f^{\e}_{\pm,0}(\infty))$ also satisfies the same
type of $\e$-Milne problem
\begin{eqnarray}
\left\{ \begin{array}{ccl}\displaystyle \pm\sin\phi\frac{\p
Z_{\pm}}{\p\eta_{\pm}}+F(\e;\eta)\cos\phi\frac{\p
Z_{\pm}}{\p\phi}+Z_{\pm}-\bar Z_{\pm}=0,\\
Z_{\pm}(0,\theta,\phi)=\dfrac{\p g_{\pm}}{\p\theta}-\dfrac{\p
f_{\pm,0}^{\e}(\infty)}{\p\theta}\ \ \text{for}\ \
\pm\sin\phi>0,\\\rule{0ex}{1em}
\lim_{\eta_{\pm}\rt\infty}Z_{\pm}(\eta_{\pm},\phi)=C.
\end{array}
\right.
\end{eqnarray}
By Theorem \ref{Milne theorem 2}, we can see $Z_{\pm}\rt C$
exponentially fast as $\eta_{\pm}\rt\infty$. It is natural to obtain
this constant $C$ must be zero. Hence, if $g_{\pm}\in
C^r(\Gamma^-)$, it is obvious to check $f_{\pm,0}^{\e}(\infty)\in
C^r(\p\Omega)$. By the standard elliptic estimate in (\ref{expansion
temp 8}), there exists a unique solution $\bu_{\pm,0}\in
W^{r,p}(\Omega)$ for arbitrary $p\geq2$ satisfying
\begin{eqnarray}
\nm{\bu_{\pm,0}}_{W^{r,p}(\Omega)}\leq
C(\Omega)\nm{f_{\pm,0}^{\e}(\infty)}_{W^{r-1/p,p}(\p\Omega)},
\end{eqnarray}
which implies $\nx\bu_{0}\in W^{r-1,p}(\Omega)$, $\nx\bu_{0}\in
W^{r-1-1/p,p}(\p\Omega)$ and $\bu_{0}\in C^{r-1,1-2/p}(\Omega)$.\\
\ \\
Step 3: Construction of $\ub_{\pm,1}$ and $\u_1$.\\
The first order boundary layer solution is defined as
\begin{eqnarray}\label{expansion temp 11}
\left\{
\begin{array}{ccl}
\ub_{\pm,1}(\eta_{\pm},\theta,\phi)=
\psi_0(\e\eta_{\pm})\bigg(f_{\pm,1}^{\e}(\eta_{\pm},\theta,\phi)-f_{\pm,1}^{\e}(\infty,\theta)\bigg),\vspace{3pt}\\
\pm\sin\phi\dfrac{\p
f_{\pm,1}^{\e}}{\p\eta_{\pm}}+F_{\pm}(\e;\eta_{\pm})\cos\phi\dfrac{\p
f_{\pm,1}^{\e}}{\p\phi}+f_{\pm,1}^{\e}-\bar
f_{\pm,1}^{\e}=\dfrac{\psi(\e\eta_{\pm})}{R_{\pm}\mp\e\eta_{\pm}}\cos\phi\dfrac{\p
\ub_{\pm,0}}{\p\theta},\vspace{3pt}\\\rule{0ex}{1em}
f_{\pm,1}^{\e}(0,\theta,\phi)=\vw\cdot\nx\u_0(\vx_{\pm},\vw)\ \
\text{for}\ \ \pm\sin\phi>0,\vspace{3pt}\\\rule{0ex}{1em}
\lim_{\eta_{\pm}\rt\infty}f_{\pm,1}^{\e}(\eta_{\pm},\theta,\phi)=f_{\pm,1}^{\e}(\infty,\theta).
\end{array}
\right.
\end{eqnarray}
 Then, we define the first
order interior solution $\u_1(\vx)$ as
\begin{eqnarray}\label{expansion temp 10}
\left\{
\begin{array}{rcl}
\u_1&=&\bu_1-\vw\cdot\nx\u_0,\\\rule{0ex}{1em} \Delta_x\bu_1&=&0\ \
\text{in}\ \ \Omega,\\\rule{0ex}{1em}
\bu_1&=&f^{\e}_{\pm,1}(\infty,\theta)\ \ \text{on}\ \
\p\Omega_{\pm}.
\end{array}
\right.
\end{eqnarray}
\ \\
Step 4: Estimates of $\dfrac{\p\ub_{\pm,1}}{\p\theta}$.\\
By Theorem \ref{Milne theorem 2}, we can easily see $f_{\pm,1}^{\e}$
is well-defined in $L^{\infty}(\Omega\times\s^1)$ and approaches
$f_{\pm,1}^{\e}(\infty)$ exponentially fast as
$\eta_{\pm}\rt\infty$. Also, since $\nx\u_0\in
W^{r-1-1/p,p}(\p\Omega)$, $\ps f_{\pm,1}^{\e}$ is well-defined and
decays exponentially fast. Hence, $f^{\e}_{\pm,1}(\infty,\theta)\in
W^{r-1-1/p,p}(\p\Omega)$. By the standard elliptic estimate in
(\ref{expansion temp 10}), there exists a unique solution
$\bu_{\pm,1}\in W^{r-1,p}(\Omega)$ and satisfies
\begin{eqnarray}
\nm{\bu_{\pm,1}}_{W^{r-1,p}(\Omega)}\leq
C(\Omega)\nm{f^{\e}_{\pm,1}(\infty)}_{W^{r-1-1/p,p}(\p\Omega)},
\end{eqnarray}
which implies $\nx\bu_1\in W^{r-2,p}(\Omega)$, $\nx\bu_1\in
W^{r-2-1/p,p}(\p\Omega)$ and $\bu_1\in C^{r-2,1-2/p}(\Omega)$.\\
\ \\
Step 5: Generalization to arbitrary $k$.\\
In a similar fashion, as long as $g$ is sufficiently smooth, above
process can be continued. We construct the $k^{th}$ order boundary layer
solution as
\begin{eqnarray}
\left\{
\begin{array}{ccl}
\ub_{\pm,k}(\eta_{\pm},\theta,\phi)=\psi_0(\e\eta_{\pm})\bigg(f_{\pm,k}^{\e}(\eta_{\pm},\theta,\phi)
-f_{\pm,k}^{\e}(\infty,\theta)\bigg),\vspace{3pt}\\
\pm\sin\phi\dfrac{\p
f_{\pm,k}^{\e}}{\p\eta_{\pm}}+F_{\pm}(\e;\eta_{\pm})\cos\phi\dfrac{\p
f_{\pm,k}^{\e}}{\p\phi}+f_{\pm,k}^{\e}-\bar
f_{\pm,k}^{\e}=\dfrac{\psi(\e\eta_{\pm})}{R_{\pm}\mp\e\eta}\cos\phi\dfrac{\p
\ub_{\pm,k-1}}{\p\theta},\vspace{3pt}\\\rule{0ex}{1em}
f_{\pm,k}^{\e}(0,\theta,\phi)=\vw\cdot\nx\u_{k-1}(\vx_{\pm},\vw)\
\ \text{for}\ \ \pm\sin\phi>0,\vspace{3pt}\\\rule{0ex}{1em}
\lim_{\eta_{\pm}\rt\infty}f_{\pm,k}^{\e}(\eta_{\pm},\theta,\phi)= f_{\pm,k}^{\e}(\infty,\theta).
\end{array}
\right.
\end{eqnarray}
Then we define the $k^{th}$ order interior solution as
\begin{eqnarray}
\left\{
\begin{array}{rcl}
\u_k&=&\bu_k-\vw\cdot\nx\u_{k-1},\\\rule{0ex}{1em}
\Delta_x\bu_k&=&0\ \ \text{in}\ \ \Omega,\\\rule{0ex}{1em}
\bu_k&=&f^{\e}_{\pm,k}(\infty,\theta)\ \ \text{on}\ \ \p\Omega.
\end{array}
\right.
\end{eqnarray}
For $g_{\pm}\in C^{k+1}(\Gamma^-)$, the interior solution and
boundary layer solution can be well-defined up to $k^{th}$ order,
i.e. up to $\u_k$ and $\ub_{\pm,k}$.

%%%%%%%%%%%%%%%%%%%%%%%%%%%%%%%%%%%%%%%%%%%%%%%%%%%%%%%%%%%%%%%%%%%%%%%%
\section{$\e$-Milne Problem with Geometric Correction}%%%%%%%%%%%%%%%%%%
%%%%%%%%%%%%%%%%%%%%%%%%%%%%%%%%%%%%%%%%%%%%%%%%%%%%%%%%%%%%%%%%%%%%%%%%

We consider the $\e$-Milne problem for
$f_{\pm}^{\e}(\eta_{\pm},\theta,\phi)$ in the domain
$(\eta_{\pm},\theta,\phi)\in[0,\infty)\times[-\pi,\pi)\times[-\pi,\pi)$
\begin{eqnarray}\label{Milne problem.}
\left\{
\begin{array}{rcl}\displaystyle
\pm\sin\phi\frac{\p
f_{\pm}^{\e}}{\p\eta_{\pm}}+F_{\pm}(\e;\eta_{\pm})\cos\phi\frac{\p
f_{\pm}^{\e}}{\p\phi}+f_{\pm}^{\e}-\bar f_{\pm}^{\e}&=&S_{\pm}^{\e}(\eta_{\pm},\theta,\phi),\\
f_{\pm}^{\e}(0,\theta,\phi)&=&h_{\pm}^{\e}(\theta,\phi)\ \ \text{for}\ \ \pm\sin\phi>0,\\
\lim_{\eta_{\pm}\rt\infty}f_{\pm}^{\e}(\eta_{\pm},\theta,\phi)&=&f^{\e}_{\pm,\infty}(\theta),
\end{array}
\right.
\end{eqnarray}
where
\begin{eqnarray}\label{Milne average}
\bar
f_{\pm}^{\e}(\eta_{\pm},\theta)=\frac{1}{2\pi}\int_{-\pi}^{\pi}f_{\pm}^{\e}(\eta_{\pm},\theta,\phi)\ud{\phi},
\end{eqnarray}
\begin{eqnarray}
F_{\pm}(\e;\eta_{\pm})=-\frac{\e}{R_{\pm}\mp\e\eta_{\pm}},
\end{eqnarray}
\begin{eqnarray}\label{Milne bounded}
\abs{h_{\pm}^{\e}(\theta,\phi)}\leq M,
\end{eqnarray}
and
\begin{eqnarray}\label{Milne decay}
\abs{S_{\pm}^{\e}(\eta_{\pm},\theta,\phi)}\leq Me^{-K\eta_{\pm}},
\end{eqnarray}
for $M>0$ and $K>0$ uniform in $\e$ and $\theta$.

The well-posedness and decay of $f_{+}$ has been extensively studied
in \cite[Section 4]{AA003}, so in this section, we will focus on the
following:
\begin{eqnarray}\label{Milne problem..}
\left\{
\begin{array}{rcl}\displaystyle
-\sin\phi\frac{\p
f_{-}^{\e}}{\p\eta_{-}}+F_{-}(\e;\eta_{-})\cos\phi\frac{\p
f_{-}^{\e}}{\p\phi}+f_{-}^{\e}-\bar f_{-}^{\e}&=&S_{-}^{\e}(\eta_{-},\theta,\phi),\\
f_{-}^{\e}(0,\theta,\phi)&=&h_{-}^{\e}(\theta,\phi)\ \ \text{for}\ \ -\sin\phi>0,\\
\lim_{\eta_{-}\rt\infty}f_{-}^{\e}(\eta_{-},\theta,\phi)&=&f^{\e}_{-,\infty}(\theta),
\end{array}
\right.
\end{eqnarray}
\ \\
Introducing the sign substitution $\phi_{-}=-\phi$, we have
\begin{eqnarray}
\left\{
\begin{array}{rcl}\displaystyle
\sin\phi_{-}\frac{\p
f_{-}^{\e}}{\p\eta_{-}}-F_{-}(\e;\eta_{-})\cos\phi_{-}\frac{\p
f_{-}^{\e}}{\p\phi_{-}}+f_{-}^{\e}-\bar f_{-}^{\e}&=&S_{-}^{\e}(\eta_{-},\theta,\phi_{-}),\\
f_{-}^{\e}(0,\theta,\phi_{-})&=&h_{-}^{\e}(\theta,\phi_{-})\ \ \text{for}\ \ \sin\phi_{-}>0,\\
\lim_{\eta_{-}\rt\infty}f_{-}^{\e}(\eta_{-},\theta,\phi_{-})&=&f^{\e}_{-,\infty}(\theta),
\end{array}
\right.
\end{eqnarray}
No abusing of the notation, we temporarily ignore the subscript $-$,
superscript $\e$, and the dependence on $\theta$ to have
\begin{eqnarray}\label{Milne problem}
\left\{
\begin{array}{rcl}\displaystyle
\sin\phi\frac{\p f}{\p\eta}-F(\eta)\cos\phi\frac{\p
f}{\p\phi}+f-\bar f&=&S(\eta,\phi),\\
f(0,\phi)&=&h(\phi)\ \ \text{for}\ \ \sin\phi>0,\\
\lim_{\eta\rt\infty}f(\eta,\phi)&=&f_{\infty},
\end{array}
\right.
\end{eqnarray}
where
\begin{eqnarray}\label{Milne average}
\bar f(\eta)=\frac{1}{2\pi}\int_{-\pi}^{\pi}f(\eta,\phi)\ud{\phi},
\end{eqnarray}
\begin{eqnarray}
F(\eta)=-\frac{\e \psi(\e \eta)}{R+\e\eta}
\end{eqnarray}
\begin{eqnarray}\label{Milne bounded}
\abs{h(\phi)}\leq M,
\end{eqnarray}
and
\begin{eqnarray}\label{Milne decay}
\abs{S(\eta,\phi)}\leq Me^{-K\eta},
\end{eqnarray}
for $M>0$ and $K>0$. We may further define a potential function
$V(\eta)$ satisfying
\begin{eqnarray}\label{Potential Function1}
V(\e, 0)=0,\qquad \dfrac{\ud{V}(\e,\eta)}{\ud{\eta}}=-F(\e, \eta)\geq 0 .
\end{eqnarray}
Since $V(\e,\eta)$ is monotonically increasing w.r.t $\eta$ and $supp[\psi] \subset [0,3/4]$, we can derive
\begin{eqnarray} \label{Property of Potential fuction}
 0\leq V(\e,\eta)\leq \ln 4,~~~1\leq \exp(V(\e,\eta))\leq 4 ~~\text{for~ all}~ \eta \in [0,\infty].
\end{eqnarray}

In this section, we introduce some special notation to describe the
norms in the space $(\eta,\phi)\in[0,\infty)\times[-\pi,\pi)$.
Define the $L^2$ norm as follows:
\begin{eqnarray}
\tnm{f(\eta,\cdot)}&=&\bigg(\int_{-\pi}^{\pi}\abs{f(\eta,\phi)}^2\ud{\phi}\bigg)^{1/2},\\
\tnnm{f}&=&\bigg(\int_0^{\infty}\int_{-\pi}^{\pi}\abs{f(\eta,\phi)}^2\ud{\phi}\ud{\eta}\bigg)^{1/2}.
\end{eqnarray}
Define the inner product in $\phi$ space
\begin{eqnarray}
\br{f,g}_{\phi}(\eta)=\int_{-\pi}^{\pi}f(\eta,\phi)g(\eta,\phi)\ud{\phi}.
\end{eqnarray}
Define the $L^{\infty}$ norm as follows:
\begin{eqnarray}
\lnm{f(\eta)}&=&\sup_{\phi\in[-\pi,\pi)}\abs{f(\eta,\phi)},\\
\lnnm{f}&=&\sup_{(\eta,\phi)\in[0,\infty)\times[-\pi,\pi)}\abs{f(\eta,\phi)},\\
\ltnm{f}&=&\sup_{\eta\in[0,\infty)}\bigg(\int_{-\pi}^{\pi}\abs{f(\eta,\phi)}^2\ud{\phi}\bigg)^{1/2}.
\end{eqnarray}
Since the boundary data $h(\phi)$ is only defined on $\sin\phi>0$,
we naturally extend above definitions on this half-domain as
follows:
\begin{eqnarray}
\tnm{h}&=&\bigg(\int_{\sin\phi>0}\abs{h(\phi)}^2\ud{\phi}\bigg)^{1/2},\\
\lnm{h}&=&\sup_{\sin\phi>0}\abs{h(\phi)}.
\end{eqnarray}
\begin{lemma}\label{Milne data}
We have
\begin{eqnarray}
\tnm{h}&\leq&C\lnm{h}\leq CM\\
\tnnm{S}&\leq&C\frac{M}{K}\\
\ltnm{S}&\leq&C\lnnm{S}\leq CM
\end{eqnarray}
\end{lemma}
\begin{proof}
They can be verified via direct computation, so we omit the proofs
here.
\end{proof}

%%%%%%%%%%%%%%%%%%%%%%%%%%%%%%%%%%%%%%%%%%%%%%%%%%%%%%%%%%%%%%%%%%%%%%%%
\subsection{$L^2$ Estimates}
%%%%%%%%%%%%%%%%%%%%%%%%%%%%%%%%%%%%%%%%%%%%%%%%%%%%%%%%%%%%%%%%%%%%%%%%

%%%%%%%%%%%%%%%%%%%%%%%%%%%%%%%%%%%%%%%%%%%%%%%%%%%%%%%%%%%%%%%%%%%%%%%%
\subsubsection{Finite Slab with $\bar S=0$}
%%%%%%%%%%%%%%%%%%%%%%%%%%%%%%%%%%%%%%%%%%%%%%%%%%%%%%%%%%%%%%%%%%%%%%%%

Consider the $\e$-Milne problem for $f^{L}(\eta,\phi)$ in a finite
slab $(\eta,\phi)\in[0,L]\times[-\pi,\pi)$
\begin{eqnarray}\label{Milne finite problem LT}
\left\{
\begin{array}{rcl}\displaystyle
\sin\phi\frac{\p f^{L}}{\p\eta}-F(\eta)\cos\phi\frac{\p
f^{L}}{\p\phi}+f^{L}-\bar f^{L}&=&S(\eta,\phi),\\
f^{L}(0,\phi)&=&h(\phi)\ \ \text{for}\ \ \sin\phi>0,\\
f^{L}(L,\phi)&=&f^{L}(L,R\phi),
\end{array}
\right.
\end{eqnarray}
where $R\phi=-\phi$ and $S$ satisfies $\bar S(\eta)=0$ for any
$\eta$. We may decompose the solution
\begin{eqnarray}
f^{L}(\eta,\phi)=q_f^{L}(\eta)+r_f^{L}(\eta,\phi),
\end{eqnarray}
where the hydrodynamical part $q_f^{L}$ of $f^L$, and the microscopic part $r_f^{L}$ is
the orthogonal complement, i.e.
\begin{eqnarray}\label{hydro}
q_f^{L}(\eta)= \bar{f}^L=\frac{1}{2\pi}\int_{-\pi}^{\pi}f^{L}(\eta,\phi)\ud{\phi},\quad
r_f^{L}(\eta,\phi)=f^{L}(\eta,\phi)-q_f^{L}(\eta).
\end{eqnarray}
In the following, we simply write
$f^{L}=q^{L}+r^{L}$ without any confusion.

\begin{lemma}\label{Milne finite LT}
Assume $\bar S(\eta)=0~(\eta\in[0,L])$ and $S(\eta,\phi)$ satisfy (\ref{Milne
bounded}) and (\ref{Milne decay}). Then there exists a solution
$f(\eta,\phi) \in L^2([0,L]\times[0, 2\pi])$ to the finite slab problem (\ref{Milne finite problem
LT}) such that
\begin{eqnarray}
\label{Milne temp 1}
\int_0^L\tnm{r^{L}(\eta,\cdot)}^2\ud{\eta}&\leq&C\bigg(M+\frac{M}{K}\bigg)^2<\infty,\\
|q^{L}(\eta)| &\leq&
C\bigg(1+M+\frac{M}{K}\bigg)(1+\eta^{1/2})+\tnm{r^{L}(\eta,\cdot)}\label{Milne
temp 2},\\
\br{\sin\phi,r^{L}}_{\phi}(\eta)&=&0,\label{Milne temp 3}
\end{eqnarray}
for arbitrary $\eta\in[0,L]$.
\end{lemma}
\begin{proof}
We divide the proof into several steps: we firstly establish the solution of the penalty problem. Secondly, the uniform estimates of the solution will be obtained. Then, by passing the limit, we achieve a solution of (\ref{Milne finite problem LT}) which satisfies (\ref{Milne temp 1}) and (\ref{Milne temp 2}). At last, the orthogonal property (\ref{Milne temp 3}) follows easily.\\
\ \\
Step 1: {\emph{Existence of the solution for the penalty problem}}

Firstly, we consider the existence of the solution for the penalized $\e$-Milne
equation with $\l>0$
\begin{eqnarray}\label{Milne finite problem LT penalty.}
\left\{
\begin{array}{rcl}\displaystyle
\l f_{\l}+\sin\phi\frac{\p f_{\l}}{\p\eta}-F(\eta)\cos\phi\frac{\p
f_{\l}}{\p\phi}+f_{\l}&=&S(\eta,\phi),\\
f_{\l}(0,\phi)&=&h(\phi)\ \ \text{for}\ \ \sin\phi>0,\\
f_{\l}(L,\phi)&=&f_{\l}(L,R\phi).
\end{array}
\right.
\end{eqnarray}
Here $\tnnm{S}<\infty$ and $\tnm{h}<\infty$.

For the existence of the solution for (\ref{Milne finite problem LT penalty.}), we also construct the iterative sequence. Define $f_{0}^{L}=0$ and $\{f_{m}^{L}\}_{m=1}^{\infty}$ are the solution of the following problem
\begin{eqnarray}\label{mt31}
\left\{
\begin{array}{rcl}\displaystyle
\l f_{m}^{L}+\sin\phi\frac{\p
f_{m}^{L}}{\p\eta}-F(\eta)\cos\phi\frac{\p
f_{m}^{L}}{\p\phi}+f_{m}^{L}-\bar f_{m-1}^{L}&=&S(\eta,\phi),\\
f_{m}^{L}(0,\phi)&=&h(\phi)\ \ \text{for}\ \ \sin\phi>0,\\
f_{m}^{L}(L,\phi)&=&f_{m}^{L}(L,R\phi).
\end{array}
\right.
\end{eqnarray}
For $m=1$, we have $\bar f_0^{L}=0$. Multiplying $f_{1}^{L}$ on both sides of (\ref{mt31}) and integrating over $\phi\in[-\pi,\pi]$, we obtain
\begin{eqnarray}
\half\frac{\ud{}}{\ud{\eta}}\br{f_1^{L},f_1^{L}\sin\phi}_{\phi}-\half F(\eta)\br{f_1^{L},f_1^{L}\sin\phi}_{\phi}
+ (1+\l)\tnm{f_1^{L}}^2 = \br{ S,~f_1^{L}}_{\phi}.
\end{eqnarray}
It leads to
\begin{eqnarray}
&& \exp[V(L)]\br{f_1^{L},f_1^{L}\sin\phi}_{\phi}(L) + (1+\l)\int_{0}^{L}\exp[V(y)]
\tnm{f_{1}^{L}(y,\cdot)}^2 \ud{y}\\
&  =& \br{f_1^{L},f_1^{L}\sin\phi}_{\phi}(0)+\int_{0}^{L}\exp[V(y)]
 \langle S, f_1^L \rangle_{\phi}\ud{y}\no\\
 &\leq & || h||_{L^2}  +\int_{0}^{L}\exp[V(y)]
 \bigg[\frac{1}{2}|| S(y,\cdot)||_{L^2}^2 +\frac{1}{2} ||f_1^L(y,\cdot)||_{L^2}^2 \bigg] \ud{y} .\nonumber
\end{eqnarray}
where $V$ is the potential function which is defined in (\ref{Potential Function1}). It is easy to show that
$\exp[V(L)]\br{f_1^{L},f_1^{L}\sin\phi}_{\phi}(L)=0 $ because the specular refection condition at $\eta=L$. From (\ref{Property of Potential fuction}) and the above inequality, one obtains that
\begin{eqnarray}
\int_0^L\tnm{f_{1}^{L}(y,\cdot)}^2\ud{y} \leq  8
\bigg(\tnm{h}^2+\tnnm{S}^2\bigg).
\end{eqnarray}
For $m\geq1$, we set $g_{m}^{L}=f_{m}^{L}-f_{m-1}^{L}$. It satisfies
\begin{eqnarray}\label{mt 03}
\left\{
\begin{array}{rcl}\displaystyle
\l g_{m}^{L}+\sin\phi\frac{\p
g_{m}^{L}}{\p\eta}-F(\eta)\cos\phi\frac{\p
g_{m}^{L}}{\p\phi}+g_{m}^{L}-\bar g_{m-1}^{L}&=&0,\\
g_{m}^{L}(0,\phi)&=&0\ \ \text{for}\ \ \sin\phi>0,\\
g_{m}^{L}(L,\phi)&=&g_{m}^{L}(L,R\phi).
\end{array}
\right.
\end{eqnarray}
By a similar argument, we can derive that
\begin{eqnarray}
&& \exp[V(L)]\br{g_{m}^{L},g_{m}^{L}\sin\phi}_{\phi}(L)+\int_{0}^{L}\exp[V(y)]
\bigg((1+\l)\tnm{g_{m}^{L}(y,\cdot)}^2+ \br{g_{m}^{L},\bar
g_{m-1}^{L}}_{\phi}(y)\bigg)\ud{y} \\
&& \leq  \br{g_{m}^{L},g_{m}^{L}\sin\phi}_{\phi}(0) \nonumber
\end{eqnarray}
Due to the facts $\br{g_{m}^{L},g_{m}^{L}\sin\phi}_{\phi}(0)\leq 0,~~\exp[V(L)]\br{g_{m}^{L},g_{m}^{L}\sin\phi}_{\phi}(L)=0$,
we have
\begin{eqnarray}
\int_0^L\exp[V(y)]\tnm{g_{m}^{L}(y,\cdot)}^2\ud{y}&\leq&\frac{1}{1+\l}\int_0^L\exp[V(y)]\br{g_{m}^{L},\bar g_{m-1}^{L}}_{\phi}(y)\ud{y}\\
&\leq&\frac{1}{2(1+\l)}\int_0^L\exp[V(y)]\bigg(\tnm{g_{m}^{L}(y,\cdot)}^2+ \tnm{g_{m-1}^{L}(y,\cdot)}^2\bigg)\ud{y}.\no
\end{eqnarray}
It implies
\begin{eqnarray}
\int_0^L\exp[V(y)]\tnm{g_{m}^{L}(y,\cdot)}^2\ud{y}&\leq&\frac{1}{1+\l}\int_0^L\exp[V(y)]\tnm{g_{m-1}^{L}(y,\cdot)}^2\ud{y}.
\end{eqnarray}
Since $\exp[V(y)] \geq 1$ for all $y\geq 0$, we know that $\{f^L_m\}_{m=1}^{\infty}$ is a contraction sequence for $\l>0$.
Let $m \rightarrow \infty$, we obtain a solution $f_{\l}^{L}$ of
(\ref{Milne finite problem LT penalty.}). Moreover, it satisfies
\begin{eqnarray}\label{One Estimate of Milne finite problem LT penalty}
\tnnm{f_{\l}^{L}}^2 \leq \tnnm{f_{1}^{L}}^2 +\sum_{m=1}^{\infty}\tnnm{g_{m}^{L}}^2 \leq \frac{8(1+\l)}{\l}\bigg(\tnm{h}^2+\tnnm{S}^2\bigg).
\end{eqnarray}
In the following, we prove the uniqueness of the solution for (\ref{Milne finite problem LT penalty.}). Assume that there are two solutions $f_1^L$ and $f_2^L$ of (\ref{Milne finite problem LT penalty.}). Define $f^L_{\ast}=f_1-f_2$, it satisfies
\begin{eqnarray}
\left\{
\begin{array}{rcl}\displaystyle
\l f^L_{\ast}+\sin\phi\frac{\p
f^L_{\ast}}{\p\eta}-F(\eta)\cos\phi\frac{\p
f^L_{\ast}}{\p\phi}+f^L_{\ast}-\bar f^L_{\ast}&=&0,\\
f^L_{\ast}(0,\phi)&=&0\ \ \text{for}\ \ \sin\phi>0,\\
f^L_{\ast}(L,\phi)&=&f^L_{\ast}(L,R\phi).
\end{array}
\right.
\end{eqnarray}
Similarly, we have
\begin{eqnarray}
\tnnm{f_{\ast}^{L}}^2 \leq \frac{1}{1+\l}\tnnm{\bar{f}_{\ast}^{L}}^2 \leq \frac{1}{1+\l}\tnnm{f_{\ast}^{L}}^2,
\end{eqnarray}
which further implies $f_{\ast}^{L}=0$ when $\lambda >0$. Therefore, the solution to (\ref{Milne finite problem LT penalty.}) is unique in $L^{2}([0,L]\times[-\pi,\pi))$.\\
\ \\
Step 2: {\emph{The uniform estimates of $r_{\l}^L$ and $q_{\l}^L$ w.r.t. $\l$}.}\\
{\it \bf Claim}: We claim that $r_{\l}^{L}$ satisfies
\begin{eqnarray}\label{Milne temp 37}
\int_{0}^L\tnm{r_{\l}^{L}(\eta,\cdot)}^2\ud{\eta} \leq
4\tnm{h}^2+8\int_{0}^L\tnm{S(\eta,\cdot)}^2\ud{\eta}
\end{eqnarray}
and, for any $0\leq \eta \leq L$, $q_{\l}^{L}$ satisfies
\begin{eqnarray}\label{Milne temp 38}
|q^{L}_{\l}(\eta)|\leq
16\pi (1+\l+|F(\eta)|)(1+\eta^{1/2}) ( \tnm{h} + \tnnm{S}) + \tnm{r_{\l}^L(\eta,\cdot)}.
\end{eqnarray}

\underline{\emph{The proof of (\ref{Milne temp 37}):}} The assumption $\bar S(\eta)=0$ leads to
\begin{eqnarray}
\br{S,f_{\l}^{L}}_{\phi}(\eta)=\br{S,q_{\l}^{L}}_{\phi}(\eta)
+\br{S,r_{\l}^{L}}_{\phi}(\eta)= \br{S,r_{\l}^{L}}_{\phi}(\eta).
\end{eqnarray}
Multiplying $f_{\l}^{L}$
on both sides of (\ref{Milne finite problem LT penalty.}) and
integrating over $\phi\in[-\pi,\pi)$, we get the energy estimate
\begin{eqnarray}\label{Milne temp 31}
&& \half\frac{\ud{}}{\ud{\eta}}\br{f_{\l}^{L}
,f_{\l}^{L}\sin\phi}_{\phi}(\eta) - \half F(\eta)\br{f_{\l}^{L}
,f_{\l}^{L}\sin\phi}_{\phi}(\eta)\\
&&\hspace{2cm} =-\l\tnm{f_{\l}^{L}(\eta,\cdot)}^2-\tnm{r_{\l}^{L}(\eta,\cdot)}^2+\br{S,r_{\l}^{L}}_{\phi}(\eta).\nonumber
\end{eqnarray}
%Hence, we have the simplified form of (\ref{Milne temp 31}) as
%follows:
%\begin{eqnarray}\label{Milne temp 32}
%\half\frac{\ud{}}{\ud{\eta}}\br{
%f_{\l}^{L},f_{\l}^{L}\sin\phi}_{\phi}(\eta)=-\l\tnm{f_{\l}^{L}(\eta)}^2-\tnm{r_{\l}^{L}(\eta)}^2+\half
%F(\eta)\br{
%f_{\l}^{L},f_{\l}^{L}\sin\phi}_{\phi}(\eta)+\br{S,r_{\l}^{L}}_{\phi}(\eta).
%\end{eqnarray}
Define
\begin{eqnarray}
\alpha(\eta)=\half\br{f_{\l}^{L},f_{\l}^{L}\sin\phi }_{\phi}(\eta).
\end{eqnarray}
Then (\ref{Milne temp 31}) can be rewritten as follows
\begin{eqnarray}
\frac{\ud{\alpha}}{\ud{\eta}}-F(\eta)\alpha(\eta)=-\l\tnm{f_{\l}^{L}(\eta,\cdot)}^2-\tnm{r_{\l}^{L}(\eta,\cdot)}^2
+\br{S,r_{\l}^{L}}_{\phi}(\eta).
\end{eqnarray}
The specular reflexive boundary
$f_{\l}^{L}(L,\phi)=f_{\l}^{L}(L,R\phi)$ ensures $\alpha(L)=0$. We can integrate above on $[\eta,L]$ and $[0,\eta]$ respectively to
obtain
\begin{eqnarray}
\label{Milne temp 4}
&&\exp[V(\eta)]\alpha(\eta)=\int_{\eta}^L\exp[V(y)]
\left(\l\tnm{f_{\l}^{L}(y,\cdot)}^2+\tnm{r_{\l}^{L}(y,\cdot)}^2-\br{S,r_{\l}^{L}}_{\phi}(y)\right)\ud{y},\\
\label{Milne temp 5}
&& \exp[V(\eta)]\alpha(\eta)=\alpha(0)+\int_{0}^{\eta}\exp[V(y)]
\left(-\l\tnm{f_{\l}^{L}(y,\cdot)}^2-\tnm{r_{\l}^{L}(y,\cdot)}^2+\br{S,r_{\l}^{L}}_{\phi}(y)\right)\ud{y}.
\end{eqnarray}
Hence, based on (\ref{Milne temp 4}), we have
\begin{eqnarray}\label{Milne temp 36}
\alpha(\eta)\geq\int_{\eta}^L\exp[V(y)-V(\eta)]\Big(-\br{S,r_{\l}^{L}}_{\phi}(y)\Big)\ud{y}.
\end{eqnarray}
Also, (\ref{Milne temp 5}) implies
\begin{eqnarray}
\alpha(\eta)&\leq&\alpha(0)\exp[-V(\eta)]+\int_{0}^{\eta}\exp[V(\eta)-V(y)]
\bigg(\br{S,r_{\l}^{L}}_{\phi}(y)\bigg)\ud{y}\\
&\leq&2\tnm{h}^2+\int_{0}^{\eta}\exp [V(\eta)-V(y)]
\bigg(\br{S,r_{\l}^{L}}_{\phi}(y)\bigg)\ud{y}\nonumber,
\end{eqnarray}
due to the fact
\begin{eqnarray}
\alpha(0)=\half\br{\sin\phi
f^{L}_{\l},f^{L}_{\l}}_{\phi}(0)\leq\half\bigg(\int_{\sin\phi>0}h^2(\phi)\sin\phi
\ud{\phi}\bigg)\leq \half\tnm{h}^2.
\end{eqnarray}
Then in (\ref{Milne temp 5}) taking $\eta=L$, from $\alpha(L)=0$, we
have
\begin{eqnarray}\label{Milne temp 6}
\int_{0}^L\exp[V(y)]\tnm{r_{\l}^{L}(y,\cdot)}^2\ud{y}
&\leq&\alpha(0)+\int_{0}^L\exp[V(y)]\br{S,r_{\l}^{L}}_{\phi}(y)\ud{y}\\
&\leq&
\half\tnm{h}^2+\int_{0}^L\exp[V(y)] \br{S,r_{\l}^{L}}_{\phi}(y)\ud{y}\nonumber.
\end{eqnarray}
On the other hand, we can directly estimate as follows:
\begin{eqnarray}\label{Milne temp 7}
\int_{0}^L\exp[V(y)]\tnm{r_{\l}^{L}(y,\cdot)}^2\ud{y}\geq
\int_{0}^L\tnm{r_{\l}^{L}(y,\cdot)}^2\ud{y}.
\end{eqnarray}
Combining (\ref{Milne temp 6}) and (\ref{Milne temp 7}) yields
\begin{eqnarray}\label{Milne temp 33}
\int_{0}^L\tnm{r_{\l}^{L}(\eta,\cdot)}^2\ud{\eta}&\leq&
\tnm{h}^2+4\int_{0}^L \left|\br{S,r_{\l}^{L}}_{\phi}(y)\right|\ud{y}.
\end{eqnarray}
By Cauchy's inequality, we have
\begin{eqnarray}\label{Milne temp 34}
\int_0^L\abs{\br{S,r_{\l}^{L}}_{\phi}(y)}\ud{y}
\leq
\frac{1}{8}\int_{0}^L\tnm{r_{\l}^{L}(\eta,\cdot)}^2\ud{\eta}+2\int_{0}^L\tnm{S(\eta,\cdot)}^2\ud{\eta}.
\end{eqnarray}
Therefore, summarizing (\ref{Milne temp 33}) and (\ref{Milne temp
34}), we deduce (\ref{Milne temp 37}).\\

\underline{\emph{The proof of (\ref{Milne temp 38}):}} Multiplying $\sin\phi$ on
both sides of (\ref{Milne finite problem LT penalty.}) and
integrating over $\phi\in[-\pi,\pi)$ lead to
\begin{eqnarray}
\label{Milne temp 18}
\frac{\ud{}}{\ud{\eta}}\br{\sin^2\phi,f_{\l}^{L}}_{\phi}(\eta)&=&-\l\br{\sin\phi,
f_{\l}^{L}}_{\phi}(\eta)
-\br{\sin\phi,r_{\l}^{L}}_{\phi}(\eta)\\
&& +\half F(\eta)\br{\sin(2\phi),\frac{\p
f_{\l}^{L}}{\p\phi}}_{\phi}(\eta)+\br{\sin\phi,S}_{\phi}(\eta)\nonumber.
\end{eqnarray}
It is nature that
\begin{eqnarray}
 \br{\sin(\phi),q_{\l}^L}_{\phi} = q_{\l}^L \br{\sin(\phi),1}=0,~~~\br{\cos(2\phi),q_{\l}^L } = q_{\l}^L \br{\cos(2\phi),1}_{\phi}=0.
\end{eqnarray}
We can further integrate by parts as follows:
\begin{eqnarray}
\half F(\eta)\br{\sin(2\phi),\frac{\p
f_{\l}^{L}}{\p\phi}}_{\phi}(\eta)=-F(\eta)\br{\cos(2\phi),f_{\l}^{L}}_{\phi}(\eta)=-F(\eta)\br{\cos(2\phi),r_{\l}^{L}}_{\phi}(\eta),
\end{eqnarray}
to obtain
\begin{eqnarray}
\frac{\ud{}}{\ud{\eta}}\br{\sin^2\phi,f_{\l}^{L}}_{\phi}(\eta)&=&-\l\br{\sin\phi,
r_{\l}^{L}}_{\phi}(\eta)
-\br{\sin\phi,r_{\l}^{L}}_{\phi}(\eta)\\
&&-F(\eta)\br{\cos(2\phi),r_{\l}^{L}}_{\phi}(\eta)+\br{\sin\phi,S}_{\phi}(\eta).\nonumber
\end{eqnarray}
Define
\begin{eqnarray}\label{Milne temp 81}
\beta_{\l}^{L}(\eta)=\br{\sin^2\phi,f_{\l}^{L}}_{\phi}(\eta)
\end{eqnarray}
and
\begin{eqnarray}\label{Milne temp 82}
D^{L}_{\l}(\eta,\phi)=-\l\br{\sin\phi,
f_{\l}^{L}}_{\phi}(\eta)-\br{\sin\phi,r_{\l}^{L}}_{\phi}(\eta)-F(\eta)\br{\cos(2\phi),r_{\l}^{L}}_{\phi}(\eta)+\br{\sin\phi,S}_{\phi}(\eta).
\end{eqnarray}
Since
\begin{eqnarray}
-\l\br{\sin\phi, f_{\l}^{L}}_{\phi}(\eta)&=&-\l\br{\sin\phi,
r_{\l}^{L}}_{\phi}(\eta)-\l\br{\sin\phi,
q_{\l}^{L}}_{\phi}(\eta)=-\l\br{\sin\phi, r_{\l}^{L}}_{\phi}(\eta).
\end{eqnarray}
we can further get
\begin{eqnarray}
D^{L}_{\l}(\eta)=-\l\br{\sin\phi,
r_{\l}^{L}}_{\phi}(\eta)-\br{\sin\phi,r_{\l}^{L}}_{\phi}(\eta)-F(\eta)\br{\cos(2\phi),r_{\l}^{L}}_{\phi}(\eta)+\br{\sin\phi,S}_{\phi}(\eta).
\end{eqnarray}
Then we can simplify (\ref{Milne temp 18}) as follows:
\begin{eqnarray}\label{Milne temp 35}
\frac{\ud{\beta^{L}_{\l}}}{\ud{\eta}}=D^{L}_{\l}(\eta,\phi),
\end{eqnarray}
We can integrate over $[0,\eta]$ in (\ref{Milne temp 35}) to obtain
\begin{eqnarray}\label{Milne t 03}
\beta^{L}_{\l}(\eta)=\beta^{L}_{\l}(0)+\int_0^{\eta}D^{L}_{\l}(y)\ud{y}.
\end{eqnarray}
On the hand, one has
\begin{eqnarray}
\beta^{L}_{\l}(0)&=&\br{\sin^2\phi,f_{\l}^{L}}_{\phi}(0)\leq
\Big(\br{
f_{\l}^{L},f_{\l}^{L}\abs{\sin\phi}}_{\phi}(0)\Big)^{1/2}||\sin\phi||_{L^3}^{3/2}\\
&\leq&
\frac{16}{3}\Big(\br{
f_{\l}^{L},f_{\l}^{L}\abs{\sin\phi}}_{\phi}(0)\Big)^{1/2}.\nonumber
\end{eqnarray}
Obviously, we have
\begin{eqnarray}
\br{f_{\l}^{L}, f_{\l}^{L}\abs{\sin\phi}
}_{\phi}(0)=\int_{\sin\phi>0}h^2(\phi)\sin\phi
\ud{\phi}-\int_{\sin\phi<0}|f_{\l}^{L}(0,\phi)|^2\sin\phi\ud{\phi}.
\end{eqnarray}
However, based on the definition of $\alpha(\eta)$ and (\ref{Milne
temp 36}), we can obtain
\begin{eqnarray}
2\alpha(0)&=&\int_{\sin\phi>0} h^2(\phi)\sin\phi\ud{\phi}+\int_{\sin\phi<0}
|f_{\l}^{L}(0,\phi)|^2\sin\phi\ud{\phi}\\
&\geq&
2\int_{0}^L\exp[V(y)] \Big(-\br{S,r_{\l}^{L}}_{\phi}(y)\Big)\ud{y}\nonumber\\
&\geq&-\half\int_0^L|\br{S,r_{\l}^{L}}_{\phi}(y)|\ud{y}\nonumber.
\end{eqnarray}
Hence, we can deduce
\begin{eqnarray}
\br{f_{\l}^{L}, f_{\l}^{L}\abs{\sin\phi}
}_{\phi}(0) &\leq& 2 \int_{\sin\phi>0}
h^2(\phi)\sin\phi\ud{\phi}+\half\int_{0}^L |\br{S,r_{\l}^{L}}_{\phi}(y)|\ud{y}\\
&\leq&
2 \tnm{h}^2+\frac{1}{4}\int_{0}^L\tnm{r_{\l}^{L}(\eta,\cdot)}^2\ud{\eta}
+\frac{1}{4}\int_{0}^L\tnm{S(\eta,\cdot)}^2\ud{\eta}\nonumber.
\end{eqnarray}
From (\ref{Milne temp 37}), we can deduce
\begin{eqnarray}\label{Milne t 05}
\beta^{L}_{\l}(0)^2&\leq&
\frac{256}{9}\br{f_{\l}^{L},f_{\l}^{L}\abs{\sin\phi}}_{\phi}(0)\\
&\leq& 128\tnm{h}^2+16\int_{0}^L\tnm{r_{\l}^{L}(\eta,\cdot)}^2\ud{\eta}+16\int_{0}^L\tnm{S(\eta,\cdot)}^2\ud{\eta}\notag\\
&\leq&196\tnm{h}^2+196\int_{0}^L\tnm{S(\eta,\cdot)}^2\ud{\eta}\nonumber.
\end{eqnarray}
On the other hand, since $D^{L}_{\l}$ depends on
$r_{\l}^{L}$ and is independent of $q_{\l}^{L}$, we can
directly estimate
\begin{eqnarray}\label{Milne t 04}
\abs{D^{L}_{\l}(\eta)}&\leq&
2\pi\bigg(1+\l+\abs{F(\eta)}\bigg)\tnm{r_{\l}^{L}(\eta,\cdot)}+ \tnm{S(\eta,\cdot)}.
%\leq  4\pi(1+\l)\tnm{r_{\l}^{L}(\eta)}+\tnm{S(\eta,\cdot)}.
\end{eqnarray}
From (\ref{Milne temp 37}), (\ref{Milne t 03}), (\ref{Milne t 04})
and (\ref{Milne t 05}), we have
\begin{eqnarray}
\abs{\beta^{L}_{\l}(\eta)}&\leq & 14(\tnm{h}+\tnnm{S}) +
2\pi(1+\l+\abs{F(\eta)})\int_0^{\eta}\tnm{r_{\l}^{L}(y,\cdot)}\ud{y}+\int_0^{\eta}\tnm{S(y,\cdot)}\ud{y} \\
&\leq&
14(\tnm{h}+\tnnm{S})  +
2\pi(1+\l+\abs{F(\eta)})\eta^{1/2}\left(\int_0^{\eta}\tnm{r^{L}_{\l}(y,\cdot)}^2\ud{y}\right)^{1/2}
\nonumber\\
&& +\eta^{1/2} \left(\int_0^{\eta}\tnm{S(y,\cdot)}^2\ud{y}\right)^{1/2}\nonumber\\
&\leq& 16 \pi(1+\l+\abs{F(\eta)})(1+\eta^{1/2})(\tnm{h}+\tnnm{S}).\nonumber\end{eqnarray}
By (\ref{Milne temp 81}) this implies
\begin{eqnarray}
\abs{q_{\l}^{L}(\eta)} \leq
\frac{1}{\pi}\left(\abs{\beta^{L}_{\l}(\eta)}+ \tnm{\sin^2(\cdot)} \tnm{r_{\l}^{L}(\eta,\cdot)}\right),
\end{eqnarray}
which completes the proof of (\ref{Milne temp 38}).\\
\ \\
Step 3: {\emph{ Passing to the limit $\l \rightarrow 0$}}\\
Since estimates (\ref{Milne temp 37}) and (\ref{Milne temp 38}) are
uniform in $\l$, we can take weakly convergent subsequence
$f_{\l}^{L}\rt f^{L}\in L^2([0,L]\times[-\pi,\pi))$ as $\l\rt0$.
That is, there is a function $f^{L} \in L^2([0,L]\times (-\pi,\pi])$, which is the solution of
(\ref{Milne finite problem LT}) and satisfies the estimates (\ref{Milne temp 1}) and
(\ref{Milne temp 2}).\\
\ \\
Step 4: {\emph{Orthogonality relation (\ref{Milne temp 3})}}.\\
A direct integration over $\phi\in[-\pi,\pi)$ in (\ref{Milne finite
problem LT}) implies
\begin{eqnarray}
\frac{\ud{}}{\ud{\eta}}\br{\sin\phi,f^{L}}_{\phi}(\eta)=F\br{\cos\phi,\frac{\ud{f^{L}}}{\ud{\phi}}}_{\phi}(\eta)
+\bar S(\eta)=F\br{\sin\phi,f^{L}}_{\phi}(\eta).
\end{eqnarray}
thanks to $\bar S=0$. The specular reflexive boundary
$f^{L}(L,\phi)=f^{L}(L,R\phi)$ implies
$\br{\sin\phi,f^{L}}_{\phi}(L)=0$. Then we have
\begin{eqnarray}
\br{\sin\phi,f^{L}}_{\phi}(\eta)=0.
\end{eqnarray}
It is easy to see
\begin{eqnarray}
\br{\sin\phi,q^{L}}_{\phi}(\eta)=0.
\end{eqnarray}
Hence, we may derive
\begin{eqnarray}
\br{\sin\phi,r^{L}}_{\phi}(\eta)=0.
\end{eqnarray}
This leads (\ref{Milne temp 3}) and completes the proof of
(\ref{Milne finite LT})
\end{proof}

%%%%%%%%%%%%%%%%%%%%%%%%%%%%%%%%%%%%%%%%%%%%%%%%%%%%%%%%%%%%%%%%%%%%%%%%
\subsubsection{Infinite Slab with $\bar S=0$}
%%%%%%%%%%%%%%%%%%%%%%%%%%%%%%%%%%%%%%%%%%%%%%%%%%%%%%%%%%%%%%%%%%%%%%%%

We turn to the $\e$-Milne problem in the infinite
slab, that is, we will consider the following problem
\begin{eqnarray}\label{Milne infinite problem LT}
\left\{
\begin{array}{rcll}\displaystyle
\sin\phi\frac{\p f}{\p\eta}-F(\eta)\cos\phi\frac{\p
f}{\p\phi}+f-\bar f&=&S(\eta,\phi),&(\eta,\phi)\in[0,\infty)\times[-\pi,\pi)\\
f(0,\phi)&=&h(\phi)\ \ & \text{for}\ \ \sin\phi>0,\\
\lim_{\eta\rt\infty}f(\eta,\phi)&=&f_{\infty},
\end{array}
\right.
\end{eqnarray}
where $h$ and $S$ satisfy the assumption (\ref{Milne data}).\\

For simplicity, we denote the kinetic part $r$ and the fluid part $q$ for $f$ as well as $r^{L}$ and $q^{L}$ for $f^{L}$.
\begin{lemma}\label{Milne infinite LT}
Assume $\bar S(\eta)=0$ for any $\eta\in[0,\infty)$ with (\ref{Milne
bounded}) and (\ref{Milne decay}). Then there exists a solution
$f(\eta,\phi)$ of the infinite slab problem (\ref{Milne infinite
problem LT}), satisfying
\begin{eqnarray}
\tnnm{r}&\leq&C\bigg(M+\frac{M}{K}\bigg) <\infty\label{Milne temp 8},\\
\br{\sin\phi,r}_{\phi}(\eta)&=&0 ~~\text{for~ any }~\eta\in[0,\infty), \label{Milne temp 19}\\
\abs{q(\eta)}&\leq&C\bigg(1+M+\frac{M}{K}+ \tnm{r(\eta,\cdot)}\bigg).\label{Milne temp 17.}
\end{eqnarray}
 Also, there exists a constant
$q_{\infty}=f_{\infty}\in\r$ such that the following estimates hold,
\begin{eqnarray}
&& \abs{q_{\infty}}\leq C\bigg(1+M+\frac{M}{K}\bigg)<\infty\label{Milne temp 17},\\
&&|q(\eta)-q_{\infty}| \leq C\bigg(\tnm{r(\eta,\cdot)}+\int_{\eta}^{\infty}\abs{F(y)}\tnm{r(y,\cdot)}\ud{y}+\int_{\eta}^{\infty}\lnm{S(y,\cdot)}\ud{y}
\bigg)\label{Milne temp 9},\\
&& \int_0^{\infty}|q(\eta)-q_{\infty}|^2\ud{\eta} \leq C\bigg(M+\frac{M}{K}\bigg)^2<\infty\label{Milne
temp 10}.
\end{eqnarray}
The solution is unique among functions such that (\ref{Milne temp
8}), (\ref{Milne temp 17})and (\ref{Milne temp 10}) hold.
\end{lemma}
\begin{proof} The existence of the solution is obtained by $L \rightarrow \infty$, the estimates (\ref{Milne temp 8})-(\ref{Milne temp 10}) follow from the equation (\ref{Milne infinite problem LT}) immediately.\\
\ \\
Step 1: {\emph{Existence of the solution and estimates (\ref{Milne temp 8}), (\ref{Milne temp 19}) and (\ref{Milne temp 17.})}}\\
By the estimates from Lemma \ref{Milne infinite LT}, the
solution $f^{L}$ of the (\ref{Milne finite problem
LT}) is bounded in $L^2_{loc}([0,\infty);L^2[-\pi,\pi))$. Then there exists a
subsequence, which is also denoted as $f^L$,  such that
\begin{eqnarray}
q^{L} \rightharpoonup q, ~~~~r^{L} \rightharpoonup r,
~~\text{weakly ~in}~ L^2_{loc}([0,\infty);L^2[-\pi,\pi)).
\end{eqnarray}
The limit function $f=q+r$
satisfies the equation and the boundary condition at $\eta=0$ in the weak sense. This shows the
existence of the solution.

Then property (\ref{Milne temp 8})
naturally holds due to the weak lower semi-continuity of norm
$\tnnm{\cdot}$.
The orthogonal relation (\ref{Milne temp 19}) is also preserved.

For the estimate of (\ref{Milne temp 17.}), we need the  facts that $F\in L^1[0,\infty)\cap L^2[0,\infty)$, $r\in
L^2([0,\infty)\times[-\pi,\pi))$ and $S$ exponentially decays at the far field, corresponding to
(\ref{Milne t 05}), (\ref{Milne temp 8}) and (\ref{Milne decay}).
We use the notation in Step 5 of the proof of Lemma
\ref{Milne finite LT}. Recall (\ref{Milne temp 81}) to (\ref{Milne finite problem LT penalty.}) with $\l=0$ and $L=\infty$, we have
\begin{eqnarray}\label{Milne t 21}
\frac{\ud{\beta}}{\ud{\eta}}=D(\eta,\phi),
\end{eqnarray}
where
\begin{eqnarray}
\beta(\eta)=\br{\sin^2\phi,f}_{\phi}(\eta)
\end{eqnarray}
and
\begin{eqnarray}
D(\eta,\phi)=-\br{\sin\phi,r}_{\phi}-F(\eta)\br{\cos(2\phi),r}_{\phi}+\br{\sin\phi,S}_{\phi}(\eta).
\end{eqnarray}
The orthogonal relation (\ref{Milne temp 19}) implies
\begin{eqnarray}
D(\eta,\phi)=-F(\eta)\br{\cos(2\phi),r}_{\phi}+\br{\sin\phi,S}_{\phi}(\eta).
\end{eqnarray}
Hence, we can integrate (\ref{Milne t 21}) over $[0,\eta]$ to show
\begin{eqnarray}\label{Expression of beta}
\beta(\eta)-\beta(0)=-\int_0^{\eta}F(y)\br{\cos(2\phi),r}_{\phi}(y)\ud{y}+\int_0^{\eta}\br{\sin\phi,S}_{\phi}(y)\ud{y}.
\end{eqnarray}
Similar to (\ref{Milne t 05}), one has
\begin{eqnarray}
|\beta(0)| \leq 14(\tnm{h}+\tnnm{S} ).
\end{eqnarray}
So, it derives to
\begin{eqnarray} \label{Milne te 35}
|\beta(\eta)|&\leq & |\beta(0)| + \int_0^{\eta}|F(y)\br{\cos(2\phi),r}_{\phi}(y)|\ud{y}+\int_0^{\eta}|\br{\sin\phi,S}_{\phi}(y)|\ud{y}\\
&\leq & 14(\tnm{h}+\tnnm{S} )+ \bigg(\int_0^{\eta}|F(y)|^2\ud{y}\bigg)^{1/2} \bigg(\int_0^{\eta}
\tnm{r(y,\cdot)}^2\ud{y}\bigg)^{1/2} \nonumber\\
&& + \pi \bigg(\int_0^{\eta}
\tnm{S(y,\cdot)}\ud{y}\bigg)^{1/2}. \nonumber
\\
&\leq & C(1+M+\frac{M}{K}).\nonumber
\end{eqnarray}
Note that
\begin{eqnarray} \label{Milne Te 37}
\beta(\eta)=\br{\sin^2\phi,f}_{\phi}(\eta)=\br{\sin^2\phi,q}_{\phi}(\eta)+\br{\sin^2\phi,r}_{\phi}(\eta)
=q(\eta)\tnm{\sin\phi}^2+\br{\sin^2\phi,r}_{\phi}(\eta).
\end{eqnarray}
The inequality (\ref{Milne temp 17.}) is valid from (\ref{Milne temp 8}) and the fact (\ref{Milne decay}).\\
\ \\
Step 2: {\emph{ Estimates (\ref{Milne temp 17}), (\ref{Milne temp 9}) and (\ref{Milne temp 10})}}
From (\ref{Expression of beta}) together with  the properties of $F$ and $S$, the limit  of $\beta(\eta)$ exists. Set $\beta_{\infty}=\lim_{\eta\rt\infty}\beta(\eta)$, from (\ref{Milne te 35}), we know that
\begin{eqnarray}\label{Milne te 36}
 |\beta_{\infty}| \leq C(1+M+\frac{M}{K}).
\end{eqnarray}
Define the constant as $q_{\infty}=\beta_{\infty}/\tnm{\sin\phi}^2=\beta_{\infty}/\pi$, then (\ref{Milne temp 17}) follows directly from (\ref{Milne te 36}).
%\begin{eqnarray}
%\abs{\beta_{\infty}}&\leq&\abs{\beta(0)}\abs{-\int_0^{\infty}F(y)\br{\cos(2\phi),r}_{\phi}(y)\ud{y}}
%\abs{\int_0^{\infty}\br{\sin\phi,S}_{\phi}(y)\ud{y}}\\
%&\leq&8+192\tnm{h}^2+200\int_{0}^{\infty}\tnm{S(\eta)}^2\ud{\eta}+2\pi\tnnm{F}
%\tnnm{r}\leq C\bigg(1+M+\frac{M}{K}\bigg)^2\nonumber.
%\end{eqnarray}
 Moreover,
\begin{eqnarray}
\beta_{\infty}-\beta(\eta)=\int_{\eta}^{\infty}D(y)\ud{y}=\int_{\eta}^{\infty}F(y)\br{\cos(2\phi),r}_{\phi}(y)\ud{y}+
\int_{\eta}^{\infty}\br{\sin\phi,S}_{\phi}(y)\ud{y}.
\end{eqnarray}
Thus, (\ref{Milne Te 37}) yields
\begin{eqnarray}\label{Milne t 22}
&& \pi |q(\eta)-q_{\infty}|
%&=&\sqrt{2\pi}\tnm{\sin\phi}^2\lnm{q(\eta)-q_{\infty}}
= \abs{\beta(\eta)-\beta_{\infty}-\br{\sin^2\phi,r}_{\phi}(\eta)}\\
&&\hspace{5mm}\leq
\abs{\br{\sin^2\phi,r}_{\phi}(\eta)}+\int_{\eta}^{\infty}\abs{F(y)\br{\cos(2\phi),r}_{\phi}(y)}\ud{y}
+\int_{\eta}^{\infty}\abs{\br{\sin\phi,S}_{\phi}(y)}\ud{y}\nonumber\\
&& \hspace{5mm}\leq \pi \tnm{r(\eta,\cdot)}+\sqrt{\pi}\int_{\eta}^{\infty}\abs{F(y)}\tnm{r(y,\cdot)}\ud{y}+
\sqrt{\pi}\int_{\eta}^{\infty}\tnm{S(y,\cdot)}\ud{y}\nonumber.
\end{eqnarray}
This implies (\ref{Milne temp 9}). Furthermore, we integrate
(\ref{Milne t 22}) over $\eta\in[0,\infty)$. The Cauchy's inequality implies
\begin{eqnarray}
&&\int_0^{\infty}\bigg(\int_{\eta}^{\infty}\abs{F(y)}\tnm{r(y,\cdot)}\ud{y}\bigg)^2\ud{\eta}\leq
\tnnm{r}^2\int_0^{\infty}\int_{\eta}^{\infty}\abs{F(y)}^2\ud{y}\ud{\eta}\leq
C.
\end{eqnarray}
The exponential decay of $S$ shows that
\begin{eqnarray}
\int_0^{\infty}\bigg(\int_{\eta}^{\infty}\tnm{S(y,\cdot)}\ud{y}\bigg)^2\ud{\eta}\leq
C.
\end{eqnarray}
Hence, the estimate (\ref{Milne temp
10}) naturally follows.\\
\ \\
Step 3: {\emph{ Uniqueness}}\\
In order to show the uniqueness of the solution, we assume there are
two solutions $f_1$ and $f_2$ to the equation (\ref{Milne infinite
problem LT}) satisfying (\ref{Milne temp 8}) and (\ref{Milne temp
19}). Then $f'=f_1-f_2$ satisfies the equation
\begin{eqnarray}\label{uniqueness equation}
\left\{
\begin{array}{rcl}\displaystyle
\sin\phi\frac{\p f'}{\p\eta}-F(\eta)\cos\phi\frac{\p
f'}{\p\phi}+f'-\bar f'&=&0,\\
f'(0,\phi)&=&0\ \ \text{for}\ \ \sin\phi>0,\\
\lim_{\eta\rt\infty}f'(\eta,\phi)&=&f'_{\infty}.
\end{array}
\right.
\end{eqnarray}
Similarly, we can define $r'$ and $q'$. Multiplying
$\ue^{V(\eta)}f'$ on both sides of (\ref{uniqueness equation}) and
integrating over $\phi\in[-\pi,\pi)$ yields
%\begin{eqnarray}\label{uniqueness temp}
%\half\frac{\ud{}}{\ud{\eta}}\bigg(\br{
%f',f'\sin\phi}_{\phi}(\eta)\ue^{V(\eta)}\bigg)=-\bigg(\tnm{r'(\eta)}^2\ue^{V(\eta)}\bigg)\leq0.
%\end{eqnarray}
%This is due to the fact
\begin{eqnarray}\label{uniqueness temp}
&&\half\frac{\ud{}}{\ud{\eta}}\bigg(\br{
f',f'\sin\phi}_{\phi}(\eta)\ue^{V(\eta)}\bigg)\\
&=&\bigg(\br{
f',\frac{\ud{f'}}{\ud{\eta}}\sin\phi}_{\phi}(\eta)\ue^{V(\eta)}\bigg)-\half\bigg(F(\eta)\br{
f',f'\sin\phi}_{\phi}(\eta)\ue^{V(\eta)}\bigg)\nonumber\\
&=&\bigg(\br{
f',\frac{\ud{f'}}{\ud{\eta}}\sin\phi}_{\phi}(\eta)\ue^{V(\eta)}\bigg)-\bigg(F(\eta)\br{
f',\frac{\ud{f'}}{\ud{\phi}}\cos\phi}_{\phi}(\eta)\ue^{V(\eta)}\bigg)\nonumber\\
&=&   \br{
f'-\bar{f'}, f'}_{\phi}(\eta)\ue^{V(\eta)}= -\bigg(\tnm{r'(\eta,\cdot)}^2\ue^{V(\eta)}\bigg)\leq0. \nonumber.
\end{eqnarray}
This implies that $
%\begin{eqnarray}\label{Milne t 06}
\gamma(\eta)=\half\br{f',f'\sin\phi}_{\phi}(\eta)\ue^{V(\eta)}
$%\end{eqnarray}
is decreasing. Since $r'\in L^2([0,\infty)\times[-\pi,\pi))$ and
$q'-q'_{\infty}\in L^2([0,\infty)\times[-\pi,\pi))$, there exists a
%convergent
subsequence $\eta_k $ such that $\tnm{r'(\eta_k,\cdot)}\rt0$ and $q'(\eta_k)-q'_{\infty}\rt0$ as $k \rt \infty$.
Hence, this implies
\begin{eqnarray}
\half\br{r',r'\sin\phi}_{\phi}(\eta_k)\ue^{V(\eta_k)}\rt0,~~~\quad k \rt \infty.
\end{eqnarray}
Also, due to the fact that $q'(\eta_k)$ is independent of $\phi$ and
it is finite dimension with respect to $\phi$, we have
\begin{eqnarray}
\gamma(\eta_k)\rt 0,~~\quad~~k \rt \infty.
\end{eqnarray}
By the monotonicity, $\gamma(\eta)$ decreases to zero and
$\gamma(\eta)\geq0$. Then we can integrate (\ref{uniqueness temp})
over $\eta\in[0,\infty)$ to obtain
\begin{eqnarray}
\gamma(\infty)-\gamma(0)=-2\int_0^{\infty}\tnm{r'(y,\cdot)}^2\ue^{V(y)}\ud{y},
\end{eqnarray}
which implies
\begin{eqnarray}
\gamma(0)=\br{f',f'\sin\phi}_{\phi}(0)\ue^{V(0)} = 2\int_0^{\infty}\tnm{r'(y,\cdot)}^2e^{V(y)}\ud{y}.
\end{eqnarray}
Also, we know
\begin{eqnarray}
0\leq\half\br{f',f'\sin\phi}_{\phi}(0)\ue^{-V(0)} = \half\br{
f',f'\sin\phi}_{\phi}(0)\leq\int_{\sin\phi>0}
(f'(0,\phi))^2\sin\phi\ud{\phi}=0.
\end{eqnarray}
Naturally, we have
\begin{eqnarray}
\br{f',f'\sin\phi}_{\phi}(0)\ue^{V(0)}=2\int_0^{\infty}\tnm{r'(y,\cdot)}^2\ue^{V(y)}\ud{y}=0.
\end{eqnarray}
Hence, we have $r'=0$ and $f'(0,\phi)=0$. Thus, $f'(\eta,\phi)=q'(\eta)$. Plugging this
into the equation (\ref{uniqueness equation}) reveals $\px q'=0$.
Therefore, $f'(\eta,\phi)=C$ for all $(\eta,\phi)\in[0,\infty)\times [-\pi,\pi]$. Naturally the
boundary data leads to $C=0$, which derives to $f'=0$. That is, $f_1=f_2$ and the uniqueness of the solution to
(\ref{uniqueness equation}) follows directly.
\end{proof}

%%%%%%%%%%%%%%%%%%%%%%%%%%%%%%%%%%%%%%%%%%%%%%%%%%%%%%%%%%%%%%%%%%%%%%%%
\subsubsection{$\bar S\neq0$ Case}
%%%%%%%%%%%%%%%%%%%%%%%%%%%%%%%%%%%%%%%%%%%%%%%%%%%%%%%%%%%%%%%%%%%%%%%%

Consider the $\e$-Milne problem for $f(\eta,\phi)$ in
$(\eta,\phi)\in[0,\infty)\times[-\pi,\pi)$ with a general source
term
\begin{eqnarray}\label{Milne remark problem}
\left\{
\begin{array}{rcl}\displaystyle
\sin\phi\frac{\p f}{\p\eta}-F(\eta)\cos\phi\frac{\p
f}{\p\phi}+f-\bar f&=&S(\eta,\phi),\\
f(0,\phi)&=&h(\phi)\ \ \text{for}\ \ \sin\phi>0,\\
\lim_{\eta\rt\infty}f(\eta,\phi)&=&f_{\infty}.
\end{array}
\right.
\end{eqnarray}
\begin{lemma}\label{Milne infinite LT general}
Assume (\ref{Milne bounded}) and (\ref{Milne decay}) hold. Then
there exists a unique solution $f(\eta,\phi)$ of the problem (\ref{Milne
remark problem}), satisfying
\begin{eqnarray}
\tnnm{r}&<&C\bigg(1+M+\frac{M}{K}\bigg)\leq\infty\label{Milne temp 39},\\
\br{\sin\phi,r}_{\phi}(\eta)&=&-\int_{\eta}^{\infty}\ue^{V(y)-V(\eta)}\bar
S(y)\ud{y},\label{Milne temp 40}\\
\abs{q(\eta)}&\leq&C\bigg(1+M+\frac{M}{K}+\tnm{r(\eta,\cdot)}\bigg)\label{Milne temp 41.}.
\end{eqnarray}
Also there exists a constant $q_{\infty}=f_{\infty}\in\r$ such that
the following estimates hold,
\begin{eqnarray}
\abs{q_{\infty}}&\leq&C\bigg(1+M+\frac{M}{K}\bigg)<\infty\label{Milne temp 41},\\
|q(\eta)-q_{\infty}| &\leq&C\bigg(\tnm{r(\eta,\cdot)}+\int_{\eta}^{\infty}\abs{F(y)}\tnm{r(y,\cdot)}\ud{y}+\int_{\eta}^{\infty}\tnm{S(y,\cdot)}\ud{y}
\bigg)\label{Milne temp 42},\\
\tnm{q(\cdot)-q_{\infty}}&\leq&C\bigg(1+M+\frac{M}{K}\bigg)<\infty\label{Milne
temp 43}.
\end{eqnarray}
The solution is unique among functions satisfying
$\tnnm{f-f_{\infty}}<\infty$.
\end{lemma}
\begin{proof}
We can apply superposition property for this linear problem, i.e.
write $S=\bar S+(S-\bar S)=S_Q+S_R$. Then we solve the problem by
the following steps. %For simplicity, we just call the estimates
%(\ref{Milne temp 39}),
%(\ref{Milne temp 41}), (\ref{Milne temp 42}) and (\ref{Milne temp 43}) as the $L^2$ estimates.
\\
\ \\
Step 1: {\emph{ Construction of auxiliary function $f^1$}}\\
For the zero mean part $S_R$, we choose $f^1$ as the solution to
\begin{eqnarray}
\left\{
\begin{array}{rcl}\displaystyle
\sin\phi\frac{\p f^1}{\p\eta}-F(\eta)\cos\phi\frac{\p
f^1}{\p\phi}+f^1-\bar f^1&=&S_R(\eta,\phi),\\
f^1(0,\phi)&=&h(\phi)\ \ \text{for}\ \ \sin\phi>0,\\
\lim_{\eta\rt\infty}f^1(\eta,\phi)&=&f_{\infty}^1.
\end{array}
\right.
\end{eqnarray}
Since $\bar S_R=0$, by Lemma \ref{Milne infinite LT}, we know there
exists a unique solution $f^1$
satisfying the $L^2$ estimates (\ref{Milne temp 39}),
(\ref{Milne temp 41}), (\ref{Milne temp 42}) and (\ref{Milne temp 43}).\\
\ \\
Step 2: {\emph{Construction of auxiliary function $f^2$}}\\
For the part $S_Q$, We seek a function $f^{2}$ satisfying
\begin{eqnarray}\label{Milne temp 83}
-\frac{1}{2\pi}\int_{-\pi}^{\pi}\bigg(\sin\phi\frac{\p
f^{2}}{\p\eta}-F(\eta)\cos\phi\frac{\p
f^{2}}{\p\phi}\bigg)\ud{\phi}+S_Q=0.
\end{eqnarray}
%The following analysis shows this type of function can always be found.
An integration by parts transforms the equation (\ref{Milne
temp 83}) into
\begin{eqnarray}\label{Milne t 07}
-\int_{-\pi}^{\pi}\sin\phi\frac{\p
f^{2}}{\p\eta}\ud{\phi}+\int_{-\pi}^{\pi}F(\eta)\sin\phi
f^{2}\ud{\phi}+2\pi S_Q=0.
\end{eqnarray}
By Setting
\begin{eqnarray}
f^{2}(\phi,\eta)=a(\eta)\sin\phi.
\end{eqnarray}
and plugging this ansatz into (\ref{Milne t 07}), we have
\begin{eqnarray}
-\frac{\ud{a}}{\ud{\eta}}\int_{-\pi}^{\pi}\sin^2\phi\ud{\phi}+F(\eta)a(\eta)\int_{-\pi}^{\pi}\sin^2\phi\ud{\phi}+2\pi
S_Q=0.
\end{eqnarray}
Hence, we have
\begin{eqnarray}
-\frac{\ud{a}}{\ud{\eta}}+F(\eta)a(\eta)+2S_Q=0.
\end{eqnarray}
%This is a first order linear ordinary differential equation, which
%possesses infinitely many solutions.
By assume $a(\infty)=0$, we can directly solve it to
obtain
\begin{eqnarray}
a(\eta)=-\ue^{\int_0^{\eta}F(y)\ud{y}}\int_{\eta}^{\infty}\ue^{-\int_0^yF(z)\ud{z}}2S_Q(y)\ud{y}.
\end{eqnarray}
In particular, for $\eta=0$, we have
\begin{eqnarray}
a(0)=-\int_0^{\infty}\ue^{-\int_0^yF(z)\ud{z}}2S_Q(y)\ud{y}.
\end{eqnarray}
Based on the exponential decay of $S_Q$, we can directly verify
$a(\eta)$ decays exponentially to zero as $\eta\rt\infty$ and $f^2$
satisfies the estimates (\ref{Milne temp 39}),
(\ref{Milne temp 41}), (\ref{Milne temp 42}) and (\ref{Milne temp 43}).\\
\ \\
Step 3: {\emph{Construction of auxiliary function $f^3$}}\\
Since the boundary condition has been changed, we construct $f^3$ verify
\begin{eqnarray}
\left\{
\begin{array}{rcl}\displaystyle
\sin\phi\frac{\p f^3}{\p\eta}-F(\eta)\cos\phi\frac{\p
f^3}{\p\phi}+f^3-\bar f^3&=&-\sin\phi\dfrac{\p
f^{2}}{\p\eta}+F(\eta)\cos\phi\dfrac{\p
f^{2}}{\p\phi}-f^{2}+\bar f^{2}+S_Q,\\
f^3(0,\phi)&=&-a(0)\sin\phi\ \ \text{for}\ \ \sin\phi>0,\\
\lim_{\eta\rt\infty}f^3(\eta,\phi)&=&f_{\infty}^3.
\end{array}
\right.
\end{eqnarray}
Since the source term satisfy
\begin{eqnarray}\label{Milne temp 84}
\int_{-\pi}^{\pi}\bigg(-\sin\phi\frac{\p
f^{2}}{\p\eta}+F(\eta)\cos\phi\frac{\p f^{2}}{\p\phi}-f^{2}+\bar
f^{2}+S_Q\bigg)\ud{\phi}=0,
\end{eqnarray}
we can apply Lemma \ref{Milne infinite LT}
to obtain a unique solution $f^3$
satisfying the estimates (\ref{Milne temp 39}),
(\ref{Milne temp 41}), (\ref{Milne temp 42}) and (\ref{Milne temp 43}).\\
\ \\
Step 4: {\emph{Construction of auxiliary function $f^4$}}\\
We now define $f^4=f^2+f^3$ and an explicit verification shows
\begin{eqnarray}
\left\{
\begin{array}{rcl}\displaystyle
\sin\phi\frac{\p f^4}{\p\eta}-F(\eta)\cos\phi\frac{\p
f^4}{\p\phi}+f^4-\bar f^4&=&S_Q(\eta,\phi),\\
f^4(0,\phi)&=&0\ \ \text{for}\ \ \sin\phi>0,\\
\lim_{\eta\rt\infty}f^4(\eta,\phi)&=&f_{\infty}^4,
\end{array}
\right.
\end{eqnarray}
and $f^4$
satisfies the $L^2$ estimates (\ref{Milne temp 39}),
(\ref{Milne temp 41}), (\ref{Milne temp 42}) and (\ref{Milne temp 43}).\\
\ \\
In summary, we deduce that $f^1+f^4$ is the solution of (\ref{Milne
remark problem}) and satisfies the estimates (\ref{Milne temp 39}),
(\ref{Milne temp 41}), (\ref{Milne temp 42}) and (\ref{Milne temp 43}). A direct
computation of $\br{\sin\phi,f^i}_{\phi}(\eta)$ for $i=1,2,3,4$
leads to (\ref{Milne temp 40}). From $\tnnm{f-f_{\infty}}<\infty$ ,
we deduce $\tnm{\bar f(\cdot)-f_{\infty}}<\infty$. Set $f_{\infty}=f^1_{\infty}+f^4_{\infty}$, a similar argument
in Lemma \ref{Milne infinite LT} shows that the uniqueness of solution.
\end{proof}
Combining all above, we have the following theorem.
\begin{theorem}\label{Milne lemma 6}
Assume (\ref{Milne bounded}) and (\ref{Milne decay}) hold, there exists a
unique solution $f(\eta,\phi)$ for the $\e$-Milne problem (\ref{Milne problem}), which satisfies the estimates
\begin{eqnarray}
\tnnm{f-f_{\infty}}\leq C\bigg(1+M+\frac{M}{K}\bigg)<\infty,
\end{eqnarray}
for some real number $f_{\infty}$ such that
\begin{eqnarray}
\abs{f_{\infty}}\leq C\bigg(1+M+\frac{M}{K}\bigg)<\infty.
\end{eqnarray}
\end{theorem}

%%%%%%%%%%%%%%%%%%%%%%%%%%%%%%%%%%%%%%%%%%%%%%%%%%%%%%%%%%%%%%%%%%%%%%%%
\subsection{$L^{\infty}$ Estimates}
%%%%%%%%%%%%%%%%%%%%%%%%%%%%%%%%%%%%%%%%%%%%%%%%%%%%%%%%%%%%%%%%%%%%%%%%

For the analysis of the $\e$-Milne problem, we need the estimate of $\lnnm{f}$. So we
will consider the $L^{\infty}$  estimates in the following subsection.

%%%%%%%%%%%%%%%%%%%%%%%%%%%%%%%%%%%%%%%%%%%%%%%%%%%%%%%%%%%%%%%%%%%%%%%%
\subsubsection{Finite Slab}
%%%%%%%%%%%%%%%%%%%%%%%%%%%%%%%%%%%%%%%%%%%%%%%%%%%%%%%%%%%%%%%%%%%%%%%%
%Now, we consider the
%solution $f_{\l}^{L}(\eta,\phi)$ to the penalized $\e$-Milne
%equation
%\begin{eqnarray}\label{Milne finite problem LT penaltyA}
%\left\{
%\begin{array}{rcl}\displaystyle
%\l f_{\l}^{L}+\sin\phi\frac{\p
%f_{\l}^{L}}{\p\eta}-F(\eta)\cos\phi\frac{\p
%f_{\l}^{L}}{\p\phi}+f_{\l}^{L}-\bar f_{\l}^{L}&=&H(\eta,\phi),\\
%f_{\l}^{L}(0,\phi)&=&h(\phi)\ \ \text{for}\ \ \sin\phi>0,\\
%f_{\l}^{L}(L,\phi)&=&f_{\l}^{L}(L,R\phi).
%\end{array}
%\right.
%\end{eqnarray}
%where $\bar f_{\l}^{L}$ is defined as (\ref{Milne average}),
%satisfies
%\begin{eqnarray}\label{Milne t 02}
%\lnnm{f_{\l}^{L}}\leq \frac{1+\l}{\l}\bigg(\lnm{h}+\lnnm{S}\bigg).
%\end{eqnarray}
We firstly consider the penalty $\e$-transport problem in a finite slab $ (\eta,\phi)\in[0,L]\times[-\pi,\pi)$
\begin{eqnarray}\label{Milne finite problem LI}
\left\{
\begin{array}{rcll}\displaystyle
\l f^L_{\l}+ \sin\phi\frac{\p f^L_{\l}}{\p\eta}-F(\eta)\cos\phi\frac{\p
f^L_{\l}}{\p\phi}+f^L_{\l}&=&H(\eta,\phi),\\
f^L_{\l}(0,\phi)&=&h(\phi),\ \ &\text{for}\ \ \sin\phi>0,\\
f^L_{\l}(L,R\phi)&=&f^L_{\l}(L,\phi)
\end{array}
\right.
\end{eqnarray}
with $R \phi=-\phi$. We have the following result.
\begin{lemma}\label{Milne infinite LT generalA}
Assume $\lnnm{H}<\infty$ and $\lnm{h}<\infty$, then there
exists a solution $f_{\l}^L(\eta,\phi)$ to the penalized
$\e$-transport equation (\ref{Milne finite problem LI}) satisfying
\begin{eqnarray}\label{Milne t 01}
\lnnm{f_{\l}^L}\leq \lnm{h}+ \lnnm{H}.
\end{eqnarray}
\end{lemma}
\begin{proof}
Define the energy as follows:
\begin{eqnarray}\label{Characteristic curves}
E(\eta,\phi)=\cos\phi e^{V(\eta)}.
\end{eqnarray}
In the plane $(\eta,\phi)\in[0,\infty)\times[-\pi,\pi)$, on the
curve $\phi=\phi(\eta)$ with constant energy, we can see
\begin{eqnarray}
\frac{\ud{E}}{\ud{\eta}}=\frac{\p E}{\p\eta}+\frac{\p
E}{\p\phi}\frac{\p\phi}{\p\eta}=-\cos\phi
F(\eta)\ue^{V(\eta)}-\sin\phi \ue^{V(\eta)}\frac{\p\phi}{\p\eta}=0,
\end{eqnarray}
which further implies
\begin{eqnarray}
\frac{\p\phi}{\p\eta}=-\frac{\cos\phi F(\eta)}{\sin\phi}.
\end{eqnarray}
Plugging this into the equation (\ref{Milne finite problem LI}), on
this curve, we deduce
\begin{eqnarray}
\frac{\ud{f}}{\ud{\eta}}=\frac{\p f}{\p\eta}+\frac{\p
f}{\p\phi}\frac{\p\phi}{\p\eta}=\frac{1}{\sin\phi}\bigg(\sin\phi\frac{\p
f}{\p\eta}-\cos\phi F(\eta)\frac{\p f}{\p\phi}\bigg).
\end{eqnarray}
Hence, this curve with constant energy is exactly the
characteristics of the equation (\ref{Milne finite problem LI}).
Along this curve, the equation can be simplified as follows:
\begin{eqnarray}
\l f_{\l}+\sin\phi\frac{\ud{f_{\l}}}{\ud{\eta}}+f_{\l}=H.
\end{eqnarray}
An implicit function $\eta^+(\eta,\phi)$ can be determined through
\begin{eqnarray}
\abs{E(\eta,\phi)}=\ue^{V(\eta^+)}.
\end{eqnarray}
which means $(\eta^+,\phi_0)$ with $\sin\phi_0=0$ is on the same
characteristics as $(\eta,\phi)$. We also define the quantities for
$0\leq\eta^+\leq\eta'\leq\eta$ as follows:
\begin{eqnarray}
\phi'(\phi,\eta,\eta')&=&\cos^{-1}(\cos\phi \ue^{V(\eta)-V(\eta')}),\\
R\phi'(\phi,\eta,\eta')&=&-\cos^{-1}(\cos\phi
\ue^{V(\eta)-V(\eta')})=-\phi'(\phi,\eta,\eta'),
\end{eqnarray}
where the inverse trigonometric function can be defined
single-valued in the domain $[0,\pi)$ and the quantities are always
well-defined due to the monotonicity of $V$. Finally, we denote
\begin{eqnarray}
G_{\eta,\eta'}^{\l}(\phi)=\int_{\eta'}^{\eta}\frac{1+\l}{\sin(\phi'(\phi,\eta,\xi))}\ud{\xi}.
%~\,\quad\,~~~~~~~\text{for}~\sin\phi>0.
%,\\ G_{\eta,\eta'}^{\l,-}(\phi)&=&\int_{\eta'}^{\eta}\frac{1+\l}{\sin(R(\phi'(\phi,\eta,\xi)))}\ud{\xi} %~\,\,~~~\text{for}~\sin\phi<0.
\end{eqnarray}
With these notations, we can define the solution to (\ref{Milne finite problem LI}) along the characteristics
as follows:\\
\ \\
{\bf Case I}. For $\sin\phi>0$ and $\abs{E(\eta,\phi)}\leq 1$,
\begin{eqnarray}\label{Milne t 08.}
f_{\l}^L(\eta,\phi) = h(\phi'(\phi,\eta,0))\exp(-G^{\l}_{\eta,0})
+\int_0^{\eta}\frac{H(\eta',\phi'(\phi,\eta,\eta'))}{\sin(\phi'(\phi,\eta,\eta'))}
\exp(-G^{\l}_{\eta,\eta'})\ud{\eta'}.
\end{eqnarray}
{\bf Case II}. For $\sin\phi>0$ and $\abs{E(\eta,\phi)}\geq 1$,
\begin{eqnarray}\label{Milne t 09.}
f_{\l}^L(\eta,\phi) = f_{\l}^L(\eta^+(\eta,\phi),\phi_0)\exp(-G^{\l}_{\eta,\eta^+})
+\int_{\eta^+}^{\eta}\frac{H(\eta',\phi'(\phi,\eta,\eta'))}{\sin(\phi'(\phi,\eta,\eta'))}
\exp(-G^{\l}_{\eta,\eta'})\ud{\eta'}.
\end{eqnarray}
{\bf Case III}. For $\sin\phi<0$ and $\abs{E(\eta,\phi)}\geq 1$,
\begin{eqnarray}\label{Milne t 10.}
f_{\l}^L(\eta,\phi)=f_{\l}^L(\eta^+(\eta,\phi),\phi_0)\exp(-G^{\l}_{\eta,\eta^+})
+\int_{\eta^+}^{\eta}\frac{H(\eta',\phi'(\phi,\eta,\eta'))}{\sin(\phi'(\phi,\eta,\eta'))}
\exp(-G^{\l}_{\eta,\eta'})\ud{\eta'}.
\end{eqnarray}
{\bf Case IV}. For $\sin\phi<0$ and $\abs{E(\eta,\phi)}\leq 1$,
\begin{eqnarray}\label{Milne t 11.}
f_{\l}^L(\eta,\phi)&=&h(\phi'(\phi,\eta,0))\exp\bigg[-G^{\l}_{L,0}-G^{\l}_{L,\eta}\bigg]\\
& & +\int_0^{L}\frac{H(\eta',R\phi'(\phi,\eta,\eta'))}{\sin(\phi'(\phi,\eta,\eta'))}
\exp\bigg[-G^{\l}_{L,\eta'}-G^{\l}_{L,\eta}\bigg]\ud{\eta'}\nonumber\\
&& +\int_{\eta}^{L}\frac{H(\eta',\phi'(\phi,\eta,\eta'))}{\sin(\phi'(\phi,\eta,\eta'))}\exp(G^{\l}_{\eta,\eta'})
\ud{\eta'}\nonumber.
\end{eqnarray}
In the following, we give the estimate of (\ref{Milne t 01}). In Case I, (\ref{Milne t 08.}) derives to
\begin{eqnarray}
\lnnm{f_{\l}^L}&\leq&\lnm{h}\exp(-G^{\l}_{\eta,0})
+\lnnm{H}\int_0^{\eta}\frac{\exp(-G^{\l}_{\eta,\eta'})}{\sin(\phi'(\phi,\eta,\eta'))}\ud{\eta'}\\
&=&\lnm{h}\exp(-G^{\l}_{\eta,0})
+\lnnm{H}\frac{1}{1+\l}\bigg(1-\exp(-G^{\l}_{\eta,0})\bigg)\no\\
&\leq&\lnm{h}+\lnnm{H}.\no
\end{eqnarray}
\ \\
In Case II and III, %For $\sin\phi>0$ and $\abs{E(\eta,\phi)}\geq 1$,
%\begin{eqnarray}\label{Milne t 09}
%f_{\l}^L(\eta,\phi)&=&f_{\l}^L(\eta^+(\eta,\phi),\phi_0)\exp(-G^{\l}_{\eta,\eta^+})
%+\int_{\eta^+}^{\eta}\frac{H(\eta',\phi'(\phi,\eta,\eta'))}{\sin(\phi'(\phi,\eta,\eta'))}\exp(-G^{\l}_{\eta,\eta'})\ud{\eta'}.
%\end{eqnarray}
the main difficulty is the lack of estimate for
$f_{\l}^L(\eta^+(\eta,\phi),\phi_0)$. But, we denote the points
$(\eta^+,\phi_0)$, $(L, \phi_L)$, and $(L,-\phi_L)$ with
\begin{eqnarray}
\phi_L=\cos^{-1}\bigg(\ue^{V(\eta^+(\eta,\phi))-V(L)}\bigg)>0
\end{eqnarray}
are on the same characteristic line (\ref{Characteristic curves}). Then, along this
characteristic line, we can compute that
\begin{eqnarray}
&& f_{\l}^L(\eta^+,\phi_0)=
f_{\l}^L(L,\phi_L)\exp(G^{\l}_{L,\eta^+}(\phi_L))-\int_{\eta^+}^{L}
\frac{H(\eta',\phi'(\phi_L,\eta,\eta'))}{\sin(\phi'(\phi_L,\eta,\eta'))}\exp(G^{\l}_{\eta',\eta^+}(\phi_L))\ud{\eta'},
\label{mt 01}\\
&& f_{\l}^L(\eta^+,\phi_0)=
f_{\l}^L(L,-\phi_L)\exp(RG^{\l}_{L,\eta^+}(\phi_L))-\int_{\eta^+}^{L}
\frac{H(\eta',R\phi'(\phi_L,\eta,\eta'))}{\sin(\phi'(\phi_L,\eta,\eta'))}\exp(RG^{\l}_{\eta',\eta^+}(\phi_L))\ud{\eta'}.
\end{eqnarray}
Then naturally we have
\begin{eqnarray}
&&f_{\l}^L(L,\phi_L)\exp(G^{\l}_{L,\eta^+}(\phi_L))-\int_{\eta^+}^{L}
\frac{H(\eta',\phi'(\phi_L,L,\eta'))}{\sin(\phi'(\phi_L,L,\eta'))}\exp(G^{\l}_{\eta',\eta^+}(\phi_L))\ud{\eta'}\\
&=& f_{\l}^L(L,-\phi_L)\exp(-G^{\l}_{L,\eta^+}(\phi_L))-\int_{\eta^+}^{L}
\frac{H(\eta',R\phi'(\phi_L,L,\eta'))}{\sin(\phi'(\phi_L,L,\eta'))}\exp(-G^{\l}_{\eta',\eta^+}(\phi_L))\ud{\eta'}.\no
\end{eqnarray}
The specular reflective boundary condition implies
$f_{\l}^L(L,\phi_L)=f_{\l}^L(L,-\phi_L)$. Then we obtain
\begin{eqnarray}
\\
&& f_{\l}^L(L,\phi_L) = \frac{1}{\exp(G^{\l}_{L,\eta^+}(\phi_L))
-\exp(-G^{\l}_{L,\eta^+}(\phi_L))}
\Bigg(\displaystyle\int_{\eta^+}^{L}\dfrac{H(\eta',\phi'(\phi_L,L,\eta'))}{\sin(\phi'(\phi_L,L,\eta'))}
\exp(G^{\l}_{\eta',\eta^+}(\phi_L))\ud{\eta'}
\notag\\
&&\hspace{2cm} -\int_{\eta^+}^{L}\dfrac{H(\eta',R\phi'(\phi_L,L,\eta'))}
{\sin(\phi'(\phi_L,L,\eta'))}\exp(-G^{\l}_{\eta',\eta^+}(\phi_L))\ud{\eta'}\bigg).
\no
\end{eqnarray}
It naturally leads to
\begin{eqnarray}\label{mt 02}
\\
|f_{\l}^L(\eta^+,\phi_0)|\leq\lnnm{H}\frac{\displaystyle\int_{\eta^+}^{L}\dfrac{\exp(-G^{\l}_{L,\eta'}(\phi_L))}
{\sin(\phi'(\phi_L,L,\eta'))}\ud{\eta'}
+\int_{\eta^+}^{L}\dfrac{\exp(G^{\l}_{L,\eta'}(\phi_L))}{\sin(\phi'(\phi_L,L,\eta'))}\ud{\eta'}}
{\exp(G^{\l}_{L,\eta^+}(\phi_L))-\exp(-G^{\l}_{L,\eta^+}(\phi_L))} \leq \lnnm{H}.\no
\end{eqnarray}
%Plugging (\ref{mt 02}) into (\ref{mt 01}), we have
%\begin{eqnarray}
%|f_{\l}^L(\eta^+,\phi_0)|\leq \lnnm{H}.
%\end{eqnarray}
Similar to the estimates in Case I, we have
\begin{eqnarray}
\lnnm{f_{\l}^L}\leq \bigg( \exp(-G^{\l}_{L,\eta^+}(\phi_L))+ \int_{\eta^+}^{L}
\frac{\exp(-G^{\l}_{L,\eta'}(\phi_L))}{\sin(\phi'(\phi_L,L,\eta'))}\ud{\eta'}\bigg) \lnnm{H} \leq \lnnm{H}.
\end{eqnarray}
\ \\
In case IV, it is similar to Case I, we can directly estimate to obtain
\begin{eqnarray}
\lnnm{f_{\l}^L}\leq \lnm{h}+ \lnnm{H}.
\end{eqnarray}
%\ \\
%We can easily get
%\begin{eqnarray}
%\lnnm{f_{\l}^L}\leq \lnnm{H}.
%\end{eqnarray}
Summarizing all above, we complete the proof of (\ref{Milne t 01}).
\end{proof}

%%%%%%%%%%%%%%%%%%%%%%%%%%%%%%%%%%%%%%%%%%%%%%%%%%%%%%%%%%%%%%%%%%%%%%%%
\subsubsection{Infinite Slab}
%%%%%%%%%%%%%%%%%%%%%%%%%%%%%%%%%%%%%%%%%%%%%%%%%%%%%%%%%%%%%%%%%%%%%%%%

Let $L\rightarrow \infty$, we consider the following problem in the infinite slab $\eta \in (0,\infty)$,
\begin{eqnarray}\label{Milne finite problem LT penalty}
\left\{
\begin{array}{rcll}\displaystyle
\l f_{\l}+\sin\phi\frac{\p
f_{\l}}{\p\eta}-F(\eta)\cos\phi\frac{\p
f_{\l}}{\p\phi}+f_{\l} &=&H(\eta,\phi), & \eta>0,\\
f_{\l}(0,\phi)&=&h(\phi) & \text{for}\ \ \sin\phi>0,\\
\lim_{\eta \rightarrow \infty}f_{\l} &=& f_{\infty}.
%f_{\l}^{L}(L,\phi)&=&f_{\l}^{L}(L,R\phi).
\end{array}
\right.
\end{eqnarray}
%where $\bar f_{\l}^{L}$ is defined as (\ref{Milne average}).
The following lemma holds.
\begin{lemma}\label{Milne infinite LT generalB}
$\lnnm{H}<\infty$ and $\lnm{h}<\infty$, then the
solution $f_{\l}(\eta,\phi)$ to the penalized $\e$-Milne
equation (\ref{Milne finite problem LT penalty}) satisfies
\begin{eqnarray}\label{Milne t 02}
\lnnm{f_{\l}}\leq C\bigg(\lnm{h}+\lnnm{H}\bigg)
\end{eqnarray}
where $C$ is a constant independent of $\l$.
\end{lemma}
Now, we use the fact that
\begin{eqnarray}
\lim_{L\rt\infty}\exp(-G^{\l}_{L,\eta})=0~~~~~~\text{for}~~\sin\phi<0.
\end{eqnarray}
It can be defined the solution via taking limit $L\rt\infty$ in
(\ref{Milne t 08.})-(\ref{Milne t 11.}) as
follows:
\begin{eqnarray}
f_{\l}(\eta,\phi)=\a_{\l}[h(\phi)]+\t_{\l}[H(\eta,\phi)],
\end{eqnarray}
where\\
\ \\
{\bf Case I}.  For $\sin\phi>0$ and $\abs{E(\eta,\phi)}\leq 1$,
\begin{eqnarray}\label{Milne temp 12}
\a_{\l}[h(\phi)]&=&h(\phi'(\phi,\eta,0))\exp(-G^{\l}_{\eta,0})\\
\t_{\l}[H(\eta,\phi)]&=&\int_0^{\eta}\frac{H(\eta',\phi'(\phi,\eta,\eta'))}{\sin(\phi'(\phi,\eta,\eta'))}\exp(-G^{\l}_{\eta,\eta'})\ud{\eta'}.
\end{eqnarray}
\ \\
{\bf Case II}.  For $\sin\phi>0$ and $\abs{E(\eta,\phi)}\geq 1$,
\begin{eqnarray}\label{Milne temp 13}
\a_{\l}[h(\phi)]&=&0\\
\t_{\l}[H(\eta,\phi)]&=&f_{\l}(\eta^+,\phi_0)\exp(-G^{\l}_{\eta,\eta^+})
+\int_{\eta^+}^{\eta}\frac{H(\eta',\phi'(\phi,\eta,\eta'))}{\sin(\phi'(\phi,\eta,\eta'))}
\exp(-G^{\l}_{\eta,\eta'})\ud{\eta'},\\
&=&\int_{\eta^+}^{\infty}\frac{H(\eta',R\phi'(\phi,\eta,\eta'))}{\sin(\phi'(\phi,\eta,\eta'))}
\exp(G^{\l}_{\eta,\eta'}-G^{\l}_{\eta,\eta^+})\ud{\eta'}\no\\
&&+\int_{\eta^+}^{\eta}\frac{H(\eta',\phi'(\phi,\eta,\eta'))}{\sin(\phi'(\phi,\eta,\eta'))}
\exp(-G^{\l}_{\eta,\eta'})\ud{\eta'}.\no
\end{eqnarray}
\ \\
{\bf Case III}.
For $\sin\phi<0$,
\begin{eqnarray}\label{Milne temp 14}
\a_{\l}[h(\phi)]&=&0\\
\t_{\l}[H(\eta,\phi)]&=&\int_{\eta}^{\infty}\frac{H(\eta',R\phi'(\phi,\eta,\eta'))}{\sin(\phi'(\phi,\eta,\eta'))}
\exp(G^{\l}_{\eta,\eta'})\ud{\eta'}.
\end{eqnarray}
In Case II, we replace $f_{\l}(\eta^+,\phi_0)$ by the integral along the characteristics in $\sin\phi<0$. Also, the latter two cases are combined into a united one. In order to achieve the estimate of $f_{\l}$, we
need to give several technical lemmas about $\a_{\l}$ and $\t_{\l}$.
From (\ref{Milne temp 12})-(\ref{Milne temp 14}), one can easily obtain
\begin{lemma}\label{Milne lemma 1}
The operator $\a_{\l}$ satisfies
\begin{eqnarray}
%\ltnm{\a[h(\phi),f(\eta^+,\phi_0)]}&\leq&
%2\pi \bigg(\lnm{h}+\ltnm{\t H}\bigg)\label{Milne ttemp 52},\\
\lnnm{\a_{\l}[h(\phi)]}\leq
\lnm{h}.\label{Milne temp 352}
\end{eqnarray}
%with $\sin\phi_0=0$.
\end{lemma}
\begin{proof} Since $\sin(\phi'(\phi,\eta,\xi))>0$ for $0\leq \xi \leq \eta$, we know that
\begin{eqnarray*}
\exp(-G^{\l}_{\eta,0})\leq 1.
\end{eqnarray*}
Then, the estimate of (\ref{Milne temp 352}) holds immediately.
\end{proof}

%It is obviously that
%\begin{eqnarray}
%\lnnm{\a[h(\phi),f(\eta^+,\phi_0)]} \leq \lnm{h}
%\end{eqnarray}
%in Case I and Case III. In Case II, we should give the estimate of $f(\eta^+,\phi_0)$. It could be expressed as
%Note that along the characteristics, the equation is
%\begin{eqnarray*}
%\sin\phi\frac{\ud{f}}{\ud{\eta}}+f=H.
%\end{eqnarray*}
%Then at $(\eta^+,\phi_0)$ for $\sin\phi_0=0$, we have
%$\sin\phi_0\dfrac{\ud{f}}{\ud{\eta}}=0$. So we have
%\begin{eqnarray*}
%f(\eta^+,\phi_0)=H(\eta^+,\phi_0).
%\end{eqnarray*}
%Therefore, the result is obvious.
\begin{lemma}\label{Milne lemma 2}
There exists some constant such that the integral operator $\t_{\l}$ satisfies
\begin{eqnarray}\label{Milne temp 53}
\lnnm{\t_{\l}[H]}\leq C\lnnm{H}~~~\quad\text{for ~all}~\l>0.
\end{eqnarray}
\end{lemma}
\begin{proof} The proof of this lemma will divided into three cases.
\ \\
In Case I, it implies that $\sin\phi>0$ and $\abs{E(\eta,\phi)}\leq 1$. One has
\begin{eqnarray}
\abs{\t_{\l}[H]}&\leq&\int_0^{\eta}\abs{H(\eta',\phi'(\phi,\eta,\eta'))}\frac{1}{\sin(\phi'(\phi,\eta,\eta'))}
\exp(-G^{\l}_{\eta,\eta'})\ud{\eta'}\\
&\leq&\lnnm{H}\int_0^{\eta}\frac{1}{\sin(\phi'(\phi,\eta,\eta'))}\exp(-G^{\l}_{\eta,\eta'})\ud{\eta'}\nonumber.
\end{eqnarray}
We can directly estimate
\begin{eqnarray}\label{Milne t 11}
\int_0^{\eta}\frac{1}{\sin(\phi'(\phi,\eta,\eta'))}\exp(-G^{\l}_{\eta,\eta'})\ud{\eta'}\leq\frac{1}{1+\l}\int_0^{\infty}
\ue^{-z}\ud{z} \leq \frac{1}{1+\l},
\end{eqnarray}
and (\ref{Milne temp 53}) naturally follows.
\ \\
In Case III, it implies that $\sin\phi<0$. So it holds that
\begin{eqnarray}
\abs{\t_{\l}[H]}&\leq&\int_{\eta}^{\infty}\abs{H(\eta',\phi'(\phi,\eta,\eta'))}\frac{1}{\sin(\phi'(\phi,\eta,\eta'))}
\exp(G^{\l}_{\eta,\eta'})\ud{\eta'}\\
&\leq&\lnnm{H}\int_{\eta}^{\infty}\frac{1}{\sin(\phi'(\phi,\eta,\eta'))}\exp(G^{\l}_{\eta,\eta'})\ud{\eta'}\nonumber.
\end{eqnarray}
Then, we have
\begin{eqnarray}
\int_{\eta}^{\infty}\frac{1}{\sin(\phi'(\phi,\eta,\eta'))}\exp(G^{\l}_{\eta,\eta'})\ud{\eta'}\leq
\frac{1}{1+\l}\int_{-\infty}^0e^{z}\ud{z}\leq \frac{1}{1+\l},
\end{eqnarray}
and (\ref{Milne temp 53}) easily follows in this case.
%The case $\sin\phi<0$ and
%$\abs{E(\eta,\phi)}\geq 1$ can be proved in a similar way, so we
%omit it here.
\ \\
In Case II, it satisfies that $\sin\phi>0$. We define that
\begin{eqnarray}
\t_{\l,1}[H]&=& f(\eta^+,\phi_0)\exp(-G^{\l}_{\eta,\eta^+})=\int_{\eta^+}^{\infty}\frac{H(\eta',R\phi'(\phi,\eta,\eta'))}{\sin(\phi'(\phi,\eta,\eta'))}
\exp(G^{\l}_{\eta,\eta'}-G^{\l}_{\eta,\eta^+})\ud{\eta'},\label{Milne T1}\\
\t_{\l,2}[H]&=& \int_{\eta^+}^{\eta}\frac{H(\eta',\phi'(\phi,\eta,\eta'))}{\sin(\phi'(\phi,\eta,\eta'))}
\exp(-G^{\l}_{\eta,\eta'})\ud{\eta'}.\label{Milne T2}
\end{eqnarray}
Since
\begin{eqnarray}
\exp(G^{\l}_{\eta,\eta'}-G^{\l}_{\eta,\eta^+})\leq \exp(-G^{\l}_{\eta,\eta'}),
\end{eqnarray}
then $\t_{\l,1}$ can be treated as in Case III.  Also, the proof of $\t_{\l,2}$ is similar to the Case I, we omit it here. The proof of Lemma \ref{Milne lemma 2} is completed.
\end{proof}

From Lemma \ref{Milne infinite LT generalB}, the bound of $f_{\l}$ in $L^{\infty}$ is independent of $\l$. So, let $\l \rightarrow 0$,
 we get a solution of
\begin{eqnarray} \label{Milne infinite problem1}
\left\{
\begin{array}{rcll}\displaystyle
\sin\phi\frac{\p
f}{\p\eta}-F(\eta)\cos\phi\frac{\p
f}{\p\phi}+f &=&H(\eta,\phi), & \eta>0,\\
f(0,\phi)&=&h(\phi) & \text{for}\ \ \sin\phi>0,\\
\lim_{\eta \rightarrow \infty}f &=& f_{\infty}.
%f_{\l}^{L}(L,\phi)&=&f_{\l}^{L}(L,R\phi).
\end{array}
\right.
\end{eqnarray}
In this case, we denote
\begin{eqnarray}
G_{\eta,\eta'}(\phi)=\int_{\eta'}^{\eta}\frac{1}{\sin(\phi'(\phi,\eta,\xi))}\ud{\xi}.
%~\,\quad\,~~~~~~~\text{for}~\sin\phi>0.
%,\\ G_{\eta,\eta'}^{\l,-}(\phi)&=&\int_{\eta'}^{\eta}\frac{1+\l}{\sin(R(\phi'(\phi,\eta,\xi)))}\ud{\xi} %~\,\,~~~\text{for}~\sin\phi<0.
\end{eqnarray}
Then, the solution $f$ can be rewritten as
\begin{eqnarray}
f(\eta,\phi)=\a[h(\phi)]+\t[H(\eta,\phi)],
\end{eqnarray}
where the operator $\a$ and $\t$ is defined as

{\bf Case I}.  For $\sin\phi>0$ and $\abs{E(\eta,\phi)}\leq 1$,
\begin{eqnarray}\label{Milne temp 32}
\a[h(\phi)]&=&h(\phi'(\phi,\eta,0))\exp(-G_{\eta,0})\\
\t[H(\eta,\phi)]&=&\int_0^{\eta}\frac{H(\eta',\phi'(\phi,\eta,\eta'))}{\sin(\phi'(\phi,\eta,\eta'))}\exp(-G_{\eta,\eta'})\ud{\eta'}.
\end{eqnarray}
\ \\
{\bf Case II}.  For $\sin\phi>0$ and $\abs{E(\eta,\phi)}\geq 1$,
\begin{eqnarray}\label{Milne temp 33}
\a[h(\phi)]&=&0\\
\t[H(\eta,\phi)]&=&f(\eta^+,\phi_0)\exp(-G_{\eta,\eta^+})
+\int_{\eta^+}^{\eta}\frac{H(\eta',\phi'(\phi,\eta,\eta'))}{\sin(\phi'(\phi,\eta,\eta'))}
\exp(-G_{\eta,\eta'})\ud{\eta'},\\
&=&\int_{\eta^+}^{\infty}\frac{H(\eta',R\phi'(\phi,\eta,\eta'))}{\sin(\phi'(\phi,\eta,\eta'))}
\exp(G_{\eta,\eta'}-G_{\eta,\eta^+})\ud{\eta'}\no\\
&&+\int_{\eta^+}^{\eta}\frac{H(\eta',\phi'(\phi,\eta,\eta'))}{\sin(\phi'(\phi,\eta,\eta'))}
\exp(-G_{\eta,\eta'})\ud{\eta'}.\no
\end{eqnarray}
\ \\
{\bf Case III}.
For $\sin\phi<0$,
\begin{eqnarray}\label{Milne temp 34}
\a[h(\phi)]&=&0\\
\t[H(\eta,\phi)]&=&\int_{\eta}^{\infty}\frac{H(\eta',R\phi'(\phi,\eta,\eta'))}{\sin(\phi'(\phi,\eta,\eta'))}
\exp(G_{\eta,\eta'})\ud{\eta'}.
\end{eqnarray}

For the estimate of the solution in $L^{\infty}$, we need the following estimate of $\t[H]$.
\begin{lemma}\label{Milne lemma 3}
For any $\delta>0$ there is a constant $C(\delta)>0$ independent of
data such that
\begin{eqnarray}
\ltnm{\t[H]}\leq C(\delta)\tnnm{H}+\delta\lnnm{H}\label{Milne temp
11}.
\end{eqnarray}
\end{lemma}
\begin{proof}
We divide the proof into several steps.\\
\ \\
Step 1: \emph{The case of $\sin\phi>0$ and $\abs{E(\eta,\phi)}\leq1$}.\\
For simplicity, we denote
$\Omega_1=\{\sin(\phi'(\phi,\eta,\eta'))>m\}$ and $\Omega_2=\{\sin(\phi'(\phi,\eta,\eta'))\leq m\}$. Then, we consider
\begin{eqnarray}
 \int_{\sin\phi>0}\abs{\t
[H(\eta,\cdot)]}^2\ud{\phi}&=&\int_{\sin\phi>0}\bigg(\int_0^{\eta}
\frac{H(\eta',\phi'(\phi,\eta,\eta'))}{\sin(\phi'(\phi,\eta,\eta'))}
\exp(-G_{\eta,\eta'})\ud{\eta'}\bigg)^2\ud{\phi} \\
&=& \int\bigg(\int{\bf{1}}_{\Omega_1}
\ldots\bigg)^2+\int\bigg(\int{\bf{1}}_{\Omega_2}
\ldots\bigg)^2\nonumber\\
& =& I_1+I_2\nonumber
\end{eqnarray}
for some $m>0$. By Cauchy's inequality, we
get
\begin{eqnarray}\label{Milne temp 55}
I_1&\leq&\int_{\sin\phi>0}\bigg(\int_0^{\eta}\abs{H(\eta',\phi'(\phi,\eta,\eta'))}^2\ud{\eta'}\bigg)
\bigg(\int_0^{\eta}{\bf{1}}_{\Omega_1}\frac{\exp(-2G_{\eta,\eta'})}{\sin^2(\phi'(\phi,\eta,\eta'))}
\ud{\eta'}\bigg)\ud{\phi}\\
&\leq&\frac{1}{m}\tnnm{H}^2\int_{\sin\phi>0}
\bigg(\int_0^{\eta}{\bf{1}}_{\Omega_1}\frac{\exp(-2G_{\eta,\eta'})}{\sin(\phi'(\phi,\eta,\eta'))}
\ud{\eta'}\bigg)\ud{\phi}\nonumber\\
&\leq&\frac{\pi}{m}\tnnm{H}^2\nonumber.
\end{eqnarray}
On the other hand, for $\eta'\leq\eta$, we can directly estimate
$\phi'(\phi,\eta,\eta')\geq\phi$. Hence, we have the relation
\begin{eqnarray}
\sin\phi\leq\sin(\phi'(\phi,\eta,\eta')).
\end{eqnarray}
Therefore, we can directly estimate $I_2$ as follows:
\begin{eqnarray}
I_2&\leq&\lnnm{H}^2\int_{\sin\phi>0}
\bigg(\int_0^{\eta}{\bf{1}}_{\Omega_2}
\frac{1}{\sin(\phi'(\phi,\eta,\eta'))}\exp(-G_{\eta,\eta'})\ud{\eta'}\bigg)^2\ud{\phi}\\
&\leq&\lnnm{H}^2\int_{\sin\phi>0}
\bigg(\int_0^{\eta}{\bf{1}}_{\Omega_2}
\frac{1}{\sin(\phi'(\phi,\eta,\eta'))}\exp(-G_{\eta,\eta'})\ud{\eta'}\bigg)^2\ud{\phi}\nonumber\\
&=&\lnnm{H}^2\int_{\Omega_2}
\bigg(\int_0^{\eta}
\frac{1}{\sin(\phi'(\phi,\eta,\eta'))}\exp(-G_{\eta,\eta'})\ud{\eta'}\bigg)^2\ud{\phi}\nonumber.
\end{eqnarray}
Since
\begin{eqnarray}
\int_0^{\eta}
\frac{1}{\sin(\phi'(\phi,\eta,\eta'))}\exp(-G_{\eta,\eta'})\ud{\eta'}\leq\int_{0}^{\infty}\ue^{-z}\ud{z}=1,
\end{eqnarray}
we have, for $m$ sufficiently small,
\begin{eqnarray}\label{Milne temp 56}
I_2&\leq&\lnnm{H}^2\int_{\sin\phi>0}{\bf{1}}_{\{\sin\phi\leq
m\}}\ud{\phi}\leq 4m \lnnm{H}^2.
\end{eqnarray}
Summing up (\ref{Milne temp 55}) and (\ref{Milne temp 56}), we deduce (\ref{Milne temp 11}) for $m$
sufficiently small.\\
\ \\
Step 2: \emph{The case of $\sin\phi<0$}.\\
We can decompose
\begin{eqnarray}
\int_{\sin\phi<0}\abs{\t H}^2\ud{\phi}
&=&\int_{\sin\phi<0}\bigg(\int_{\eta}^{\infty}
\frac{H(\eta',R\phi'(\phi,\eta,\eta'))}{\sin(\phi'(\phi,\eta,\eta'))}
\exp(G_{\eta,\eta'})\ud{\eta'}\bigg)^2\ud{\phi}\\
&=&\int\bigg(\int
{\bf{1}}_{\Omega_1}\ldots\bigg)^2+\int\bigg(\int{\bf{1}}_{\Omega_2} {\bf{1}}_{\{\eta'-\eta\geq\sigma\}}
\ldots\bigg)^2\nonumber\\
&&+\int\bigg(\int{\bf{1}}_{\Omega_2}
{\bf{1}}_{\{\eta'-\eta\leq\sigma\}}
\ldots\bigg)^2\nonumber\\
&=&I_1+I_2+I_3,\nonumber
\end{eqnarray}
for some $m>0$ and $\sigma>0$. We can directly estimate $I_1$ as
follows:
\begin{eqnarray}\label{Milne temp 57}
I_1&\leq&\int_{\sin\phi<0}\bigg(\int_{\eta}^{\infty}\abs{H(\eta',R\phi'(\phi,\eta,\eta'))}^2\ud{\eta'}\bigg)
\bigg(\int_{\eta}^{\infty}{\bf{1}}_{\Omega_1}\frac{\exp(2G_{\eta,\eta'})}{\sin^2(\phi'(\phi,\eta,\eta'))}
\ud{\eta'}\bigg)\ud{\phi}\\
&\leq&\frac{1}{m}\tnnm{H}^2\int_{\sin\phi<0}
\bigg(\int_{\eta}^{\infty}{\bf{1}}_{\Omega_1}\frac{\exp(2G_{\eta,\eta'})}{\sin(\phi'(\phi,\eta,\eta'))}
\ud{\eta'}\bigg)\ud{\phi}\nonumber\\
&\leq&\frac{\pi}{m}\tnnm{H}^2\nonumber.
\end{eqnarray}
On the other hand, for $I_2$ we have
\begin{eqnarray}
I_2 \leq \lnnm{H}^2\int_{\sin\phi<0}
\bigg(\int_{\eta}^{\infty}{\bf{1}}_{\Omega_2}{\bf{1}}_{\{\eta'-\eta\geq\sigma\}}
\frac{\exp(G_{\eta,\eta'})}{\sin(\phi'(\phi,\eta,\eta'))}\ud{\eta'}\bigg)^2\ud{\phi}.
\end{eqnarray}
Note that
\begin{eqnarray}
G^{\l}_{\eta,\eta'}=\int_{\eta'}^{\eta}\frac{1}{\sin(\phi'(\phi,\eta,y))}\ud{y}\leq-\frac{\eta'-\eta}{m}=-\frac{\sigma}{m},
\end{eqnarray}
 we can obtain
\begin{eqnarray}\label{Milne temp 58}
I_2&\leq&\lnnm{H}^2\int_{\sin\phi<0}
\bigg(\int^{-\sigma/m}_{\infty}\ue^z\ud{z}\bigg)^2\ud{\phi}=4
\ue^{-\frac{\sigma}{m}}\lnnm{H}^2.
\end{eqnarray}
For $I_3$, we can estimate as follows:
\begin{eqnarray}
I_3&\leq&\lnnm{H}^2\int_{\sin\phi<0}
\bigg(\int_{\eta}^{\infty}{\bf{1}}_{\Omega_2}{\bf{1}}_{\{\eta'-\eta\leq\sigma\}}
\frac{\exp(G_{\eta,\eta'})}{\sin(\phi'(\phi,\eta,\eta'))}\ud{\eta'}\bigg)^2\ud{\phi}\\
&\leq&\lnnm{H}^2\int_{\sin\phi<0}
\bigg(\int_{\eta}^{\eta+\sigma}{\bf{1}}_{\Omega_2}{\bf{1}}_{\{\eta'-\eta\leq\sigma\}}
\frac{\exp(G_{\eta,\eta'})}{\sin(\phi'(\phi,\eta,\eta'))}\ud{\eta'}\bigg)^2\ud{\phi}\nonumber.
\end{eqnarray}
Note that
\begin{eqnarray}
\int_{\eta}^{\infty}
\frac{1}{\sin(\phi'(\phi,\eta,\eta'))}\exp(G_{\eta,\eta'})\ud{\eta'}\leq\int_{-\infty}^{0}\ue^{z}\ud{z}=1.
\end{eqnarray}
Then $1\leq\alpha=\ue^{V(\eta')-V(\eta)}\leq
\ue^{V(\eta+\sigma)-V(\eta)}\leq 1+4\sigma$, and for
$\eta'\in[\eta,\eta+\sigma]$,
$\sin\phi'(\phi,\eta,\eta'))=\sin\bigg(\cos^{-1}(\alpha\cos\phi)\bigg)$,
$\sin(\phi'(\phi,\eta,\eta'))<m$ lead to $\cos^2(\phi'(\phi,\eta,\eta'))\geq 1-m^2$. Thus, one has
\begin{eqnarray}
\abs{\sin\phi}&=&\sqrt{1-\cos^2\phi}=\sqrt{1-\frac{\cos^2\phi'(\phi,\eta,\eta')}{\alpha^2}}\leq \frac{\sqrt{\alpha^2-1+m^2}}{\alpha}
\\&\leq&
\frac{\sqrt{(1+4\sigma)^2-1+m^2}}{\alpha}\leq
\sqrt{9\sigma+m^2}.\nonumber
\end{eqnarray}
Hence, we can obtain
\begin{eqnarray}\label{Milne temp 59}
I_3&\leq&\lnnm{H}^2\int_{\sin\phi<0}{\bf{1}}_{\Omega_2}\ud{\phi}\leq\lnnm{H}^2
\int_{\sin\phi<0}{\bf{1}}_{\{\abs{\sin\phi}\leq \sqrt{9\sigma+m^2}\}}\ud{\phi}\\
&\leq&4\sqrt{9\sigma+m^2}\lnnm{H}^2\nonumber.
\end{eqnarray}
\ \\
Step 3: \emph{The case of $\sin\phi>0$ and $\abs{E(\eta,\phi)}\geq1$}.\\
\begin{eqnarray}
\a[h(\phi)]
&=&\int_{\eta^+}^{\infty}\ldots
\exp(G_{\eta,\eta'}-G_{\eta,\eta^+})\ud{\eta'}+\int_{\eta^+}^{\eta}\ldots
\exp(-G_{\eta,\eta'})\ud{\eta'}\\
&=&\int_{\eta}^{\infty}\ldots
\exp(G_{\eta,\eta'}-G_{\eta,\eta^+})\ud{\eta'}\no\\
&&+\bigg(\int_{\eta^+}^{\eta}\ldots
\exp(G_{\eta,\eta'}-G_{\eta,\eta^+})\ud{\eta'}+\int_{\eta^+}^{\eta}\ldots
\exp(-G_{\eta,\eta'})\ud{\eta'}\bigg)\no\\
&=&I_1+I_2.\no
\end{eqnarray}
Then $I_1$ can be treated as in Case I and $I_2$ as in Case III. Hence, it is already well-treated.\\
\ \\
Summarizing all three cases, we can choose small $\sigma$ and $m$ to guarantee the relation (\ref{Milne
temp 11}).
\end{proof}

%%%%%%%%%%%%%%%%%%%%%%%%%%%%%%%%%%%%%%%%%%%%%%%%%%%%%%%%%%%%%%%%%%%%%%%%
\subsubsection{Estimates of $\e$-Milne Equation}
%%%%%%%%%%%%%%%%%%%%%%%%%%%%%%%%%%%%%%%%%%%%%%%%%%%%%%%%%%%%%%%%%%%%%%%%

The difference $z=f-f_{\infty}$ satisfies the following equation
\begin{eqnarray}\label{difference equation}
\left\{
\begin{array}{rcl}\displaystyle
\sin\phi\frac{\p z}{\p\eta}-F(\eta)\cos\phi\frac{\p
z}{\p\phi}+z&=&\bar z+S,\\
z(0,\phi)= p(\phi)&=&h(\phi)-f_{\infty}\ \ \text{for}\ \ \sin\phi>0,\\
\lim_{\eta\rt\infty}z(\eta,\phi)&=&0.
\end{array}
\right.
\end{eqnarray}
\begin{lemma}\label{Milne lemma 4}
Assume (\ref{Milne bounded}) and (\ref{Milne decay}) hold. Then
there exists a constant $C$ such that the
solution of equation (\ref{difference equation}) verifies
\begin{eqnarray} \label{Milne T34}
\lnnm{z}\leq C\bigg(1+M+\frac{M}{K}+\tnnm{z}\bigg).
\end{eqnarray}
\end{lemma}
\begin{proof}
Before giving the proof, we first show the following important inequalities for all function $l$ such that
\begin{eqnarray}
\tnm{\bar l}\leq\tnnm{l}\label{Milne temp 61},~~~~~~~\lnm{\bar l}\leq \frac{1}{\sqrt{2\pi}}\ltnm{l}.%\label{Milne temp 62}.
\end{eqnarray}
It could be directly derived by Cauchy's inequality as follows:
\begin{eqnarray}
\tnm{\bar
l}^2&=&\int_0^{\infty}\bigg(\frac{1}{2\pi}\bigg)^2
\bigg(\int_{-\pi}^{\pi}l(\eta,\phi)\ud{\phi}\bigg)^2\ud{\eta}
\\
&\leq& \int_0^{\infty} \frac{1}{2\pi}\int_{-\pi}^{\pi}l^2(\eta,\phi)\ud{\phi}\ud{\eta}
= \int_0^{\infty}\int_{-\pi}^{\pi}l^2(\eta,\phi)\ud{\phi}\ud{\eta}=\tnnm{l}^2\nonumber,\\
 \lnm{\bar l}^2 &=&\bigg(\sup_{\eta}\bar
l(\eta)\bigg)^2=\sup_{\eta}\bigg(\frac{1}{2\pi}\int_{-\pi}^{\pi}l(\eta,\phi)\ud{\phi}\bigg)^2
\\
&\leq & \frac{1}{2\pi}\sup_{\eta}\int_{-\pi}^{\pi}l^2(\eta,\phi)\ud{\phi}=\frac{1}{2\pi}\ltnm{l}^2\nonumber.
\end{eqnarray}
Then $z$ can be rewritten along the characteristics as follows:
\begin{eqnarray}
 z(\eta,\phi)=\a[p]+\t [(\bar{z}+S)(\eta,\phi)].
\end{eqnarray}
%By (\ref{difference equation}), the solution $z$ can be expressed by the operators
%\begin{eqnarray*}z=\a[p,z(\eta^+,\phi_0)]+\t[\bar z+S].\end{eqnarray*}
%which leads to
%\begin{eqnarray}
%z-\a[p,z]= \t[\bar z+S].
%\end{eqnarray}
Then by Lemma \ref{Milne lemma 3} and (\ref{Milne temp 61}), %and (\ref{Milne temp 62})
 for sufficiently small $\delta$, we can show that
\begin{eqnarray}
\lnnm{\a[p]} \leq  \lnm{p},
\end{eqnarray}
and
\begin{eqnarray}\label{Milne temp 63}
 \ltnm{ \t[\bar{z} +S]}&\leq & C(\delta)\bigg(\tnm{\bar{z}}+\tnnm{S}\bigg)+\delta\bigg(\lnm{\bar z}+\lnnm{S}\bigg)\\
&\leq& C(\delta)\bigg(\tnnm{z}+\tnnm{S}\bigg)+\delta\bigg(\ltnm{z}+\lnnm{S}\bigg)\nonumber.
%\\
%\ltnm{ \t[\bar{g}_{\l,m-1} ]} &\leq &  C(\delta)\tnm{\bar{g}_{\l,m-1}} + \delta \lnnm{\bar{g}_{\l,m-1}}\\
%&\leq& C(\delta)\tnnm{g_{\l,m-1}}+\delta \ltnm{g_{\l,m-1}}\nonumber.
\end{eqnarray}
We can directly get
\begin{eqnarray}
\ltnm{z}&\leq& 2\pi \lnnm{\a[p]}+\ltnm{\t[\bar{z}+S]}\\
&\leq & 2\pi \lnm{p} + C(\delta)\bigg(\tnnm{z}+\tnnm{S}\bigg)+\delta\bigg(\ltnm{z}+\lnnm{S}\bigg).\nonumber
%\ltnm{g_{\l,m}}&\leq& 2\pi \lnnm{\a[0,g_{\l,m}(\eta^+,\phi_0)]}+\ltnm{\t[\bar{g}_{\l,m-1}]}\\
%&\leq & 4\pi \lnnm{\bar{g}_{\l,m-1}} \leq  C(\delta)\tnnm{g_{\l,m-1}}+ \delta \ltnm{g_{\l,m-1}}\nonumber.
\end{eqnarray}
By taking $\delta=1/2$, there exists a constant $C$ independent of $\delta$ such that
\begin{eqnarray}\label{Milne temp 64}
\ltnm{z}\leq
C\bigg(\lnm{p}+\tnnm{z}+\tnnm{S}+\lnnm{S}\bigg).
\end{eqnarray}
Therefore, based on Lemma \ref{Milne lemma 2}, (\ref{Milne temp
41.}), (\ref{Milne temp 61}) and (\ref{Milne temp 64}), we can
achieve
\begin{eqnarray}
\lnnm{z} &\leq& \lnm{\a[p]}+\lnnm{\t[\bar z +S]}\\
&\leq& C\bigg(\lnm{p}+\lnm{\bar z}+\lnnm{S}\bigg)\nonumber\\
&\leq& C\bigg(\lnm{p}+\lnnm{S}+\ltnm{z}\bigg)\nonumber\\
&\leq& C\bigg(\lnm{p}+\lnnm{S}+\tnnm{S}+\tnnm{z}\bigg)\nonumber.
\end{eqnarray}
where $C$ is independent of $\l$.\\
Since $\lnm{p}$, $\tnnm{S}$, $\lnnm{S}$ and $\tnnm{z}$ are finite, we
%get a
%solution $f$ of (\ref{difference equation}) by letting $\l \rightarrow 0$.
%Denote that $z=f-f_{\infty}$, we
can yield that $z$ satisfies (\ref{Milne T34}). Then the proof of Lemma \ref{Milne lemma 4} is completed.
\end{proof}
%Lemma \ref{Milne lemma 4} naturally implies the following.
%\begin{theorem}\label{Milne lemma 5}
%The solution $f(\eta,\phi)$ to the Milne problem (\ref{Milne
%problem}) satisfies
%\begin{eqnarray}
%\lnnm{f-f_{\infty}}\leq
%\bigg(1+M+\frac{M}{K}+\tnnm{f-f_{\infty}}\bigg).
%\end{eqnarray}
%\end{theorem}
Combining Lemma \ref{Milne lemma 4} and Theorem \ref{Milne lemma
6}, we deduce the main theorem.
\begin{theorem}\label{Milne theorem 1}
There exists a unique solution $f(\eta,\phi)$ to the $\e$-Milne
problem (\ref{Milne problem}) satisfying
\begin{eqnarray}
\lnnm{f-f_{\infty}}\leq C\bigg(1+M+\frac{M}{K}\bigg).
\end{eqnarray}
\end{theorem}

%%%%%%%%%%%%%%%%%%%%%%%%%%%%%%%%%%%%%%%%%%%%%%%%%%%%%%%%%%%%%%%%%%%%%%%%
\subsection{Exponential Decay}
%%%%%%%%%%%%%%%%%%%%%%%%%%%%%%%%%%%%%%%%%%%%%%%%%%%%%%%%%%%%%%%%%%%%%%%%

In this section, we prove the spatial decay of the solution to the
Milne problem.
\begin{theorem}\label{Milne theorem 2}
Assume (\ref{Milne bounded}) and (\ref{Milne decay}) hold. For
$K_0>0$ sufficiently small, the solution $f(\eta,\phi)$ to the
$\e$-Milne problem (\ref{Milne problem}) satisfies
\begin{eqnarray}\label{Milne temp 85}
\lnnm{\ue^{K_0\eta}(f-f_{\infty})}\leq C\bigg(1+M+\frac{M}{K}\bigg),
\end{eqnarray}
\end{theorem}
\begin{proof}
Define $Z=\ue^{K_0\eta}z$ for $z=f-f_{\infty}$.
We divide the analysis into several steps:\\
\ \\
Step 1: \\
We firstly obtain the weighted $\tnnm{\cdot}$ estimate of the difference $f-f_{\infty}$.  That is, there is some constant $K_0$ is small enough such that
\begin{eqnarray}\label{Milne temp 72}
\tnnm{Z}^2=\int_0^{\infty}\ue^{2K_0\eta}\int_{-\pi}^{\pi}(f(\eta,\phi)-f_{\infty})^2\ud{\phi}\ud{\eta}\leq
C\bigg(1+M+\frac{M}{K}\bigg)^2.
\end{eqnarray}
As a consequence, it also holds that
\begin{eqnarray}\label{Milne temp 76}
\lnm{\ue^{K_0\eta}(q-q_{\infty})}\leq
C\bigg(1+M+\frac{M}{K}\bigg).
\end{eqnarray}
\emph{The proof of (\ref{Milne temp 72}):} The orthogonal property
(\ref{Milne temp 3}) reveals
\begin{eqnarray}
\br{f,f\sin\phi}_{\phi}(\eta)=\br{r,r\sin\phi}_{\phi}(\eta).
\end{eqnarray}
Multiplying $\ue^{2K_0\eta}f$ on both sides of equation (\ref{Milne
problem}) and integrating over $\phi\in[-\pi,\pi)$, we obtain
\begin{eqnarray}\label{Milne temp 15}
\ue^{2K_0\eta}\br{S,f}_{\phi}(\eta)&=& \half\frac{\ud{}}{\ud{\eta}}\bigg(\ue^{2K_0\eta}\br{r,r\sin\phi}_{\phi}(\eta)\bigg)
-\half
F(\eta)\bigg(\ue^{2K_0\eta}\br{r,r\sin\phi}_{\phi}(\eta)\bigg)
\\
& & -\ue^{2K_0\eta}\bigg(K_0\br{r,r\sin\phi}_{\phi}(\eta)-\br{r,
r}_{\phi}(\eta)\bigg) \nonumber.
\end{eqnarray}
For $K_0< 1/2$, we have
\begin{eqnarray}\label{Milne temp 16}
\half\tnm{r(\eta,\cdot)}^2 \leq -K_0\br{r,r\sin\phi}_{\phi}(\eta)+\br{r,r}_{\phi}(\eta)\leq
\frac{3}{2}\tnm{r(\eta,\cdot)}^2.
\end{eqnarray}
Let $K_0 <\min\{1/2,K\}$, similar to the proof of Lemma \ref{Milne finite LT} and Lemma
\ref{Milne infinite LT}, formulas (\ref{Milne temp 15}) and
(\ref{Milne temp 16}) imply
\begin{eqnarray}\label{Milne temp 71}
\tnnm{\ue^{K_0\eta}r}^2=\int_0^{\infty}\ue^{2K_0\eta}\br{r,r}_{\phi}(\eta)\ud{\eta}\leq
C\bigg(1+M+\frac{M}{K}\bigg)^2.
\end{eqnarray}
From (\ref{Milne temp 71}) and Cauchy's inequality, noticing $q_{\infty}=f_{\infty}$, we can deduce
\begin{eqnarray}
&&\int_0^{\infty}\ue^{2K_0\eta}\bigg(\int_{-\pi}^{\pi}(f(\eta,\phi)-f_{\infty})^2\ud{\phi}\bigg)\ud{\eta}\\
&&\leq \int_0^{\infty}\ue^{2K_0\eta}\bigg(\int_{-\pi}^{\pi}r^2(\eta,\phi)\ud{\phi}\bigg)\ud{\eta}+
\int_0^{\infty}\ue^{2K_0\eta}\bigg(\int_{-\pi}^{\pi}(q(\eta)-q_{\infty})^2\ud{\phi}\bigg)\ud{\eta}
\nonumber\\
&&\leq \int_0^{\infty}\ue^{2K_0\eta}\tnm{r(\eta,\cdot)}^2\ud{\eta}\nonumber\\
&&\hspace{3mm}+\int_0^{\infty}\ue^{2K_0\eta}\bigg(\int_{\eta}^{\infty}\abs{F(y)}\tnm{r(y,\cdot)}\ud{y}\bigg)^2\ud{\eta}
+\int_0^{\infty}\ue^{2K_0\eta}\bigg(\int_{\eta}^{\infty}\lnm{S(y,\cdot)}\ud{y}\bigg)^2\ud{\eta}
\nonumber\\
&&\leq C\bigg(1+M+\frac{M}{K}\bigg)^2+C\bigg(\int_0^{\infty}\ue^{2K_0\eta}\tnm{r(\eta,\cdot)}^2\ud{\eta}\bigg)
\bigg(\int_0^{\infty}\int_{\eta}^{\infty}\ue^{2K_0(\eta-y)}F^2(y)\ud{y}\ud{\eta}\bigg)\nonumber\\
&&\hspace{3mm}
+\int_0^{\infty}\ue^{2K_0\eta}\bigg(\int_{\eta}^{\infty}\lnm{S(y,\cdot)}\ud{y}\bigg)^2\ud{\eta}
\nonumber\\
&&\leq C\bigg(1+M+\frac{M}{K}\bigg)^2+C\bigg(\int_0^{\infty}\ue^{2K_0\eta}\tnm{r(\eta)}^2\ud{\eta}\bigg)
\bigg(\int_0^{\infty}F^2(y)\ud{y}\bigg)
\nonumber\\
&&\hspace{3mm} +\int_0^{\infty}\ue^{2K_0\eta}\bigg(\int_{\eta}^{\infty}\lnm{S(y,\cdot)}\ud{y}\bigg)^2\ud{\eta}
\nonumber\\
&&\leq C\bigg(1+M+\frac{M}{K}\bigg)^2\nonumber.
\end{eqnarray}
This completes the proof of (\ref{Milne temp 72}). From the proof of Lemma \ref{Milne infinite LT}, one gets that
\begin{eqnarray}
&&\lnm{\ue^{K_0\eta}(q-q_{\infty})}\no\\
&&\leq\lnm{\ue^{2K_0\eta}\int_{\eta}^{\infty}\abs{F(y)}\tnm{r(y,\cdot)}\ud{y}}+\lnm{\ue^{2K_0\eta}\int_{\eta}^{\infty}\lnm{S(y,\cdot)}\ud{y}}\no\\
&&\leq\lnm{\bigg(\int_{\eta}^{\infty}F^2(y)\ud{y}\bigg)^{\half}\bigg(\int_{\eta}^{\infty}\ue^{2K_0y}\tnm{r(y,\cdot)}^2\ud{y}
\bigg)^{\half}}
+\lnm{\int_{\eta}^{\infty}\ue^{2K_0y}\lnm{S(y,\cdot)}\ud{y}}\no\\
&&\leq C\bigg(1+M+\frac{M}{K}\bigg)\nonumber.
\end{eqnarray}
This shows (\ref{Milne temp 76}) when $\bar{S}=0$. Noting that all the
auxiliary functions constructed in Lemma \ref{Milne infinite LT general} satisfy the estimates (\ref{Milne temp 72}) and (\ref{Milne temp 76}), then we can extend
above $L^2$ estimates to the general $S$ case by the method
introduced in Lemma \ref{Milne infinite LT general}. \\
\ \\
Step 2: \\
We consider the decay rate of $f-f_{\infty}$ w.r.t. the spatial variable $\eta$. By a similar argument as before, we can easily show that there exists $K_0$ small enough such that
\begin{eqnarray}
\lnm{\ue^{K_0\eta}\a[p]}\leq
\lnm{p} \label{Milne temp 52}.
\end{eqnarray}
and the integral operator $\t$ satisfies
\begin{eqnarray}\label{Milne temp 54}
\lnnm{\ue^{K_0\eta}\t[H]}\leq C\lnnm{H},
\end{eqnarray}
where $C$ is a universal constant independent $H$. With these estimates and Lemma \ref{Milne lemma 3}, we will
obtain
\begin{eqnarray}\label{Milne temp 75}
\lnnm{Z}\leq C\bigg(1+M+\frac{M}{K}+\tnnm{Z}\bigg).
\end{eqnarray}
\emph{Proof of (\ref{Milne temp 75}):} $Z$ satisfies the equation
\begin{eqnarray}\label{decay equation}
\left\{
\begin{array}{rcl}\displaystyle
\sin\phi\frac{\p Z}{\p\eta}+F(\eta)\cos\phi\frac{\p
Z}{\p\phi}+Z&=&\bar Z+\ue^{K_0\eta}S+K_0\sin\phi Z,\\
Z(0,\phi)&=&p(\phi)=h(\phi)-f_{\infty}\ \ \text{for}\ \ \sin\phi>0.
\end{array}
\right.
\end{eqnarray}
Since we know
\begin{eqnarray}
Z=\a[p]+\t[\bar Z+\ue^{K_0\eta}S+K_0\sin\phi Z],
\end{eqnarray}
then by Lemma \ref{Milne lemma 3}, (\ref{Milne temp 61}), we can show
\begin{eqnarray}\label{Milne temp 7}
&&\ltnm{\t[\bar Z+\ue^{K_0\eta}S+K_0\sin\phi Z]}\\
&&\leq C(\delta)\bigg(\tnnm{\bar
Z+ \ue^{K_0\eta}S+K_0 Z}\bigg)
+\delta\bigg(\lnnm{\bar Z+ \ue^{K_0\eta}S+K_0 Z}\bigg)\nonumber\\
&&\leq
C(\delta)\bigg(\tnnm{Z}+\tnnm{\ue^{K_0\eta}S}+K_0\tnnm{Z}\bigg)\no\\
&& \qquad\qquad +\delta\bigg(\ltnm{Z}+\lnnm{\ue^{K_0\eta}S}+K_0\lnnm{Z}\bigg)\nonumber.
\end{eqnarray}
Therefore, based on Lemma \ref{Milne lemma 1} and (\ref{Milne temp
63}), we can directly estimate
\begin{eqnarray}
\ltnm{Z}
&\leq& 2 \pi \lnm{\a[p]}+\ltnm{\t[\bar Z+\ue^{K_0\eta}S+K_0\sin\phi Z]}\nonumber\\
&\leq& 2\pi\lnm{p}+C(\delta)\bigg(\tnnm{Z}+\tnnm{\ue^{K_0\eta}S}\bigg)\no\\
&&+\delta\bigg(\ltnm{Z}+\lnnm{\ue^{K_0\eta}S}+K_0\lnnm{Z}\bigg)\nonumber.
\end{eqnarray}
By taking $\delta=1/2$, we obtain
\begin{eqnarray}\label{Milne temp 74}
 \ltnm{Z}\leq
C\bigg(\lnm{p}+\tnnm{\ue^{K_0\eta}S}+\lnnm{\ue^{K_0\eta}S}+\tnnm{Z}+K_0\lnnm{Z}\bigg).
\end{eqnarray}
Then based on Lemma \ref{Milne lemma 1}, Lemma \ref{Milne lemma 2}, Lemma \ref{Milne lemma 3} and (\ref{Milne temp 76}), we can deduce
\begin{eqnarray}
\lnnm{Z}
&&\leq \lnnm{\ue^{K_0\eta}\a[p]}+\lnnm{\t[\bar Z+ e^{K_0\eta}S +K_0 Z]}\\
&&\leq \lnm{p}+\lnnm{\bar Z} +\lnnm{\ue^{K_0\eta}S}+K_0 \lnnm{ Z}\nonumber\\
&&\leq \lnm{p}+ \ltnm{Z}+\lnnm{ \u e^{K_0\eta}S}+K_0 \lnnm{Z}\nonumber\\
&&\leq C(1+M+\frac{M}{K}) + C\bigg(\tnnm{Z}+K_0\lnnm{Z}\bigg)\nonumber.
\end{eqnarray}
Taking $K_0$ sufficiently small, this completes the proof of
(\ref{Milne temp 75}).\\
\ \\
Combining (\ref{Milne temp 72}) and (\ref{Milne temp 75}), we deduce
(\ref{Milne temp 85}).
\end{proof}

%%%%%%%%%%%%%%%%%%%%%%%%%%%%%%%%%%%%%%%%%%%%%%%%%%%%%%%%%%%%%%%%%%%%%%%%
\subsection{Maximum Principle}
%%%%%%%%%%%%%%%%%%%%%%%%%%%%%%%%%%%%%%%%%%%%%%%%%%%%%%%%%%%%%%%%%%%%%%%%

\begin{theorem}\label{Milne theorem 3}
The solution $f(\eta,\phi)$ to the $\e$-Milne problem (\ref{Milne
problem}) with $S=0$ satisfies the maximum principle, i.e.
\begin{eqnarray}
\min_{\sin\phi>0}h(\phi)\leq f(\eta,\phi)\leq
\max_{\sin\phi>0}h(\phi).
\end{eqnarray}
\end{theorem}
\begin{proof}
We claim it is sufficient to show $f(\eta,\phi)\leq0$ whenever
$h(\phi)\leq0$. Suppose this claim is justified. Denote
$m=\min_{\sin\phi>0}h(\phi)$ and $M=\max_{\sin\phi>0}h(\phi)$. Then
$f^1=f-M$ satisfies the equation
\begin{eqnarray}
\left\{
\begin{array}{rcl}\displaystyle
\sin\phi\frac{\p f^1}{\p\eta}-F(\eta)\cos\phi\frac{\p
f^1}{\p\phi}+f^1-\bar f^1&=&0,\\
f^1(0,\phi)&=&h(\phi)-M\ \ \text{for}\ \ \sin\phi>0,\\
\lim_{\eta\rt\infty}f^1(\eta,\phi)&=&f^1_{\infty}.
\end{array}
\right.
\end{eqnarray}
Hence, $h-M\leq0$ implies $f^1\leq0$ which is actually $f\leq M$. On
the other hand, $f^2=m-f$ satisfies the equation
\begin{eqnarray}
\left\{
\begin{array}{rcl}\displaystyle
\sin\phi\frac{\p f^2}{\p\eta}-F(\eta)\cos\phi\frac{\p
f^2}{\p\phi}+f^2-\bar f^2&=&0,\\
f^2(0,\phi)&=&m-h(\phi)\ \ \text{for}\ \ \sin\phi>0,\\
\lim_{\eta\rt\infty}f^2(\eta,\phi)&=&f^2_{\infty}.
\end{array}
\right.
\end{eqnarray}
Thus, $m-h\leq0$ implies $f^2\leq0$ which further leads to $f\geq
m$. Therefore, the maximum principle is established.\\
\ \\
We now prove if $h(\phi)\leq0$, we have $f(\eta,\phi)\leq0$.
We divide the proof into several steps:\\
\ \\
Step 1: Penalized $\e$-Milne problem in a finite slab.\\
Assuming $h(\phi)\leq0$, we then consider the penalized Milne
problem for $f^L_{\l}(\eta,\phi)$ in the finite slab
$(\eta,\phi)\in[0,L]\times[-\pi,\pi)$
\begin{eqnarray}\label{Milne finite problem LT penalty_}
\left\{
\begin{array}{rcl}\displaystyle
\l f_{\l}^{L}+\sin\phi\frac{\p
f_{\l}^{L}}{\p\eta}-F(\eta)\cos\phi\frac{\p
f_{\l}^{L}}{\p\phi}+f_{\l}^{L}-\bar f_{\l}^{L}&=&0,\\
f_{\l}^{L}(0,\phi)&=&h(\phi)\ \ \text{for}\ \ \sin\phi<0,\\
f_{\l}^{L}(L,\phi)&=&f_{\l}^{L}(L,R\phi).
\end{array}
\right.
\end{eqnarray}
In order to construct the solution of (\ref{Milne finite problem LT
penalty_}), we iteratively define the sequence
$\{f^{L}_{m}\}_{m=1}^{\infty}$ as $f^{L}_{0}=0$ and
\begin{eqnarray}
\left\{
\begin{array}{rcl}\displaystyle
\l f^{L}_{m}+\sin\phi\frac{\p
f^{L}_{m}}{\p\eta}-F(\eta)\cos\phi\frac{\p
f^{L}_{m}}{\p\phi}+f^{L}_{m}-\bar f^{L}_{m-1}&=&0,\\
f^{L}_{m}(0,\phi)&=&h(\phi)\ \ \text{for}\ \ \sin\phi<0,\\
f^{L}_{m}(L,\phi)&=&f^{L}_{m}(L,R\phi).
\end{array}
\right.
\end{eqnarray}
Along the characteristics, it is easy to see we always have
$f^{L}_{m}<0$. In the proof of Lemma \ref{Milne finite LT}, we have
shown $f^{L}_{m}$ converges strongly in
$L^{\infty}([0,L]\times[-\pi,\pi))$ to $f_{\l}^{L}$ which satisfies
(\ref{Milne finite problem LT penalty_}). Also, $f_{\l}^{L}$
satisfies
\begin{eqnarray}
\lnnm{f_{\l}^{L}}\leq \frac{1+\l}{\l}\lnm{h}.
\end{eqnarray}
Naturally, we obtain $f_{\l}^{L}\in L^2([0,L]\times[-\pi,\pi))$
and $f_{\l}^{L}\leq0$.\\
\ \\
Step 2: $\e$-Milne problem in a finite slab.\\
Consider the Milne problem for $f^{L}(\eta,\phi)$ in a finite slab
$(\eta,\phi)\in[0,L]\times[-\pi,\pi)$
\begin{eqnarray}\label{Milne finite problem LT_}
\left\{
\begin{array}{rcl}\displaystyle
\sin\phi\frac{\p f^L}{\p\eta}-F(\eta)\cos\phi\frac{\p
f^L}{\p\phi}+f^L-\bar f^L&=&0,\\
f^L(0,\phi)&=&h(\phi)\ \ \text{for}\ \ \sin\phi<0,\\
f^L(L,\phi)&=&f^L(L,R\phi).
\end{array}
\right.
\end{eqnarray}
In the proof of Lemma \ref{Milne finite LT}, we have shown
$f_{\l}^{L}$ is uniformly bounded in $L^{2}([0,L)\times[-\pi,\pi))$
with respect to $\l$, which implies we can take weakly convergent
subsequence $f_{\l}^{L}\rightharpoonup f^L$ as $\l\rt0$ with $f^L\in
L^2([0,L]\times[-\pi,\pi))$.
Naturally, we have $f^L(\eta,\phi)\leq0$.\\
\ \\
Step 3: $\e$-Milne problem in an infinite slab.\\
Finally, in the proof of Lemma \ref{Milne infinite LT}, by taking
$L\rt\infty$, we have
\begin{eqnarray}
f^L\rightharpoonup f\ \ in\ \ L^2_{loc}([0,L)\times[-\pi,\pi),
\end{eqnarray}
where $f$ satisfies (\ref{Milne problem}). Certainly, we have
$f(\eta,\phi)\leq0$. This justifies the claim in Step 1. Hence, we
complete the proof.
\end{proof}
\begin{remark}\label{Milne remark}
Note that when $F=0$, then all the previous proofs can be recovered
and Theorem \ref{Milne theorem 1}, Theorem \ref{Milne theorem 2} and
Theorem \ref{Milne theorem 3} still hold. Hence, we can deduce the
well-posedness, decay and maximum principle of the classical Milne
problem
\begin{eqnarray}\label{classical Milne problem}
\left\{
\begin{array}{rcl}\displaystyle
\sin\phi\frac{\p f}{\p\eta}+f-\bar f&=&S(\eta,\phi),\\
f(0,\phi)&=&h(\phi)\ \ \text{for}\ \ \sin\phi>0,\\
\lim_{\eta\rt\infty}f(\eta,\phi)&=&f_{\infty}.
\end{array}
\right.
\end{eqnarray}
\end{remark}

%%%%%%%%%%%%%%%%%%%%%%%%%%%%%%%%%%%%%%%%%%%%%%%%%%%%%%%%%%%%%%%%%%%%%%%%
\section{Remainder Estimate}%%%%%%%%%%%%%%%%%%%%%%%%%%%%%%%%%%%%%%%%%%%%
%%%%%%%%%%%%%%%%%%%%%%%%%%%%%%%%%%%%%%%%%%%%%%%%%%%%%%%%%%%%%%%%%%%%%%%%

In this section, we consider the remainder estimate of the equation
\begin{eqnarray}
\left\{ \begin{array}{rcl} \epsilon\vec w\cdot\nabla_xu+u-\bar
u&=&f(\vx,\vw)\ \ \text{in}\ \
\Omega\label{neutron},\\\rule{0ex}{1.0em} u(\vec x_0,\vec
w)&=&g(\vec x_0,\vec w)\ \ \text{for}\ \ \vec x_0\in\p\Omega\ \
\text{and}\ \vw\cdot\vec n<0.
\end{array}
\right.
\end{eqnarray}
 %that the probability measure $\ud{\vw}$ on $R^2$ and
Let the measurable functions be defined on a.e.
$\Omega\times \s^1 $. The Lebesgue spaces of measurable functions on
$\Omega \times \s^1$ are denoted by $L^p(\Omega \times \s^1), 1 \leq
p \leq \infty$. These spaces are complete with respect to the norm
\begin{eqnarray}
\nm{f}_{L^p(\Omega \times \s^1)} =
\bigg(\int_{\Omega}\int_{\s^1}\abs{f(\vx,\vw)}^p\ud{\vw}\ud{\vx}\bigg)^{1/p},
\end{eqnarray}
and $L^2(\Omega \times \s^1)$ is a Hilbert space with scalar product
\begin{eqnarray}
(f,g) = (f,g)_{L^2(\Omega \times \s^1)} = \int_{\Omega}\int_{\s^1} f
g \ud{\vw}\ud{\vx}.
\end{eqnarray}
In particular, the $L^2$ and $L^{\infty}$ norms are defined as
follows
\begin{eqnarray}
\nm{f}_{L^2(\Omega\times\s^1)}&=&\bigg(\int_{\Omega}\int_{\s^1}\abs{f(\vx,\vw)}^2\ud{\vw}\ud{\vx}\bigg)^{1/2},\\
\nm{f}_{L^{\infty}(\Omega\times\s^1)}&=&\sup_{(\vx,\vw)\in\Omega\times\s^1}\abs{f(\vx,\vw)}.
\end{eqnarray}
Let $\ud{s}$ be the Lebesgue measure on $\p\Omega$, then we consider
the trace spaces $L^p(\Gamma¡À)$ for $p\geq1$ endowed with the norm
\begin{eqnarray}
\nm{f}_{L^p(\Gamma)}&=&\left(\int_{\Gamma}\abs{f(\vx,\vw)}^p\ud{\xi}\right)^{1/2},\\
\nm{f}_{L^p(\Gamma^{\pm})}&=&\left(\int_{\Gamma^{\pm}}\abs{f(\vx,\vw)}^p\ud{\xi} \right)^{1/2}.%\\
\end{eqnarray}
where $\ud{\xi} = \abs{\vw\cdot\vec n(x)}\ud{s} \ud{\vw}$. At the
same time, the $L^{\infty}$ norm of the function on the boundary is
defined as
\begin{eqnarray}
\nm{f}_{L^{\infty}(\Gamma)}&=&\sup_{(\vx,\vw)\in\Gamma}\abs{f(\vx,\vw)},\\
\nm{f}_{L^{\infty}(\Gamma^{\pm})}&=&\sup_{(\vx,\vw)\in\Gamma^{\pm}}\abs{f(\vx,\vw)}.
\end{eqnarray}
In what follows, we would show the well-posedness of the solution to
(\ref{neutron}).

%%%%%%%%%%%%%%%%%%%%%%%%%%%%%%%%%%%%%%%%%%%%%%%%%%%%%%%%%%%%%%%%%%%%%%%%
\subsection{Preliminaries}
%%%%%%%%%%%%%%%%%%%%%%%%%%%%%%%%%%%%%%%%%%%%%%%%%%%%%%%%%%%%%%%%%%%%%%%%

In order to show the $L^2$ and $L^{\infty}$ estimates of the
equation (\ref{neutron}), we start with some preparations with the
penalized neutron transport equation.
\begin{lemma}\label{well-posedness lemma 1}
Assume $f(\vx,\vw)\in L^{\infty}(\Omega\times\s^1)$ and $g(\vec
x_0,\vw)\in L^{\infty}(\Gamma)$ with $\vec x_0 \in \partial \Omega$.
Then, for any $\lambda>0$ and $\epsilon$, there exists a solution
$u_{\l}(\vx,\vw)\in L^{\infty}(\Omega\times\s^1)$ of the penalized
transport equation
\begin{eqnarray}\label{penalty equation}
\left\{
\begin{array}{rcl}
\lambda u_{\l}+\epsilon\vec
w\cdot\nabla_xu_{\l}+u_{\l}&=&f(\vx,\vw)\ \ \text{in}\ \ \Omega
\times \s^1,\\\rule{0ex}{1.0em} u_{\l}&=&g\  \text{on}\ \
\partial\Omega \times \s^1, \ \vw\cdot\vec n<0,
\end{array}
\right.
\end{eqnarray}
which satisfies the following bound
\begin{eqnarray}
\im{u_{\l}}{\Omega\times\s^1}\leq \im{g}{\Gamma^-}
+\im{f}{\Omega\times\s^1}.
\end{eqnarray}
\end{lemma}
\begin{proof}
The characteristics $(X(s),W(s))$ of the equation (\ref{penalty
equation}) which goes through $(\vx,\vw)$ is defined by
\begin{eqnarray}\label{character}
\left\{
\begin{array}{lll}
\dfrac{\ud{X(s)}}{\ud{s}}=\e W(s),
~~~~\dfrac{\ud{W(s)}}{\ud{s}}=0,\\ \rule{0ex}{2.0em}
(X(0),W(0))=(\vx,\vw).
\end{array}
\right.
\end{eqnarray}
which implies
\begin{eqnarray}
X(s)=\vx+(\e\vw)s,~~~ W(s)=\vw.
\end{eqnarray}
Hence, we can rewrite the equation (\ref{penalty equation}) along
the characteristics as
\begin{eqnarray}\label{well-posedness temp 31}
u_{\l}(\vx,\vw)= g(\vx-\e
t_b\vw,\vw)e^{-(1+\l)t_b}+\int_{0}^{t_b}f(\vx-\e(t_b-s)\vw,\vw)e^{-(1+\l)(t_b-s)}\ud{s},
\end{eqnarray}
where the backward exit time $t_b$ is defined as
\begin{eqnarray}\label{exit time}
t_b(\vx,\vw)=\inf\{t\geq0: (\vx-\e t\vw,\vw)\in\Gamma^-\}.
\end{eqnarray}
Then, since $t_{b} \geq 0$, it gives the following estimate
\begin{eqnarray}
|u_{\l}(\vx,\vw)|&\leq&e^{-(1+\l)t_b}\im{g}{\Gamma^-}+\frac{1-e^{(1+\l)t_b}}{1+\l}\im{f}{\Omega\times\s^1}\\
&\leq&\im{g}{\Gamma^-}+\im{f}{\Omega\times\s^1}\nonumber.
\end{eqnarray}
Since $u_{\l}$ can be explicitly traced back to the boundary data,
the existence of the solution to the equation (\ref{penalty
equation}) naturally follows from the above estimate.
\end{proof}

\begin{lemma}\label{well-posedness lemma 2}
Assume $f(\vx,\vw)\in L^{\infty}(\Omega\times\s^1)$ and $g(\vec
x,\vw)\in L^{\infty}(\Gamma)$ with $\vec x \in \partial \Omega$.
Then, for any $\lambda>0$ and $\epsilon$, there exists a solution
$u_{\l}(\vx,\vw)\in L^{\infty}(\Omega\times\s^1)$ of the penalized
transport equation
\begin{eqnarray}\label{penalty equation}
\left\{
\begin{array}{rcl}
\lambda u_{\l}+\epsilon\vec
w\cdot\nabla_xu_{\l}+u_{\l}-\bar{u}_{\l}&=&f(\vx,\vw)\ \ \text{in}\
\ \Omega \times \s^1,\\\rule{0ex}{1.0em} u_{\l}&=&g\  \text{on}\ \
\partial\Omega \times \s^1, \ \vw\cdot\vec n<0,
\end{array}
\right.
\end{eqnarray}
which satisfies the following bound
 \begin{eqnarray}
\im{u_{\l}}{\Omega\times\s^1}\leq \frac{1+\l}{\l}\bigg(
\im{f}{\Omega\times\s^1}+\im{g}{\Gamma^-}\bigg).
\end{eqnarray}
\end{lemma}
\begin{proof}
We define an approximating sequence $\{u_{\l}^k\}_{k=0}^{\infty}$,
where $u_{\l}^0=0$ and
%\begin{eqnarray}\label{penalty temp 1}
%\left\{
%\begin{array}{rcl}
%\lambda u_{\l}^{0}+\epsilon\vec w\cdot\nabla_xu_{\l}^0 +u_{\l}^0&=& 0\ \ \text{in}\ \ \Omega \times \s^1 %,\\\rule{0ex}{1.0em}
%u_{\l}^0&=&g \ \ \text{on}\ \ \vec
%\p\Omega \times \s^1\ \ \text{and}\ \vw\cdot\vec n<0,
%\end{array}
%\right.
%\end{eqnarray}
%and $u^k_{\l}$ satisfies
\begin{eqnarray}\label{penalty temp 1}
\left\{
\begin{array}{rcl}
\lambda u_{\l}^{k}+\epsilon\vec
w\cdot\nabla_xu_{\l}^k+u_{\l}^k&=&\bar u_{\l}^{k-1}+f(\vx,\vw)\ \
\text{in}\ \ \Omega \times \s^1 ,\\\rule{0ex}{1.0em} u_{\l}^k&=&g \
\ \text{on}\ \ \p\Omega \times \s^1\ \ \text{and}\ \vw\cdot\vec n<0.
\end{array}
\right.
\end{eqnarray}
By Lemma \ref{well-posedness lemma 1}, $u_{\l}^1$ is well-defined
and \begin{eqnarray}\label{Estimate sequence 0}
\im{u_{\l}^1}{\Omega\times\s^1}<
\im{g}{\Gamma^-}+\im{f}{\Omega\times\s^1} .\end{eqnarray}
%The characteristics and the backward exit time are defined as
%(\ref{character}) and (\ref{exit time}), so we rewrite equation
%(\ref{penalty temp 2}) along the characteristics as
%\begin{eqnarray}
%u_{\l}^k(\vx,\vw)= g(\vx-\e
%t_b\vw,\vw)e^{-(1+\l)t_b}+\int_{0}^{t_b}(f+\bar
%u_{\l}^{k-1})(\vx-\e(t_b-s)\vw,\vw)e^{-(1+\l)(t_b-s)}\ud{s}.
%\end{eqnarray}
We define the difference $v^k=u_{\l}^{k}-u_{\l}^{k-1}$ for $k\geq
1$. Recall (\ref{average 1}) for $\bar v^k$, then $v^k$ satisfies
\begin{eqnarray}
\label{penalty temp 3} \left\{
\begin{array}{rcl}
\lambda v^{k}+\epsilon\vec w\cdot\nabla_x v^k +v^k &=& \bar v^{k-1}
\ \text{in}\ \ \Omega \times \s^1 ,\\\rule{0ex}{1.0em} v_{\l}^k&=&0
\ \ \text{on}\ \ \p\Omega \times \s^1\ \ \text{and}\ \vw\cdot\vec
n<0.
\end{array}
\right.
\end{eqnarray}
Since $\im{\bar
v^k}{\Omega\times\s^1}\leq\im{v^k}{\Omega\times\s^1}$,
%we can directly estimate
%\begin{eqnarray}
%\im{v^{k+1}}{\Omega\times\s^1}&\leq&\im{v^{k}}{\Omega\times\s^1}\int_0^{t_b}e^{-(1+\l)(t_b-s)}\ud{s}
%\leq\frac{1-e^{-(1+\l)t_b}}{1+\l}\im{v^{k}}{\Omega\times\s^1}.
%\end{eqnarray}
%Hence,
we naturally have
\begin{eqnarray}\label{Estimate sequence 1}
\im{v^{k+1}}{\Omega\times\s^1}\leq \frac{1}{1+\l}\im{\bar
v^{k}}{\Omega\times\s^1}\leq
\frac{1}{1+\l}\im{v^{k}}{\Omega\times\s^1}.
\end{eqnarray}
Thus, $(u_{\l}^k)_{k=0}^{\infty}$ is a contraction series for
$\l>0$.
%Now, we consider $v^1$, which satisfies that
%\begin{eqnarray}
%\label{penalty temp 4}
%\left\{
%\begin{array}{rcl}
%\lambda v^1+\epsilon\vec w\cdot\nabla_x v^1 +v^1 &=& \bar
%u^0+f(\vx,\vw ) \ \text{in}\ \ \Omega \times \s^1 ,\\\rule{0ex}{1.0em}
%v^1&=&0 \ \ \text{on}\ \ \vec
%\p\Omega \times \s^1\ \ \text{and}\ \vw\cdot\vec n<0.
%\end{array}
%\right.
%\end{eqnarray}
%Based on Lemma \ref{well-posedness lemma 1}, we can directly
%get
%\begin{eqnarray}\label{Estimate sequence 2}
%\im{v^1}{\Omega\times\s^1}\leq \frac{1}{1+\l}\bigg(
%\im{f}{\Omega\times\s^1}+\im{g}{\Gamma^-}\bigg).
%\end{eqnarray}
%\begin{eqnarray}
%\im{v^1}{\Omega\times\s^1}\leq \bigg(\frac{1}{1+\l}\bigg) \im{u^0_{\l}}{\Omega\times\s^1},
%\end{eqnarray}
%for $k\geq1$. Therefore, $u_{\l}^k$ converges strongly in
%$L^{\infty}$ to a limit solution $u_{\l}$ satisfying
%\begin{eqnarray}\label{well-posedness temp 1}
%\im{u_{\l}}{\Omega\times\s^1}\leq\sum_{k=1}^{\infty}\im{v^{k}}{\Omega\times\s^1}\leq\frac{1+\l}{\l}\im{u_{\l}^1}{\Omega\times\s^1}.
%\end{eqnarray}
%Since $u_{\l}^1$ can be rewritten along the characteristics as
%\begin{eqnarray}
%u_{\l}^1(\vx,\vw)&=&g(\vx-\e
%t_b\vw,\vw)e^{-(1+\l)t_b}+\int_{0}^{t_b}f(\vx-\e(t_b-s)\vw,\vw)e^{-(1+\l)(t_b-s)}\ud{s},
%\end{eqnarray}
%based on Lemma \ref{well-posedness lemma 1}, we can directly
%estimate
%\begin{eqnarray}\label{well-posedness temp 2}
%\im{u_{\l}^1}{\Omega\times\s^1}\leq
%\im{f}{\Omega\times\s^1}+\im{g}{\Gamma^-}.
%\end{eqnarray}
Since $v^1 =u_{\l}^1$, by combining (\ref{Estimate sequence 0}) and
(\ref{Estimate sequence 1}), it easily deduces
\begin{eqnarray}\label{Estimate sequence 3}
\im{u^k}{\Omega\times\s^1}\leq \sum_{k=1}^{\infty}
\im{v^k}{\Omega\times\s^1} \leq \frac{1+\l}{\l} \left(
\im{g}{\Gamma^-}+\im{f}{\Omega\times\s^1}\right).
\end{eqnarray}
Let $k \rightarrow \infty$, we get the existence of the solution to
(\ref{penalty equation}). This completes the proof of Lemma
\ref{well-posedness lemma 2}.
\end{proof}
From (\ref{penalty equation}), the bound of the solution depends on $\lambda$. Then we can not get the solution of the equation by letting $\lambda$ tends to zero. So, we need to show a uniform estimate of the solution to the penalized
neutron transport equation (\ref{penalty equation}) with respect to
$\l$.

%%%%%%%%%%%%%%%%%%%%%%%%%%%%%%%%%%%%%%%%%%%%%%%%%%%%%%%%%%%%%%%%%%%%%%%%
\subsection{Uniform $L^2$ Estimate}
%%%%%%%%%%%%%%%%%%%%%%%%%%%%%%%%%%%%%%%%%%%%%%%%%%%%%%%%%%%%%%%%%%%%%%%%

We recall the following Green's Identity, which could be found in
\cite[Chapter 9]{Cercignani.Illner.Pulvirenti1994} and
\cite{Esposito.Guo.Kim.Marra2013}.
\begin{lemma}(Green's Identity)\label{well-posedness lemma 3}
Assume $f(\vx,\vw),\ g(\vx,\vw)\in L^2(\Omega\times\s^1)$ and
$\vw\cdot\nx f,\ \vw\cdot\nx g\in L^2(\Omega\times\s^1)$ with $f,\
g\in L^2(\Gamma)$. Then
\begin{eqnarray}
\iint_{\Omega\times\s^1}\bigg((\vw\cdot\nx f)g+(\vw\cdot\nx
g)f\bigg)\ud{\vx}\ud{\vw}=\int_{\Gamma}fg\ud{\gamma},
\end{eqnarray}
where $\ud{\gamma}=(\vw\cdot\vec n)\ud{s}$ on the boundary.
\end{lemma}

We firstly give the uniform $L^2$ estimate of the hydrodynamic part of the solution.

\begin{lemma}\label{well-posedness lemma 4}
The solution $u_{\l}$ to the equation (\ref{penalty equation})
satisfies the uniform estimate
\begin{eqnarray}\label{well-posedness temp 3}
\e\tm{\bar u_{\l}}{\Omega\times\s^1}\leq C(\Omega)\bigg(
\tm{u_{\l}-\bar
u_{\l}}{\Omega\times\s^1}+\tm{f}{\Omega\times\s^1}+\e\tm{u_{\l}}{\Gamma^{+}}+\e\tm{g}{\Gamma^-}\bigg),
\end{eqnarray}
for $0\leq\l \leq 1$ and $0<\e\leq 1$.
\end{lemma}
\begin{proof} From Lemma \ref{well-posedness lemma 2}, it is nature that
$u_{\l} \in L^2(\Omega\times \s^1)$ as well as $u_{\l} \in
L^2(\Gamma)$. It follows that $\vw \cdot \nabla u_{\l} \in
L^2(\Omega \times \s^1)$ from the equation (\ref{penalty equation}).
Then, for any $\phi\in L^2(\Omega\times\s^1)$ satisfying
$\vw\cdot\nx\phi\in L^2(\Omega\times\s^1)$ and $\phi\in
L^2(\Gamma)$,  we have
\begin{eqnarray}\label{well-posedness temp 4}
\l\iint_{\Omega\times\s^1}u_{\l}\phi+\e\int_{\Gamma}u_{\l}\phi\ud{\gamma}
-\e\iint_{\Omega\times\s^1}u_{\l}(\vw\cdot\nx\phi)+\iint_{\Omega\times\s^1}(u_{\l}-\bar
u_{\l})\phi=\iint_{\Omega\times\s^1}f\phi.
\end{eqnarray}
Similar to \cite{AA003}, it needs to choose a particular test
function $\phi$. We define $\zeta(\vx)$ on $\Omega$ satisfying
\begin{eqnarray}\label{test temp 1}
\left\{
\begin{array}{rcl}
\Delta \zeta&=&\bar u_{\l}\ \ \text{in}\ \
\Omega,\\\rule{0ex}{1.0em} \zeta&=&0\ \ \text{on}\ \ \p\Omega.
\end{array}
\right.
\end{eqnarray}
Based on the standard elliptic estimate, we have
\begin{eqnarray}\label{test temp 3}
\nm{\zeta}_{H^2(\Omega)}\leq C(\Omega)\nm{\bar
u_{\l}}_{L^2(\Omega)}.
\end{eqnarray}
The test function is defined as
\begin{eqnarray}\label{test temp 2}
\phi=-\vw\cdot\nx\zeta
\end{eqnarray}
Naturally, we have
\begin{eqnarray}\label{test temp 4}
\nm{\phi}_{H^1(\Omega)}+\nm{\phi}_{L^{\infty}(\Omega)}\leq
C\nm{\zeta}_{H^2(\Omega)}\leq C(\Omega)\nm{\bar
u_{\l}}_{L^2(\Omega)}.
\end{eqnarray}
We can decompose
\begin{eqnarray}\label{test temp 5}
-\e\iint_{\Omega\times\s^1}(\vw\cdot\nx\phi)u_{\l}&=&-\e\iint_{\Omega\times\s^1}(\vw\cdot\nx\phi)\bar
u_{\l}-\e\iint_{\Omega\times\s^1}(\vw\cdot\nx\phi)(u_{\l}-\bar
u_{\l}).
\end{eqnarray}
For the first part, we have
\begin{eqnarray}\label{wellposed temp 1}
-\e\iint_{\Omega\times\s^1}(\vw\cdot\nx\phi)\bar
u_{\l}&=&\e\iint_{\Omega\times\s^1}\bar
u_{\l}\bigg(w_1(w_1\p_{11}\zeta+w_2\p_{12}\zeta)+w_2(w_1\p_{12}\zeta+w_2\p_{22}\zeta)\bigg)\\
&=&\e\iint_{\Omega\times\s^1}\bar
u_{\l}\bigg(w_1^2\p_{11}\zeta+2w_1w_2\p_{12}+ w_2^2\p_{22}\zeta\bigg)\nonumber\\
&=&\e\pi\int_{\Omega}\bar u_{\l}(\p_{11}\zeta+\p_{22}\zeta)\nonumber\\
&=&\e\pi\nm{\bar u_{\l}}_{L^2(\Omega)}^2\nonumber\\
&=&\half\e\nm{\bar u_{\l}}_{L^2(\Omega\times\s^1)}^2\nonumber.
\end{eqnarray}
On the other hand,  H\"older's inequality and the elliptic estimate
imply
\begin{eqnarray}\label{wellposed temp 2}
\abs{\e\iint_{\Omega\times\s^1}(\vw\cdot\nx\phi)(u_{\l}-\bar
u_{\l})} &\leq&C(\Omega)\e\nm{u_{\l}-\bar u_{\l}}_{L^2(\Omega\times\s^1)}\nm{\zeta}_{H^2(\Omega)}\\
&\leq&C(\Omega)\e\nm{u_{\l}-\bar
u_{\l}}_{L^2(\Omega\times\s^1)}\nm{\bar
u_{\l}}_{L^2(\Omega\times\s^1)}\nonumber.
\end{eqnarray}
Based on (\ref{test temp 3}), (\ref{test temp 4}), the boundary
condition of the penalized neutron transport equation (\ref{penalty
equation}), the trace theorem, H\"older's inequality and the
elliptic estimate, we have
\begin{eqnarray}\label{wellposed temp 3}
\abs{\e\int_{\Gamma}u_{\l}\phi\ud{\gamma}} & \leq & \e \bigg(\int_{\Gamma}|u_{\l}|^2 \ud{\gamma}\bigg)^{1/2} \bigg(\int_{\Gamma}|\phi|^2 \ud{\gamma}\bigg)^{1/2} \\
&\leq&
C(\Omega)\e\bigg(\nm{u_{\l}}_{L^2(\Gamma^+)}+\tm{g}{\Gamma^-}\bigg)\nm{\bar
u_{\l}}_{L^2(\Omega\times\s^1)}\nonumber,
\end{eqnarray}
\begin{eqnarray}\label{wellposed temp 4}
\l\abs{\iint_{\Omega\times\s^1}u_{\l}\phi } &=&
\l\abs{\iint_{\Omega\times\s^1}\bar
u_{\l}\phi+\l\iint_{\Omega\times\s^1}(u_{\l}-\bar u_{\l})\phi}\\
&=&\l\iint_{\Omega\times\s^1}\abs{ (u_{\l}-\bar u_{\l})\phi }
\nonumber \\
&\leq & C(\Omega)\l\nm{u_{\l}-\bar
u_{\l}}_{L^2(\Omega\times\s^1)}\nm{\bar
u_{\l}}_{L^2(\Omega\times\s^1)}\nonumber,
\end{eqnarray}
\begin{eqnarray}\label{wellposed temp 5}
\abs{\iint_{\Omega\times\s^1}(u_{\l}-\bar u_{\l})\phi}\leq
C(\Omega)\nm{u_{\l}-\bar u_{\l}}_{L^2(\Omega\times\s^1)}\nm{\bar
u_{\l}}_{L^2(\Omega\times\s^1)},
\end{eqnarray}
\begin{eqnarray}\label{wellposed temp 6}
\abs{\iint_{\Omega\times\s^1}f\phi } \leq C(\Omega)\nm{\bar
u_{\l}}_{L^2(\Omega\times\s^1)}\nm{f}_{L^2(\Omega\times\s^1)}.
\end{eqnarray}
Collecting terms in (\ref{wellposed temp 1}), (\ref{wellposed temp
2}), (\ref{wellposed temp 3}), (\ref{wellposed temp 4}),
(\ref{wellposed temp 5}) and (\ref{wellposed temp 6}), we obtain
\begin{eqnarray}
\e\nm{\bar u_{\l}}_{L^2(\Omega\times\s^1)} & \leq&
C(\Omega)\bigg((1+\e+\l)\nm{u_{\l}-\bar
u_{\l}}_{L^2(\Omega\times\s^1)}\\
& &~~~
+\e\nm{u_{\l}}_{L^2(\Gamma^+)}+\nm{f}_{L^2(\Omega\times\s^1)}+\e\tm{g}{\Gamma^-}\bigg)\nonumber,
\end{eqnarray}
So we get the desired uniform estimate with respect to $0\leq
\lambda\leq 1$.
\end{proof}
\begin{theorem}\label{LT estimate}
Assume $f(\vx,\vw)\in L^2(\Omega\times\s^1)$ and $g(\vec x_0,\vw)\in
L^2(\Gamma^-)$. Then for the steady neutron transport equation
(\ref{neutron}), there exists a unique solution $u(\vx,\vw)\in
L^2(\Omega\times\s^1)$ satisfying
\begin{eqnarray}
\tm{u}{\Omega\times\s^1}\leq C(\Omega)\bigg(
\frac{1}{\e^2}\tm{f}{\Omega\times\s^1}+\frac{1}{\e^{1/2}}\tm{g}{\Gamma^-}\bigg).
\end{eqnarray}
\end{theorem}
\begin{proof}
In the weak formulation (\ref{well-posedness temp 4}), we may take
the test function $\phi=u_{\l}$ to get the energy estimate
\begin{eqnarray}
\l\nm{u_{\l}}_{L^2(\Omega\times\s^1)}^2+\half\e\int_{\Gamma}\abs{u_{\l}}^2\ud{\gamma}+\nm{u_{\l}-\bar
u_{\l}}_{L^2(\Omega\times\s^1)}^2=\iint_{\Omega\times\s^1}fu_{\l}.
\end{eqnarray}
Hence, this naturally implies
\begin{eqnarray}\label{well-posedness temp 5}
\half\e\nm{u_{\l}}_{L^2(\Gamma^+)}^2+\nm{u_{\l}-\bar
u_{\l}}_{L^2(\Omega\times\s^1)}^2\leq\Big| \iint_{\Omega\times\s^1}fu_{\l}\Big|+\half\e\nm{g}_{L^2(\Gamma^-)}^2.
\end{eqnarray}
On the other hand, we can square on both sides of
(\ref{well-posedness temp 3}) to obtain
\begin{eqnarray}\label{well-posedness temp 6}
\\
\e^2\tm{\bar u_{\l}}{\Omega\times\s^1}^2\leq C(\Omega)\bigg(
\tm{u_{\l}-\bar
u_{\l}}{\Omega\times\s^1}^2+\tm{f}{\Omega\times\s^1}^2+\e^2\tm{u_{\l}}{\Gamma^{+}}^2+\e^2\tm{g}{\Gamma^-}^2\bigg).\nonumber
\end{eqnarray}
Multiplying a sufficiently small constant on both sides of
(\ref{well-posedness temp 6}) and adding it to (\ref{well-posedness
temp 5}) to absorb $\nm{u_{\l}}_{L^2(\Gamma^+)}^2$ and
$\nm{u_{\l}-\bar u_{\l}}_{L^2(\Omega\times\s^1)}^2$, we deduce
\begin{eqnarray}
&&\e\nm{u_{\l}}_{L^2(\Gamma^+)}^2+\e^2\nm{\bar
u_{\l}}_{L^2(\Omega\times\s^1)}^2+\nm{u_{\l}-\bar
u_{\l}}_{L^2(\Omega\times\s^1)}^2\\&&\qquad\qquad\qquad\leq
C(\Omega)\bigg(\tm{f}{\Omega\times\s^1}^2+
\Big|\iint_{\Omega\times\s^1}fu_{\l}\Big|+\e\nm{g}_{L^2(\Gamma^-)}^2\bigg).\nonumber
\end{eqnarray}
Hence, we have
\begin{eqnarray}\label{well-posedness temp 7}
\e\nm{u_{\l}}_{L^2(\Gamma^+)}^2+\e^2\nm{u_{\l}}_{L^2(\Omega\times\s^1)}^2\leq
C(\Omega)\bigg(\tm{f}{\Omega\times\s^1}^2+
\Big|\iint_{\Omega\times\s^1}fu_{\l}\Big|+\e\nm{g}_{L^2(\Gamma^-)}^2\bigg).
\end{eqnarray}
A simple application of Cauchy's inequality leads to
\begin{eqnarray}
\iint_{\Omega\times\s^1}fu_{\l}\leq\frac{1}{4C\e^2}\tm{f}{\Omega\times\s^1}^2+C\e^2\tm{u_{\l}}{\Omega\times\s^1}^2.
\end{eqnarray}
Taking $C$ sufficiently small, we can divide (\ref{well-posedness
temp 7}) by $\e^2$ to obtain
\begin{eqnarray}\label{well-posedness temp 21}
\frac{1}{\e}\nm{u_{\l}}_{L^2(\Gamma^+)}^2+\nm{u_{\l}}_{L^2(\Omega\times\s^2)}^2\leq
C(\Omega)\bigg(
\frac{1}{\e^4}\nm{f}_{L^2(\Omega\times\s^2)}^2+\frac{1}{\e}\nm{g}_{L^2(\Gamma^-)}^2\bigg).
\end{eqnarray}
Since above estimate does not depend on $\l$, it gives a uniform
estimate for the penalized neutron transport equation (\ref{penalty
equation}). Thus, we can extract a weakly convergent subsequence
$u_{\l}\rt u$ as $\l\rt0$. The weak lower semi-continuity of norms
$\nm{\cdot}_{L^2(\Omega\times\s^2)}$ and
$\nm{\cdot}_{L^2(\Gamma^+)}$ implies $u$ also satisfies the estimate
(\ref{well-posedness temp 21}). Hence, in the weak formulation
(\ref{well-posedness temp 4}), we can take $\l\rt0$ to deduce that
$u$ satisfies equation (\ref{neutron}). Also $u_{\l}-u$ satisfies
the equation
\begin{eqnarray}
\left\{
\begin{array}{rcl}
\epsilon\vec w\cdot\nabla_x(u_{\l}-u)+(u_{\l}-u)-(\bar u_{\l}-\bar
u)&=&-\l u_{\l}\ \ \text{in}\ \
\Omega\label{remainder},\\\rule{0ex}{1.0em} (u_{\l}-u)(\vec x_0,\vec
w)&=&0\ \ \text{for}\ \ \vec x_0\in\p\Omega\ \ and\ \vw\cdot\vec
n<0.
\end{array}
\right.
\end{eqnarray}
By a similar argument as above, we can achieve
\begin{eqnarray}
\nm{u_{\l}-u}_{L^2(\Omega\times\s^2)}^2\leq
C(\Omega)\bigg(\frac{\l}{\e^4}\nm{u_{\l}}_{L^2(\Omega\times\s^2)}^2\bigg).
\end{eqnarray}
When $\l\rt0$, the right-hand side approaches zero, which implies
the convergence is actually in the strong sense. The uniqueness
easily follows from the energy estimates.
\end{proof}

%%%%%%%%%%%%%%%%%%%%%%%%%%%%%%%%%%%%%%%%%%%%%%%%%%%%%%%%%%%%%%%%%%%%%%%%
\subsection{$L^{\infty}$ Estimate}
%%%%%%%%%%%%%%%%%%%%%%%%%%%%%%%%%%%%%%%%%%%%%%%%%%%%%%%%%%%%%%%%%%%%%%%%

\begin{theorem}\label{LI estimate}
Assume $f(\vx,\vw)\in L^{\infty}(\Omega\times\s^1)$ and $g(\vec
x_0,\vw)\in L^{\infty}(\Gamma^-)$. Then for the neutron transport
equation (\ref{neutron}), there exists a unique solution
$u(\vx,\vw)\in L^{\infty}(\Omega\times\s^1)$ satisfying
\begin{eqnarray}
\im{u}{\Omega\times\s^1}\leq C(\Omega)\bigg(
\frac{1}{\e^{3}}\im{f}{\Omega\times\s^1}+\frac{1}{\e^{3/2}}\im{g}{\Gamma^-}\bigg).
\end{eqnarray}
\end{theorem}
\begin{proof}
We divide the proof into several steps to bootstrap an $L^2$ solution to an $L^{\infty}$ solution:\\
\ \\
Step 1: Double Duhamel's iterations.\\
The characteristics of the equation (\ref{neutron}) is given by
(\ref{character}). Hence, we can rewrite the equation
(\ref{neutron}) along the characteristics as
\begin{eqnarray}
u(\vx,\vw)&=&g(\vx-\e t_b\vw,\vw)\ue^{-t_b}+\int_{0}^{t_b}f(\vx-\e(t_b-s)\vw,\vw)\ue^{-(t_b-s)}\ud{s}\\
&&+\frac{1}{2\pi}\int_{0}^{t_b}\bigg(\int_{\s^1}u(\vx-\e(t_b-s)\vw,\vw_t)\ud{\vw_t}\bigg)\ue^{-(t_b-s)}\ud{s}\nonumber.
\end{eqnarray}
where the backward exit time $t_b$ is defined as (\ref{exit time}).
Note we have replaced $\bar u$ by the integral of $u$ over the dummy
velocity variable $\vw_t$. For the last term in this formulation, we
apply the Duhamel's principle again to $u(\vx-\e(t_b-s)\vw,\vw_t)$
and obtain
\begin{eqnarray}\label{well-posedness temp 8}
u(\vx,\vw)&=&g(\vx-\e t_b\vw,\vw)\ue^{-t_b}+\int_{0}^{t_b}f(\vx-\e(t_b-s)\vw,\vw)\ue^{-(t_b-s)}\ud{s}\\
&&+\frac{1}{2\pi}\int_{0}^{t_b}\int_{\s^1}g(\vx-\e(t_b-s)\vw-\e
s_b\vw_t,\vw_t)\ue^{-s_b}\ud{\vw_t}\ue^{-(t_b-s)}\ud{s}\nonumber\\
&&+\frac{1}{2\pi}\int_{0}^{t_b}\int_{\s^1}\bigg(\int_{0}^{s_b}f(\vx-\e(t_b-s)\vw-\e
(s_b-r)\vw_t,\vw_t)\ue^{-(s_b-r)}\ud{r}\bigg)\ud{\vw_t}\ue^{-(t_b-s)}\ud{s}\nonumber\\
&&+
\frac{1}{(2\pi)^2}\int_{0}^{t_b}\int_{\s^1}\ue^{-(t_b-s)}\bigg(\int_{0}^{s_b}\int_{\s^1}\ue^{-(s_b-r)}u(\vx-\e(t_b-s)\vw-\e
(s_b-r)\vw_t,\vw_s)\ud{\vw_s}\ud{r}\bigg)\ud{\vw_t}\ud{s}\nonumber,
\end{eqnarray}
where we introduce another dummy velocity variable $\vw_s$ and
\begin{eqnarray}
s_b(\vx,\vw,s,\vw_t)=\inf\{r\geq0: (\vx-\e(t_b-s)\vw-\e
r\vw_t,\vw_t)\in\Gamma^-\}.
\end{eqnarray}
\ \\
Step 2: Estimates of all but the last term in (\ref{well-posedness temp 8}).\\
We can directly estimate as follows:
\begin{eqnarray}\label{im temp 1}
\abs{g(\vx-\e t_b\vw,\vw)e^{-t_b}}\leq\im{g}{\Gamma^-},
\end{eqnarray}
\begin{eqnarray}\label{im temp 2}
\abs{\frac{1}{2\pi}\int_{0}^{t_b}\int_{\s^1}g(\vx-\e(t_b-s)\vw-\e
s_b\vw_t,\vw_t)\ue^{-s_b}\ud{\vw_t}\ue^{-(t_b-s)}\ud{s}} \leq
\im{g}{\Gamma^-},
\end{eqnarray}
\begin{eqnarray}\label{im temp 3}
\abs{\int_{0}^{t_b}f(\vx-\e(t_b-s)\vw,\vw)\ue^{-(t_b-s)}\ud{s}}\leq
\im{f}{\Omega\times\s^1},
\end{eqnarray}
\begin{eqnarray}\label{im temp 4}
\\
\abs{\frac{1}{2\pi}\int_{0}^{t_b}\int_{\s^1}\bigg(\int_{0}^{s_b}f(\vx-\e(t_b-s)\vw-\e
(s_b-r)\vw_t,\vw_t)\ue^{-(s_b-r)}\ud{r}\bigg)\ud{\vw_t}\ue^{-(t_b-s)}\ud{s}}
\leq \im{f}{\Omega\times\s^1}\nonumber.
\end{eqnarray}
\ \\
Step 3: Estimates of the last term in (\ref{well-posedness temp 8}).\\
Now we decompose the last term in (\ref{well-posedness temp 8}) as
\begin{eqnarray}
\int_{0}^{t_b}\int_{\s^1}\int_0^{s_b}\int_{\s^1}=\int_{0}^{t_b}\int_{\s^1}\int_{s_b-r\leq\delta}\int_{\s^1}+
\int_{0}^{t_b}\int_{\s^1}\int_{s_b-r\geq\delta}\int_{\s^1}=I_1+I_2,
\end{eqnarray}
for some $\delta>0$. We can estimate $I_1$ directly as
\begin{eqnarray}\label{im temp 5}
\abs{I_1}
&\leq&\int_{0}^{t_b}\ue^{-(t_b-s)}\bigg(\int_{\max\{0,s_b-\delta\}}^{s_b}\im{u}{\Omega\times\s^1}\ud{r}\bigg)\ud{s}\leq\delta\im{u}{\Omega\times\s^1}.
\end{eqnarray}
Then we can bound $I_2$ as
\begin{eqnarray}
\\
I_2\leq
\int_{0}^{t_b}\int_{\s^1}\int_{0}^{\max\{0,s_b-\delta\}}\int_{\s^1}\abs{u(\vx-\e(t_b-s)\vw-\e
(s_b-r)\vw_t,\vw_s)}\ue^{-(t_b-s)}\ue^{-(s_b-r)}\ud{\vw_s}\ud{r}\ud{\vw_t}\ud{s}.\nonumber
\end{eqnarray}
By the definition of $t_b$ and $s_b$, we always have
$\vx-\e(t_b-s)\vw-\e (s_b-r)\vw_t\in\bar\Omega$. Hence, we may interchange the order of integration
and apply H\"older's inequality to obtain
\begin{eqnarray}
\abs{I_2}&\leq&\bigg(\int_{0}^{t_b}\int_{\s^1}\int_{0}^{\max\{0,s_b-\delta\}}\int_{\s^1}\ue^{-(t_b-s)}\ue^{-(s_b-r)}\ud{\vw_s}\ud{r}\ud{\vw_t}\ud{s}\bigg)^{1/2}\\
&&\bigg(\int_{0}^{t_b}\int_{\s^1}\int_{0}^{\max\{0,s_b-\delta\}}\int_{\s^1}{\bf{1}}_{\vx-\e(t_b-s)\vw-\e (s_b-r)\vw_t\in\bar\Omega}\no\\
&&\abs{u}^2(\vx-\e(t_b-s)\vw-\e
(s_b-r)\vw_t,\vw_s)\ue^{-(t_b-s)}\ue^{-(s_b-r)}\ud{\vw_s}\ud{r}\ud{\vw_t}\ud{s}\bigg)^{1/2}\no\\
&\leq&\bigg(\int_{0}^{t_b}\int_{\s^1}\int_{0}^{\max\{0,s_b-\delta\}}\int_{\s^1}{\bf{1}}_{\vx-\e(t_b-s)\vw-\e (s_b-r)\vw_t\in\bar\Omega}\no\\
&&\abs{u}^2(\vx-\e(t_b-s)\vw-\e
(s_b-r)\vw_t,\vw_s)\ue^{-(t_b-s)}\ue^{-(s_b-r)}\ud{\vw_s}\ud{r}\ud{\vw_t}\ud{s}\bigg)^{1/2}.\no
\end{eqnarray}

We may write it in a new variable $\psi$ as
$\vw_t=(\cos\psi,\sin\psi)$. By the change of the variable
$[-\pi,\pi] \times \r' \rightarrow \Omega: (\psi,r) \rightarrow
(y_1,y_2) =\vx-\e(t_b-s)\vw-\e(s_b-r)\vw_t$, i.e.
\begin{eqnarray}
\left\{
\begin{array}{rcl}
y_1&=&x_1-\e(t_b-s)w_1-\e(s_b-r)\cos\psi,\\
y_2&=&x_2-\e(t_b-s)w_2-\e(s_b-r)\sin\psi.
\end{array}
\right.
\end{eqnarray}
Therefore, for $s_b-r\geq\delta$, we can directly compute the Jacobian
\begin{eqnarray}
\abs{\frac{\p{(y_1,y_2)}}{\p{(\psi,r)}}}=\abs{\abs{\begin{array}{rc}
\e(s_b-r)\sin\psi&\e\cos\psi\\
-\e(s_b-r)\cos\psi&\e\sin\psi
\end{array}}}=\e^2(s_b-r) \geq\e^2\delta.
\end{eqnarray}
%Hence, we may simplify (\ref{well-posedness temp 22}) as
%\begin{eqnarray}
%I_2&\leq&\frac{C}{\e\sqrt{\delta}}\int_{0}^{t_b}\int_{\s^1}\bigg(\int_{\Omega}\abs{u(\vec
%y,\vw_s)}^2\ud{\vec y}\bigg)^{1/2}e^{-(t_b-s)}\ud{\vw_s}\ud{s}.
%\end{eqnarray}
Then we may further utilize Cauchy's inequality and the $L^2$
estimate of $u$ in Theorem \ref{LT estimate} to obtain we get
\begin{eqnarray}\label{im temp 6}
\abs{I_2} &\leq & \frac{C}{\e \sqrt{\delta}} \nm{u}_{L^2(\Omega \times \s^1)} \\
&\leq&\frac{C(\Omega)}{\sqrt{\delta}}\bigg(\frac{1}{\e^{3}}\tm{f}{\Omega\times\s^1}+\frac{1}{\e^{3/2}}\tm{g}{\Gamma^-}\bigg)\nonumber\\
&\leq&\frac{C(\Omega)}{\sqrt{\delta}}\bigg(\frac{1}{\e^{3}}\im{f}{\Omega\times\s^1}+\frac{1}{\e^{3/2}}\im{g}{\Gamma^-}\bigg)\nonumber.
\end{eqnarray}
\ \\
In summary, collecting (\ref{im temp 1}), (\ref{im temp 2}),
(\ref{im temp 3}), (\ref{im temp 4}), (\ref{im temp 5}) and (\ref{im
temp 6}), for fixed $0<\delta<1$, we have
\begin{eqnarray}
\abs{u(\vx,\vw)}\leq \delta
\im{u}{\Omega\times\s^1}+\frac{C(\Omega)}{\sqrt{\delta}}
\bigg(\frac{1}{\e^{3}}\im{f}{\Omega\times\s^1}+\frac{1}{\e^{3/2}}\im{g}{\Gamma^-}\bigg).
\end{eqnarray}
%Then we may take $0<\delta\leq1/2$ to obtain
%\begin{eqnarray}
%\abs{u(\vx,\vw)}\leq
%\half\im{u}{\Omega\times\s^1}+\frac{C(\Omega)}{\sqrt{\delta}}
%\bigg(\frac{1}{\e^{5/2}}\im{f}{\Omega\times\s^1}+\frac{1}{\e}\im{g}{\Gamma^-}\bigg).
%\end{eqnarray}
%we have
%\begin{eqnarray}
%\im{u}{\Omega\times\s^1}\leq
%\half\im{u}{\Omega\times\s^1}+\frac{C(\Omega)}{\sqrt{\delta}}
%\bigg(\frac{1}{\e^{3}}\im{f}{\Omega\times\s^1}+\frac{1}{\e^{3/2}}\im{g}{\Gamma^-}\bigg).
%\end{eqnarray}
Taking supremum of $u$ over all $(\vx,\vw)$ and absorbing
$\im{u}{\Omega\times\s^1}$, for fixed $0<\delta\leq1/2$, we get
\begin{eqnarray}
\im{u}{\Omega\times\s^1}\leq C(\Omega)\bigg(
\frac{1}{\e^{3}}\im{f}{\Omega\times\s^1}+\frac{1}{\e^{3/2}}\im{g}{\Gamma^-}\bigg).
\end{eqnarray}

\end{proof}

%%%%%%%%%%%%%%%%%%%%%%%%%%%%%%%%%%%%%%%%%%%%%%%%%%%%%%%%%%%%%%%%%%%%%%%%
\section{Diffusive Limit}%%%%%%%%%%%%%%%%%%%%%%%%%%%%%%%%%%%%%%%%%%%%%%%
%%%%%%%%%%%%%%%%%%%%%%%%%%%%%%%%%%%%%%%%%%%%%%%%%%%%%%%%%%%%%%%%%%%%%%%%

%\begin{theorem}
%Assume $g_{\pm}(\vx_0,\vw)\in C^3(\Gamma^-)$. Then for the steady
%neutron transport equation (\ref{transport}), the unique solution
%$u^{\e}(\vx,\vw)\in L^{\infty}(\Omega\times\s^1)$ satisfies
%\begin{eqnarray}\label{main theorem 1}
%\lnm{u^{\e}-\u_0-\ub_{+,0}-\ub_{-,0}}=O(\e)
%\end{eqnarray}
%where the interior solution $\u_0$ and boundary layer $\ub_{\pm,0}$
%are defined in (\ref{expansion temp 8}) and (\ref{expansion temp
%9}). Moreover, if $g_+(\phi,\theta))=0$ and
%$g_{-}(\theta, \phi)=\cos\phi$, then there exists a positive constant $C>0$ such that
%\begin{eqnarray}\label{main theorem 2}
%\lnm{u^{\e}-\uc_0-\ubc_{+,0}-\ubc_{-,0}}\geq C>0
%\end{eqnarray}
%when $\e$ is sufficiently small, where $\uc_0$ and $\ubc_{\pm,0}$
%are defined in (\ref{classical temp 2}) and (\ref{classical temp
%1}).
%\end{theorem}
\noindent{\bf The proof of \ref{main 1}}
We divide the proof into several steps:\\

\noindent Step 1: Remainder definitions.\\
We may rewrite the asymptotic expansion as follows:
\begin{eqnarray}
u^{\e}&\sim&\sum_{k=0}^{\infty}\e^k\u_k+\sum_{k=0}^{\infty}\e^k\ub_{+,k}+\sum_{k=0}^{\infty}\e^k\ub_{-,k}.
\end{eqnarray}
The remainder can be defined as
\begin{eqnarray}\label{pf 1}
R_N&=&u^{\e}-\sum_{k=0}^{N}\e^k\u_k-\sum_{k=0}^{\infty}\e^k\ub_{+,k}-\sum_{k=0}^{\infty}\e^k\ub_{-,k}=u^{\e}-\q_N-\qb_{+,N}-\qb_{-,N},
\end{eqnarray}
where
\begin{eqnarray}
\q_N&=&\sum_{k=0}^{N}\e^k\u_k,\\
\qb_{+,N}&=&\sum_{k=0}^{N}\e^k\ub_{+,k},\\
\qb_{-,N}&=&\sum_{k=0}^{N}\e^k\ub_{-,k}.
\end{eqnarray}
Noting the equation (\ref{transport temp}) is equivalent to the
equation (\ref{transport}), we write $\ll$ to denote the neutron
transport operator as follows:
\begin{eqnarray}
\ll u&=&\e\vw\cdot\nx u+u-\bar u\\
&=&\pm\sin\phi\frac{\p
u}{\p\eta_{\pm}}-\frac{\e}{R_{\pm}\mp\e\eta_{\pm}}\cos\phi\bigg(\frac{\p
u}{\p\phi}+\frac{\p u}{\p\theta}\bigg)+u-\bar u.\nonumber
\end{eqnarray}
\ \\
Step 2: Estimates of $\ll \q_N$.\\
The interior contribution can be estimated as
\begin{eqnarray}
\ll\q_0=\e\vw\cdot\nx \q_0+\q_0-\bar \q_0&=&\e\vw\cdot\nx
\u_0+(\u_0-\bu_0)=\e\vw\cdot\nx \u_0.
\end{eqnarray}
We have
\begin{eqnarray}
\abs{\e\vw\cdot\nx \u_0}\leq C\e\abs{\nx \u_0}\leq C\e.
\end{eqnarray}
This implies
\begin{eqnarray}
\abs{\ll \q_0}\leq C\e.
\end{eqnarray}
Similarly, for higher order term, we can estimate
\begin{eqnarray}
\ll\q_N=\e\vw\cdot\nx \q_N+\q_N-\bar \q_N&=&\e^{N+1}\vw\cdot\nx
\u_N.
\end{eqnarray}
We have
\begin{eqnarray}
\abs{\e^{N+1}\vw\cdot\nx \u_N}\leq C\e^{N+1}\abs{\nx \u_N}\leq
C\e^{N+1}.
\end{eqnarray}
This implies
\begin{eqnarray}\label{pf 2}
\abs{\ll \q_N}\leq C\e^{N+1}.
\end{eqnarray}
\ \\
Step 3: Estimates of $\ll \qb_{\pm,N}$.\\
The boundary layer solution is
$\ub_k=(f_{\pm,k}^{\e}-f_{\pm,k}^{\e}(\infty))\cdot\psi_0=\v_{\pm,k}\psi_0$
where $f_{\pm,k}^{\e}(\eta_{\pm},\theta,\phi)$ solves the $\e$-Milne
problem and $\v_{\pm,k}=f_{\pm,k}^{\e}-f_{\pm,k}^{\e}(\infty)$.
Notice $\psi_0\psi=\psi_0$, so the boundary layer contribution can
be estimated as
\begin{eqnarray}\label{remainder temp 1}
&&\ll\qb_{\pm,0}\\
&=&\pm\sin\phi\frac{\p
\qb_{\pm,0}}{\p\eta_{\pm}}-\frac{\e}{R_{\pm}\mp\e\eta_{\pm}}\cos\phi\bigg(\frac{\p
\qb_{\pm,0}}{\p\phi}+\frac{\p
\qb_{\pm,0}}{\p\theta}\bigg)+\qb_{\pm,0}-\bar
\qb_{\pm,0}\no\\
&=&\pm\sin\phi\bigg(\psi_0\frac{\p
\v_{\pm,0}}{\p\eta_{\pm}}+\v_{\pm,0}\frac{\p\psi_0}{\p\eta_{\pm}}\bigg)-\frac{\psi_0\e}{R_{\pm}\mp\e\eta_{\pm}}\cos\phi\bigg(\frac{\p
\v_{\pm,0}}{\p\phi}+\frac{\p \v_{\pm,0}}{\p\theta}\bigg)+\psi_0 \v_{\pm,0}-\psi_0\bar\v_{\pm,0}\nonumber\\
&=&\pm\sin\phi\bigg(\psi_0\frac{\p
\v_{\pm,0}}{\p\eta_{\pm}}+\v_{\pm,0}\frac{\p\psi_0}{\p\eta_{\pm}}\bigg)-\frac{\psi_0\psi\e}{R_{\pm}\mp\e\eta_{\pm}}\cos\phi\bigg(\frac{\p
\v_{\pm,0}}{\p\phi}+\frac{\p \v_{\pm,0}}{\p\theta}\bigg)+\psi_0 \v_{\pm,0}-\psi_0\bar\v_{\pm,0}\nonumber\\
&=&\psi_0\bigg(\pm\sin\phi\frac{\p
\v_{\pm,0}}{\p\eta_{\pm}}-\frac{\e\psi}{R_{\pm}\mp\e\eta_{\pm}}\cos\phi\frac{\p
\v_{\pm,0}}{\p\phi}+\v_{\pm,0}-\bar\v_{\pm,0}\bigg)\pm\sin\phi
\frac{\p\psi_0}{\p\eta_{\pm}}\v_{\pm,0}-\frac{\psi_0\e}{R_{\pm}\mp\e\eta_{\pm}}\cos\phi\frac{\p
\v_{\pm,0}}{\p\theta}\nonumber\\
&=&\pm\sin\phi
\frac{\p\psi_0}{\p\eta_{\pm}}\v_{\pm,0}-\frac{\psi_0\e}{R_{\pm}\mp\e\eta_{\pm}}\cos\phi\frac{\p
\v_{\pm,0}}{\p\theta}\nonumber.
\end{eqnarray}
Since $\psi_0=1$ when $\eta_{\pm}\leq 1/(4\e)(R_{+}-R_{-})$, the
effective region of $\px\psi_0$ is $\eta\geq1/(4\e)(R_{+}-R_{-})$
which is further and further from the origin as $\e\rt0$. By Theorem
4.13 in \cite{AA003} and Theorem \ref{Milne theorem 2}, the first
term in (\ref{remainder temp 1}) can be controlled as
\begin{eqnarray}
\abs{\pm\sin\phi\frac{\p\psi_0}{\p\eta_{\pm}}\v_{\pm,0}}&\leq&
Ce^{-\frac{K_0}{\e}}\leq C\e.
\end{eqnarray}
For the second term in (\ref{remainder temp 1}), we have
\begin{eqnarray}
\abs{-\frac{\psi_0\e}{R_{\pm}\mp\e\eta_{\pm}}\cos\phi\frac{\p
\v_{\pm,0}}{\p\theta}}&\leq&C\e\abs{\frac{\p
\v_{\pm,0}}{\p\theta}}\leq C\e.
\end{eqnarray}
This implies
\begin{eqnarray}
\abs{\ll \qb_{\pm,0}}\leq C\e.
\end{eqnarray}
Similarly, for higher order term, we can estimate
\begin{eqnarray}\label{remainder temp 2}
\ll\qb_{\pm,N}&=&\sin\phi\frac{\p
\qb_{\pm,N}}{\p\eta_{\pm}}-\frac{\e}{R_{\pm}\mp\e\eta_{\pm}}\cos\phi\bigg(\frac{\p
\qb_{\pm,N}}{\p\phi}+\frac{\p
\qb_{\pm,N}}{\p\theta}\bigg)+\qb_{\pm,N}-\bar
\qb_{\pm,N}\\
&=&\sum_{i=0}^N\e^i\sin\phi
\frac{\p\psi_0}{\p\eta_{\pm}}\v_{\pm,i}-\frac{\psi_0\e^{N+1}}{R_{\pm}\mp\e\eta_{\pm}}\cos\phi\frac{\p
\v_{\pm,N}}{\p\theta}\nonumber.
\end{eqnarray}
Away from the origin, the first term in (\ref{remainder temp 2}) can
be controlled as
\begin{eqnarray}
\abs{\sum_{i=0}^N\e^i\sin\phi
\frac{\p\psi_0}{\p\eta_{\pm}}\v_{\pm,i}}&\leq&
Ce^{-\frac{K_0}{\e}}\leq C\e^{N+1}.
\end{eqnarray}
For the second term in (\ref{remainder temp 2}), we have
\begin{eqnarray}
\abs{-\frac{\psi_0\e^{N+1}}{R_{\pm}\mp\e\eta_{\pm}}\cos\phi\frac{\p
\v_{\pm,N}}{\p\theta}}&\leq&C\e^{N+1}\abs{\frac{\p
\v_{\pm,N}}{\p\theta}}\leq C\e^{N+1}.
\end{eqnarray}
This implies
\begin{eqnarray}\label{pf 3}
\abs{\ll \qb_N}\leq C\e^{N+1}.
\end{eqnarray}
\ \\
Step 4: Diffusive limit.\\
\underline{\it Proof of (\ref{main theorem 1})}. In summary, since $\ll u^{\e}=0$, collecting (\ref{pf 1}), (\ref{pf
2}) and (\ref{pf 3}), we can prove
\begin{eqnarray}
\abs{\ll R_N}\leq C\e^{N+1}.
\end{eqnarray}
Consider the asymptotic expansion to $N=3$, then the remainder $R_3$
satisfies the equation
\begin{eqnarray}
\left\{
\begin{array}{rcl}
\e \vw\cdot\nabla_x R_3+R_3-\bar R_3&=&\ll R_3\ \ \text{for}\ \ \vx\in\Omega,\\
R_3(\vx_0,\vw)&=&0\ \ \text{for}\ \ \vw\cdot\vec n<0\ \ \text{and}\
\ \vx_0\in\p\Omega.
\end{array}
\right.
\end{eqnarray}
By Theorem 3.6 in \cite{AA003}, we have
\begin{eqnarray}
\im{R_3}{\Omega\times\s^1}\leq \frac{C(\Omega)}{\e^{3}}\im{\ll
R_3}{\Omega\times\s^1}\leq\frac{C(\Omega)}{\e^{3}}(C\e^4)\leq C(\Omega)\e.
\end{eqnarray}
Hence, we have
\begin{eqnarray}
\nm{u^{\e}-\sum_{k=0}^3\e^k\u_k-\sum_{k=0}^3\e^k\ub_{+,k}-\sum_{k=0}^3\e^k\ub_{-,k}}_{L^{\infty}(\Omega\times\s^1)}=O(\e).
\end{eqnarray}
Since it is easy to see
\begin{eqnarray}
\nm{\sum_{k=1}^3\e^k\u_k+\sum_{k=1}^3\e^k\ub_{+,k}+\sum_{k=1}^3\e^k\ub_{-,k}}_{L^{\infty}(\Omega\times\s^1)}=O(\e),
\end{eqnarray}
our result naturally follows. This completes the proof of (\ref{main
theorem 1}).\\
\ \\
Step 5: Counterexample of the expansion.\\

\noindent\underline{{\it Proof of (\ref{main theorem 2})}}. It is divided it into the following steps.

(1). The classical Milne problem.

By (\ref{classical temp 1}), the solution $\f_{\pm,0}$ satisfies the
Milne problem
\begin{eqnarray}
\left\{
\begin{array}{rcl}\displaystyle
\pm\sin(\theta+\xi)\frac{\p \f_{\pm,0}}{\p\eta_{\pm}}+\f_{\pm,0}-\bar \f_{\pm,0}&=&0,\\
\f_{\pm,0}(0,\theta,\xi)&=&g_{\pm}(\theta,\xi)\ \ \text{for}\ \
\pm\sin(\theta+\xi)>0,\\\rule{0ex}{1em}
\lim_{\eta_{\pm}\rt\infty}\f_{\pm,0}(\eta_{\pm},\theta,\xi)&=&f_{\pm,0}(\infty,\theta).
\end{array}
\right.
\end{eqnarray}
For convenience of comparison, we make the substitution
$\phi=\theta+\xi$ to obtain
\begin{eqnarray}
\left\{
\begin{array}{rcl}\displaystyle
\pm\sin\phi\frac{\p \f_{\pm,0}}{\p\eta_{\pm}}+\f_{\pm,0}-\bar \f_{\pm,0}&=&0,\\
\f_{\pm,0}(0,\theta,\phi)&=&g_{\pm}(\theta,\phi)\ \ \text{for}\ \
\pm\sin\phi>0,\\\rule{0ex}{1em}
\lim_{\eta_{\pm}\rt\infty}\f_{\pm,0}(\eta_{\pm},\theta,\phi)&=&f_{\pm,0}(\infty,\theta).
\end{array}
\right.
\end{eqnarray}
Assume (\ref{main theorem 2}) is incorrect, i.e.
\begin{eqnarray}
\lim_{\e\rt0}\lnm{(\uc_0+\ubc_{+,0}+\ubc_{-,0})-(\u_0+\ub_{+,0}+\ub_{-,0})}=0.
\end{eqnarray}
Since the boundary $g_{\pm}(\phi)$ independent of $\theta$, by
(\ref{classical temp 1}) and (\ref{expansion temp 9}), it is obvious
the limit of zeroth order boundary layer $f_{\pm,0}(\infty,\theta)$
and $f_{\pm,0}^{\e}(\infty,\theta)$ satisfy
$f_{\pm,0}(\infty,\theta)=C_{\pm,1}$ and
$f_{\pm,0}^{\e}(\infty,\theta)=C_{\pm,2}$ for some constant
$C_{\pm,1}$ and $C_{\pm,2}$ independent of $\theta$. It is easy to
see $C_{+,1}=C_{+,2}=0$. By (\ref{classical temp 2}) and
(\ref{expansion temp 8}), we can derive the interior solutions are
smooth in the domain $\Omega$. Hence, for $\abs{\eta_{-}}\leq\e$ we
may further derive
\begin{eqnarray}\label{compare temp 5}
\lim_{\e\rt0}\lnm{(f_{-,0}(\infty)+\ubc_{-,0})-(f_{-,0}^{\e}(\infty)+\ub_{-,0})}=0.
\end{eqnarray}
For $0\leq\eta\leq 1/(2\e)(R_{+}-R_{-})$, we have $\psi_0=1$, which
means $\f_{-,0}=\ubc_{-,0}+f_{-,0}(\infty)$ and
$f_{-,0}^{\e}=\ub_{-,0}+f_{-,0}^{\e}(\infty)$ on $[0,\e]$. Define
$u=f_{-,0}+2$, $U=f_{-,0}^{\e}+2$ and $G_{-}=g_{-}+2=\cos\phi+2$,
using the substitution and notation in Section 3, then
$u(\eta,\phi)$ satisfies the equation
\begin{eqnarray}\label{compare flat equation}
\left\{
\begin{array}{rcl}\displaystyle
\sin\phi\frac{\p u}{\p\eta}+u-\bar u&=&0,\\
u(0,\phi)&=&G(\phi)\ \ \text{for}\ \ \sin\phi>0,\\\rule{0ex}{1em}
\lim_{\eta\rt\infty}u(\eta,\phi)&=&2+f_{0}(\infty),
\end{array}
\right.
\end{eqnarray}
and $U(\eta,\phi)$ satisfies the equation
\begin{eqnarray}\label{compare force equation}
\left\{
\begin{array}{rcl}\displaystyle
\sin\phi\frac{\p U}{\p\eta}-F(\e;\eta)\cos\phi \frac{\p
U}{\p\phi}+U-\bar
U&=&0,\\
U(0,\phi)&=&G(\phi)\ \ \text{for}\ \ \sin\phi>0,\\\rule{0ex}{1em}
\lim_{\eta\rt\infty}U(\eta,\phi)&=&2+f_{0}^{\e}(\infty).
\end{array}
\right.
\end{eqnarray}
Based on (\ref{compare temp 5}), we have
\begin{eqnarray}
\lim_{\e\rt0}\lnm{U(\eta,\phi)-u(\eta,\phi)}=0.
\end{eqnarray}
Then it naturally implies
\begin{eqnarray}
\lim_{\e\rt0}\lnm{\bar U(\eta)-\bar u(\eta)}=0.
\end{eqnarray}

(2) Continuity of $\bar u$ and $\bar U$ at $\eta=0$.

For the problem (\ref{compare flat equation}), we have for any
$r_0>0$
\begin{eqnarray}
\abs{\bar u(\eta)-\bar
u(0)}&\leq&\frac{1}{2\pi}\bigg(\int_{\sin\phi\leq
r_0}\abs{u(\eta,\phi)-u(0,\phi)}\ud{\phi}+\int_{\sin\phi\geq
r_0}\abs{u(\eta,\phi)-u(0,\phi)}\ud{\phi}\bigg).
\end{eqnarray}
Since we have shown $u\in L^{\infty}([0,\infty)\times[-\pi,\pi))$,
then for any $\delta>0$, we can take $r_0$ sufficiently small such
that
\begin{eqnarray}
\frac{1}{2\pi}\int_{\sin\phi\leq
r_0}\abs{u(\eta,\phi)-u(0,\phi)}\ud{\phi}&\leq&\frac{C}{2\pi}\arcsin
r_0\leq \frac{\delta}{2}.
\end{eqnarray}
For fixed $r_0$ satisfying above requirement, we estimate the
integral on $\sin\phi\geq r_0$. By Ukai's trace theorem, $u(0,\phi)$
is well-defined in the domain $\sin\phi\geq r_0$ and is continuous.
Also, by consider the relation
\begin{eqnarray}
\frac{\p u}{\p\eta}(0,\phi)=\frac{\bar u(0)-u(0,\phi)}{\sin\phi},
\end{eqnarray}
we can obtain in this domain $\px u$ is bounded, which further
implies $u(\eta,\phi)$ is uniformly continuous at $\eta=0$. Then
there exists $\delta_0>0$ sufficiently small, such that for any
$0\leq\eta\leq\delta_0$, we have
\begin{eqnarray}
\frac{1}{2\pi}\int_{\sin\phi\geq
r_0}\abs{u(\eta,\phi)-u(0,\phi)}\ud{\phi}&\leq&\frac{1}{2\pi}\int_{\sin\phi\geq
r_0}\frac{\delta}{2}\ud{\phi}\leq\frac{\delta}{2}.
\end{eqnarray}
In summary, we have shown for any $\delta>0$, there exists
$\delta_0>0$ such that for any $0\leq\eta\leq\delta_0$,
\begin{eqnarray}
\abs{\bar u(\eta)-\bar
u(0)}\leq\frac{\delta}{2}+\frac{\delta}{2}=\delta.
\end{eqnarray}
Hence, $\bar u(\eta)$ is continuous at $\eta=0$. By a similar
argument along the characteristics, we can show $\bar U(\eta,\phi)$
is also continuous at $\eta=0$.

In the following, by the continuity, we assume for arbitrary
$\delta>0$, there exists a $\delta_0>0$ such that for any
$0\leq\eta\leq\delta_0$, we have
\begin{eqnarray}
\abs{\bar u(\eta)-\bar u(0)}&\leq&\delta\label{compare temp 1},\\
\abs{\bar U(\eta)-\bar U(0)}&\leq&\delta\label{compare temp 2}.
\end{eqnarray}

(3). The $\e-$Milne formulation.

We consider the solution at a specific point $(\eta,\phi)=(n\e,\e)$
for some fixed $0<n\leq1/2$. The solution along the characteristics
can be rewritten as follows:
\begin{eqnarray}\label{compare temp 3}
u(n\e,\e)=G(\e)\ue^{-\frac{1}{\sin\e}n\e}
+\int_0^{n\e}\ue^{-\frac{1}{\sin\e}(n\e-\k)}\frac{1}{\sin\e}\bar
u(\k)\ud{\k},
\end{eqnarray}
\begin{eqnarray}\label{compare temp 4}
U(n\e,\e)=G(\e_0)\ue^{-\int_0^{n\e}\frac{1}{\sin\phi(\zeta)}\ud{\zeta}}
+\int_0^{n\e}\ue^{-\int_{\k}^{n\e}\frac{1}{\sin\phi(\zeta)}\ud{\zeta}}\frac{1}{\sin\phi(\k)}\bar
U(\k)\ud{\k},
\end{eqnarray}
where we have the conserved energy along the characteristics
\begin{eqnarray}
E(\eta,\phi)=\cos\phi \ue^{V(\eta)},
\end{eqnarray}
in which $(0,\e_0)$ and $(\zeta,\phi(\zeta))$ are in the same
characteristics of $(n\e,\e)$.

(4) Estimates of (\ref{compare temp 3}).

We turn to the Milne problem for $u$. We have the natural estimate
\begin{eqnarray}
\int_0^{n\e}\ue^{-\frac{1}{\sin\e}(n\e-\k)}\frac{1}{\sin\e}\ud{\k}&=&\int_0^{n\e}\ue^{-\frac{1}{\e}(n\e-\k)}\frac{1}{\e}\ud{\k}+o(\e)\\
&=&\ue^{-n}\int_0^{n\e}\ue^{\frac{\k}{\e}}\frac{1}{\e}\ud{\k}+o(\e)\nonumber\\
&=&\ue^{-n}\int_0^ne^{\zeta}\ud{\zeta}+o(\e)\nonumber\\
&=&(1-\ue^{-n})+o(\e)\nonumber.
\end{eqnarray}
Then for $0<\e\leq\delta_0$, we have $\abs{\bar u(0)-\bar
u(\k)}\leq\delta$, which implies
\begin{eqnarray}
\int_0^{n\e}\ue^{-\frac{1}{\sin\e}(n\e-\k)}\frac{1}{\sin\e}\bar
u(\k)\ud{\k}&=&
\int_0^{n\e}\ue^{-\frac{1}{\sin\e}(n\e-\k)}\frac{1}{\sin\e}\bar u(0)\ud{\k}+O(\delta)\\
&=&(1-\ue^{-n})\bar u(0)+o(\e)+O(\delta)\nonumber.
\end{eqnarray}
For the boundary data term, it is easy to see
\begin{eqnarray}
G(\e)\ue^{-\frac{1}{\sin\e}n\e}&=&\ue^{-n}G(\e)+o(\e)
\end{eqnarray}
In summary, we have
\begin{eqnarray}
u(n\e,\e)=(1-\ue^{-n})\bar u(0)+\ue^{-n}G(\e)+o(\e)+O(\delta).
\end{eqnarray}
\ \\
(5). Approximation of (\ref{compare temp 4}).

We consider the $\e$-Milne problem for $U$. For $\e<<1$ sufficiently
small, $\psi(\e)=1$. Then we may estimate
\begin{eqnarray}
\cos\phi(\zeta)\ue^{V(\zeta)}=\cos\e \ue^{V(n\e)},
\end{eqnarray}
which implies
\begin{eqnarray}
\cos\phi(\zeta)=\frac{1-\e\zeta}{1-n\e^2}\cos\e.
\end{eqnarray}
and hence
\begin{eqnarray}
\sin\phi(\zeta)=\sqrt{1-\cos^2\phi(\zeta)}=\sqrt{-\frac{\e(n\e-\zeta)(2-\e\zeta-n\e^2)}{(1-n\e^2)^2}\cos^2\e+\sin^2\e}.
\end{eqnarray}
For $\zeta\in[0,\e]$ and $n\e$ sufficiently small, by Taylor's
expansion, we have
\begin{eqnarray}
2-\e\zeta-n\e^2&=&2+o(\e),\\
\sin^2\e&=&\e^2+o(\e^3),\\
\cos^2\e&=&1-\e^2+o(\e^3).
\end{eqnarray}
Hence, we have
\begin{eqnarray}
\sin\phi(\zeta)=\sqrt{\e(\e-2n\e+2\zeta)}+o(\e^2).
\end{eqnarray}
Since $\sqrt{\e(\e-2n\e+2\zeta)}=O(\e)$, we can further estimate
\begin{eqnarray}
\frac{1}{\sin\phi(\zeta)}&=&\frac{1}{\sqrt{\e(\e-2n\e+2\zeta)}}+o(1)\\
-\int_{\k}^{n\e}\frac{1}{\sin\phi(\zeta)}\ud{\zeta}&=&-\sqrt{\frac{\e-2n\e+2\zeta}{\e}}\bigg|_{\k}^{n\e}+o(\e)
=\sqrt{\frac{\e-2n\e+2\k}{\e}}-1+o(\e).
\end{eqnarray}
Then we can easily derive the integral estimate
\begin{eqnarray}
\int_0^{n\e}\ue^{-\int_{\k}^{n\e}\frac{1}{\sin\phi(\zeta)}\ud{\zeta}}\frac{1}{\sin\phi(\k)}\ud{\k}&=&
\ue^{-1}\int_0^{n\e}\ue^{\sqrt{\frac{\e-2n\e+2\k}{\e}}}\frac{1}{\sqrt{\e(\e-2n\e+2\k)}}\ud{\k}+o(\e)\\
&=&\half \ue^{-1}\int^{\e}_{(1-2n)\e}\ue^{\sqrt{\frac{\sigma}{\e}}}\frac{1}{\sqrt{\e\sigma}}\ud{\sigma}+o(\e)\nonumber\\
&=&\half \ue^{-1}\int^{1}_{1-2n}\ue^{\sqrt{\rho}}\frac{1}{\sqrt{\rho}}\ud{\rho}+o(\e)\nonumber\\
&=&\ue^{-1}\int^{1}_{\sqrt{{1-2n}}}\ue^{t}\ud{t}+o(\e)\nonumber\\
&=&(1-\ue^{\sqrt{1-2n}-1})+o(\e)\nonumber.
\end{eqnarray}
Then for $0<\e\leq\delta_0$, we have $\abs{\bar U(0)-\bar
U(\k)}\leq\delta$, which implies
\begin{eqnarray}
\int_0^{n\e}\ue^{-\int_{\k}^{n\e}\frac{1}{\sin\phi(\zeta)}\ud{\zeta}}\frac{1}{\sin\phi(\k)}\bar
U(\k)\ud{\k}&=&
\int_0^{n\e}\ue^{-\int_{\k}^{n\e}\frac{1}{\sin\phi(\zeta)}\ud{\zeta}}\frac{1}{\sin\phi(\k)}\bar U(0)\ud{\k}+O(\delta)\\
&=&(1-\ue^{\sqrt{1-2n}-1})\bar U(0)+o(\e)+O(\delta)\nonumber.
\end{eqnarray}
For the boundary data term, since $G(\phi)$ is $C^1$, a similar
argument shows
\begin{eqnarray}
G(\e_0)\ue^{-\int_0^{n\e}\frac{1}{\sin\phi(\zeta)}\ud{\zeta}}&=&\ue^{\sqrt{1-2n}-1}G(\sqrt{1-2n}\e)+o(\e).
\end{eqnarray}
Therefore, we have
\begin{eqnarray}
U(n\e,\e)=(1-\ue^{\sqrt{1-2n}-1})\bar
U(0)+\ue^{\sqrt{1-2n}-1}G(\sqrt{1-2n}\e)+o(\e)+O(\delta).
\end{eqnarray}

(6). Approximation of (\ref{main theorem 2}).

In summary, we have the estimate
\begin{eqnarray}
u(n\e,\e)&=&(1-\ue^{-n})\bar u(0)+\ue^{-n}G(\e)+o(\e)+O(\delta),\\
U(n\e,\e)&=&U(n\e,\e)=(1-\ue^{\sqrt{1-2n}-1})\bar
U(0)+\ue^{\sqrt{1-2n}-1}G(\sqrt{1-2n}\e)+o(\e)+O(\delta).
\end{eqnarray}
The boundary data is $G=\cos\phi+2$. Then by the maximum principle
in Theorem \ref{Milne theorem 3}, we can achieve $1\leq
u(0,\phi)\leq3$ and $1\leq U(0,\phi)\leq3$. Since
\begin{eqnarray}
\bar u(0)&=&\frac{1}{2\pi}\int_{-\pi}^{\pi}u(0,\phi)\ud{\phi}
=\frac{1}{2\pi}\int_{\sin\phi>0}u(0,\phi)\ud{\phi}+\frac{1}{2\pi}\int_{\sin\phi<0}u(0,\phi)\ud{\phi}\\
&=&\frac{1}{2\pi}\int_{\sin\phi>0}(2+\cos\phi)\ud{\phi}+\frac{1}{2\pi}\int_{\sin\phi<0}u(0,\phi)\ud{\phi}\nonumber\\
&=&2+\frac{1}{2\pi}\int_{\sin\phi>0}\cos\phi\ud{\phi}+\frac{1}{2\pi}\int_{\sin\phi<0}u(0,\phi)\ud{\phi}\nonumber,
\end{eqnarray}
we naturally obtain $3/2\leq\bar u(0)\leq 5/2$. Similarly, we can
obtain $3/2\leq\bar U(0)\leq 5/2$. Furthermore, for $\e$
sufficiently small, we have
\begin{eqnarray}
G(\sqrt{1-2n}\e)&=&3+o(\e),\\
G(\e)&=&3+o(\e).
\end{eqnarray}
Hence, we can obtain
\begin{eqnarray}
u(n\e,\e)&=&\bar u(0)+\ue^{-n}(-\bar u(0)+3)+o(\e)+O(\delta),\\
U(n\e,\e)&=&\bar U(0)+\ue^{\sqrt{1-2n}-1}(-\bar
U(0)+3)+o(\e)+O(\delta).
\end{eqnarray}
Then we can see $\lim_{\e\rt0}\lnm{\bar U(0)-\bar u(0)}=0$ naturally
leads to $\lim_{\e\rt0}\lnm{(-\bar u(0)+3)-(-\bar U(0)+3)}=0$. Also,
we have $-\bar u(0)+3=O(1)$ and $-\bar U(0)+3=O(1)$. Due to the
smallness of $\e$ and $\delta$, and also $\ue^{-n}\neq
\ue^{\sqrt{1-2n}-1}$, we can obtain
\begin{eqnarray}
\abs{U(n\e,\e)-u(n\e,\e)}=O(1).
\end{eqnarray}
However, above result contradicts our assumption that
$\lim_{\e\rt0}\lnm{U(\eta,\phi)-u(\eta,\phi)}=0$ for any
$(\eta,\phi)$. This completes the proof of (\ref{main theorem 2}).

\bibliographystyle{siam}
\bibliography{Reference}

\end{document}